\documentclass[11pt,preqno,a4paper]{amsart}

\usepackage{enumitem}
\usepackage{amssymb,amsmath,mathtools}
\usepackage{hyperref}
\usepackage[nobysame]{amsrefs}
\usepackage{xcolor}

\usepackage{microtype}
\usepackage{bbm}
\usepackage{a4wide}
\newcommand*{\mailto}[1]{\href{mailto:#1}{\nolinkurl{#1}}}

 \allowdisplaybreaks

\newtheorem{theorem}{Theorem}[section]
\newtheorem{assumption}{Assumption}[section]
\newtheorem{lemma}[theorem]{Lemma}
\newtheorem{proposition}[theorem]{Proposition}
\newtheorem{corollary}[theorem]{Corollary}
\newtheorem{remark}[theorem]{Remark}
\newtheorem{definition}[theorem]{Definition}

\newenvironment{theoremprime}[1]{%
  \renewcommand{\thetheorem}{#1'}
  \addtocounter{theorem}{-1}
  \begin{theorem}}
{\end{theorem}}

\newcommand{\N}{{\mathbb N}}

\DeclareMathOperator{\supp}{supp}

\newcommand{\be}{\begin{equation}}
\newcommand{\ee}{\end{equation}}

\newcommand{\al}{\alpha}

\newcommand{\eq}{\begin{equation}}
\newcommand{\eeq}{\end{equation}}

\numberwithin{equation}{section}

\makeatletter
\newcommand{\dlmf}[1]{%
	\cite[%
	\def\nextitem{\def\nextitem{, }}%
	\@for \el\coloneqq#1\do{\nextitem\href{http://dlmf.nist.gov/\el}{(\el)}}%
	]{dlmf}%
}
\makeatother

\begin{document}

\allowdisplaybreaks

\title[Local decay estimates]{Local decay estimates for the bi-Laplacian Nonautonomous  Schr\"{o}dinger  equation}

\author{Jiayan Wu, Ting Zhang and Ruze Zhou}
\address{School of mathematics, South China University of Technology
\\
Guangzhou 510641, China
}
\address{School of mathematics, Zhejiang university
\\
Hangzhou 310058, China
}
\address{School of mathematics, Zhejiang university
\\
Hangzhou 310058, China
}
\email{wujiayan@scut.edu.cn;  zhangting79@zju.edu.cn; 12435042@zju.edu.cn}

\date{\today}


\begin{abstract}
In this paper, we establish local decay estimates for the bi-Laplacian  Schr\"{o}dinger equation with time-dependent (in particular, quasi-periodic) potentials in spatial dimension $n\ge14$.  Moreover, under stronger spectral regularity hypotheses, the same result can be extended to dimension $n\ge9$.
Our approach, based on asymptotic completeness and the existence of the channel wave operator, departs from standard resolvent-based methods.   In addition, global-in-time Strichartz estimates are derived from the local decay estimates.
\end{abstract}

\maketitle

\section{Introduction}
\subsection{Background}

In the study of linear and nonlinear dispersive equations,   local decay estimates provide a quantitative measure of the rate at which the $L^p$-mass or energy of a solution leaks out of any compact spatial region as time evolves.  Concretely, if $\psi(t) $ is the solution   for the Schr\"odinger equation with some decay potentials, a typical local decay estimate takes the form:
$$
\int_0^\infty \bigl\|\langle x\rangle^{-\eta}\,\psi(t)\bigr\|_{L^2_x}^2\,dt
\;\le C\,\|\psi_0\|_{L^2_x}^2,
\quad \eta>0,
$$
where $\langle x\rangle=(1+|x|^2)^{1/2}$.  Here $\langle x\rangle^{-\eta}$ localizes the measurement to a neighborhood of the origin (or any compact spatial region), and the bound shows that the weighted $L^2$-norm decays in time. Physically, this reflects the fact that dispersive waves cannot remain trapped in a bounded region under repulsive or oscillatory potentials  but must radiate their mass toward spatial infinity.  Mathematically,  local decay estimates play a central role in the scattering theory of dispersive equations. Specifically, they provide a foundation for establishing Strichartz estimates, which are essential tools for analyzing the behavior of solutions to dispersive equations. Furthermore, the local decay property established here directly yields the existence and completeness of the M\"{o}ller wave operators. These estimates have important applications in both linear and nonlinear time-dependent resonance theory \cites{SW1998,SW2005}.

Recently,  Liu and Soffer \cite{MR4721470} developed a general framework for proving asymptotic completeness in both linear and nonlinear dispersive equations. However, their approach requires radial symmetry for nonlinear setting. Subsequently, Soffer and Wu \cites{SW20224,SW20221,SW20223} refined this approach by constructing the free channel wave operator in an adapted way.  Building on the large time asymptotics of the solution \cite{SW20221}, Soffer and Wu \cite{sw20225} derived local decay estimates for the   Schr\"odinger equation with time-dependent quasi‐periodic potentials.

Motivated by \cite{SW20221}, Soffer and the authors \cite{SWWZ} further established asymptotic completeness  for the bi-Laplacian (bi-harmonic) Schr\"odinger operator $(-\Delta)^2+V(t, x)$ by constructing the corresponding channel wave operator.  The bi-Laplacian Schr\"odinger equation was introduced by Karpman \cite{Karpman1} and Karpman–Shagalov \cite{MR1396248} to incorporate small fourth-order dispersive effects in the propagation of intense laser beams in bulk media with Kerr nonlinearity.

It is nature to ask whether  local decay estimates hold for the bi-Laplacian Schr\"odinger equation with time-dependent  potentials. In this paper, building on the asymptotic completeness framework of \cite{SWWZ}, we establish local decay estimates  for the bi-Laplacian Schr\"odinger equation with a time‐dependent quasi‐periodic potential
$V(x,t)$. The time-independent case was treated by resolvent methods in \cite{FSY2018}; however, those techniques do not apply in the quasi-periodic setting. Overcoming this obstruction is the main novelty of our work.

There is also a wealth of results on the Schr\"{o}dinger operator with time-dependent potentials; see, for example, Rodnianski and Schlag \cite{RS2004} and the more recent works \cites{MR2831875,MR2484934,MR4907959}.


\subsection{Preliminary and main results}
Let $H_0=(-\Delta)^2$. In this paper, we consider the  bi-Laplacian Schr\"odinger equation with a time-dependent potential $V(x,t)$
\begin{equation}\label{fourth-order}
		\left\{
		\begin{array}{ll}
			i\partial_t\psi(x,t)=(H_0+V(x,t))\psi(t),  \quad(x,t)\in \mathbb{R}^n\times\mathbb{R}^+, n\geq 9,\\
			 \psi(x,0)=\psi_0\in L^2_x(\mathbb{R}^n),
		\end{array}
		\right.
\end{equation}
where $V(x,t)$  satisfies
\begin{assumption}\label{asp:1} $V(x,t)$ satisfies
   $\langle x\rangle^\sigma V(x,t)\in L^{\infty}_{x,t}(\mathbb{R}^{n+1})$ for some $\sigma >\max\{12+\frac{n}{2},20\}.$
\end{assumption}

\begin{assumption}\label{asp:2} $V(x,t)$ is quasi-periodic in $t$, more precisely,
\begin{align}
V(x,t)=V_0(x)+\sum_{j=1}^{N}V_j(x,t),\quad N\in \mathbb{N}^+,
\end{align}
where each $V_j(x,t)$ is periodic in $t$ with minimum periods $T_j>0$, where $T_j$ are irrational to each other, that is, $T_{j_1}/T_{j_2}$ are non-integers for all $  j_1\neq j_2$.
\end{assumption}


We define the evolution operators $U(t,s),t,s\in \mathbb{R}$ generated by $H=H_0+V(x,t)$ are unitary on $L_x^2$. The projection on the space of  all scattering states is given by
\begin{align}\label{DefPc}
P_c(t)\coloneqq s\text{-}\lim\limits_{v\to+\infty} U(t,t+v)e^{-iv H_0}F_{\leq1}\left(\frac{|x|}{v^{\alpha}} \right)e^{i
v H_0}U(t+v ,t),\quad \text{ on }L^2_x(\mathbb{R}^n)
\end{align}
for $\alpha\in (0, \frac{1}{2}-\frac{2}{n}), n\geq 5$. This is constructed in  \cite{SWWZ}   provided that $V(x,t)\in L^\infty_tL^2_x(\mathbb{R}^{n+1}), n\geq 5$.  It is worth noting that the right hand side of \eqref{DefPc}  is independent of $\alpha$ (See \eqref{WaveO-def} below and (1.34) of \cite{SWWZ}) and  $P_c(t) $ defined in \eqref{DefPc} satisfies $P_c(t)=\Omega_{+}(t)\Omega^*_{+}(t)$, where
\begin{align}
   \Omega_{+}(t)\coloneqq s\text{-}\lim\limits_{v\to\infty}U(t,t+v)e^{-ivH_0} \quad \mbox{on }\ L^2_x(\mathbb{R}^{n})
\end{align} and $\Omega^*_{+}(t)$ is the conjugate of $\Omega_{+}(t)$. 

First,  the main part of this paper is to establish the following local decay estimate for the scattering part of the solution to \eqref{fourth-order}.

\begin{theorem}\label{Thm1.1}
    If Assumptions \ref{asp:1} and \ref{asp:2}  hold, then, for all $\eta>\frac{5}{2}$,
\begin{align}\label{dis: eq}
    \int_0^{+\infty}\|\langle x\rangle^{-\eta}U(t,t_0)P_c(t_0)\psi_0\|^2_{L_x^2(\mathbb{R}^n)} dt\lesssim_{\eta}\|\psi_0\|^2_{L_x^2(\mathbb{R}^n)}
\end{align}
holds for $n\geq 14$.
\end{theorem}

\begin{remark}
  For  bi-Laplacian Schr\"odinger operator,   Mizutani and  Yao \cite{MY202410} proved that the free flow admits $L^2$ local decay estimate:
 \begin{align}\label{localdecay}
     \int_0^\infty \|\langle x\rangle^{-\delta}e^{-itH_0}\psi_0\|_{L_x^2}^2dt\lesssim\|\psi_0\|_{L_x^2}^2,\ \text{for all }\delta>2,\ n>4
 \end{align}
and $L^2$ local smoothing estimate:
\begin{align}\label{localsmooth}
    \int_0^\infty \|\langle x\rangle^{-\frac{1}{2}-\epsilon}|p|^{\frac{3}{2}}e^{-itH_0}\psi_0\|_{L_x^2}^2dt\lesssim\|\psi_0\|_{L_x^2}^2,\ \text{for any } \epsilon>0,\ n>4.
\end{align}
Thus the conclusion of Theorem \ref{Thm1.1} is not sharp. Analogously to the free flow estimate \eqref{localdecay}, we guess we can optimize the exponent in \eqref{dis: eq} from $\frac{5}{2}$ to 2. Whether this improvement is justified and its potential implications for our hypotheses remain formally unverified. We leave a rigorous verification of this improvement to future work.

For classical  Schr\"odinger operator, the $L^2$ local smoothing estimate was first established
by Kato and Yajima \cite{MR1061120} by using uniform resolvent estimates, based on the smooth perturbation method. Later, it was reproved by Ben–Artzi and Klainerman \cite{MR1226935} by using the spectral
measure integral.

\end{remark}

Next, we introduce the extra assumptions needed to extend Theorem \ref{Thm1.1} to more general cases.  These focus on the complementary projection $P_b(t)=I - P_c(t)$. The existence of $P_c(t)$ implies the existence of $P_b(t)$.
Concretely, we require that
\begin{assumption}\label{asp:3}
    For all $\epsilon \in (0,\frac{1}{4})$ and some $\delta >\frac{5}{2}$, $P_b(t)$ satisfies
    \begin{align}\label{con1}
        \sup_{t\in(0,\infty)} \left\| P_b(t)  \langle x \rangle^{\delta}\right\|_{2\to2}\lesssim_{\delta}1
    \end{align}
    and
    \begin{align}\label{con2}
        \sup_{t\in(0,\infty)} \left\| P_b(t) |p|^{-\frac{3}{2}}\langle x \rangle^{\frac{1}{2}+\epsilon}\right\|_{2\to2}\lesssim_{\epsilon}1.
    \end{align}
\end{assumption}
These conditions come up when estimating the non‐scattering part $P_b(t)\,e^{-itH_0}f$, $f\in L^2_x.$ Indeed, we can insert weights in two ways:$$
P_b(t)e^{-itH_0}f
= P_b(t)\,\langle x\rangle^{\delta}\,\langle x\rangle^{-\delta}e^{-itH_0}f,
$$
which ties into the standard local decay \eqref{localdecay}, or
$$
P_b(t)e^{-itH_0}f
= P_b(t)\,|p|^{- \frac{3}{2}}\langle x\rangle^{ \frac{1}{2}+\epsilon}
\,\langle x\rangle^{- \frac{1}{2}-\epsilon}|p|^{ \frac{3}{2}}e^{-itH_0}f,
$$
which uses the local smoothing estimate \eqref{localsmooth}.

In fact, under Assumption \ref{asp:1}, we show in Proposition \ref{Prop4.12} that for all
$\delta\in[0,\min\{\tfrac n2-4,4\}]$ and $n\ge9$,
$$
\sup_{t\in\mathbb R}\bigl\|P_b(t)\,\langle x\rangle^{\delta}\bigr\|_{2\to2}
\;\lesssim_{\delta,n}\;1.
$$
The proof reveals that any non-scattering state must be localized at low frequencies, especially near zero energy.

\begin{theorem}\label{Thm1.3}
    If Assumptions \ref{asp:1}, \ref{asp:2} and \ref{asp:3} hold, then \eqref{dis: eq} is valid for $n\geq 9$.

\end{theorem}
\begin{remark}
    The improvement in Theorem \ref{Thm1.3} (lowering the dimension threshold from $n \geq 14$ to $n \geq 9$) is a consequence of the additional structural information provided by Assumption \ref{asp:3}, which compensates for the weaker dispersive decay in lower dimension cases by leveraging localized bounds on the bound state projection.
\end{remark}

\subsection{Applications}
The Strichartz estimates for the $H_0=(-\Delta)^2$ and $H=H_0+V(x)$ have been prove in \cite{BENARTZI200087} and \cite{FSY2018} respectively. As an application, we obtain the Strichartz estimates for $H=H_0+V(x,t)$ with quasi-periodic potential by above theorems. The Strichartz estimates state that
\begin{align}\label{Strich1}
    \|U(t,0)P_c(0)f\|_{L^q_t L^r_x }\leq C \|f\|_{L^2_x},
\end{align}
for $2\leq q\leq \infty,$ $\frac{n}{r}+\frac{4}{q}=\frac{n}{2}$. In particularly, $r=\frac{2n}{n-4}$ when $q=2.$
\begin{theorem}\label{Thm1.5}
    If $V(x,t)$ satisfies Assumptions \ref{asp:1} and \ref{asp:2}, then Strichartz estimates \eqref{Strich1} hold for all $n\geq 14.$
\end{theorem}
\begin{corollary}\label{cor}
    If Assumptions \ref{asp:1}, \ref{asp:2} and \ref{asp:3} hold, then the Strichartz estimates \eqref{Strich1} hold for all $n\geq 9.$
\end{corollary}
\begin{remark}
  Theorem \ref{Thm1.5} and Corollary \ref{cor} establish the Strichartz estimates for the quasi-periodic potential. As for the Strichartz estimates with the time-independent potential constitutes a comparatively simpler problem, for which numerous results are already established in \cites{JSS1990,RS2004,MY202410}.
\end{remark}

\subsection{Notations}

Throughout this paper, let $L^p_{x,\sigma}$ denote the weighted $L^p_x(\mathbb{R}^n)$ space by the definition $$L^p_{x,\sigma}(\mathbb{R}^n)=\{f,\langle x\rangle^{\sigma}f\in L^p_x(\mathbb{R}^n) \}$$with the norm$$\|f\|_{L^p_{x,\sigma} (\mathbb{R}^n)}=\|\langle x\rangle^{\sigma}f\|_{L^p_x(\mathbb{R}^n)}.$$ And we define the space

$$\mathcal{H}_{F}=\{f(x,\mathbf{s}): f(x,\mathbf{s})\in L^2_xL^2_{\mathbf{s}}(\mathbb{R}^n\times \mathbb{T}_1\times \cdots\times \mathbb{T}_N )\}$$
with norm
$$\|f(x,\mathbf{s})\|^2_{\mathcal{H}_{F}}=\int_{\mathbb{T}_1\times \cdots\times \mathbb{T}_N}\int_{\mathbb{R}^n}|f(x,\mathbf{s})|^2 dxd\mathbf{s}.$$

We use $C$ to denote a generic constant, which may vary from line to line. We write $A \lesssim B$ or $A \gtrsim B$ to indicate that $A \leq C B$ or $C A \geq B$, respectively, for some constant $C > 0$ which is independent of relevant parameters. If the implicit constant $C_a > 0$ depends on a parameter $a$, we write as $A \lesssim_a B$ or $A \gtrsim_a B$. For notational simplicity, we shall write
$$
\quad\|T\|_{2\to2}
$$
in place of
$$
\quad\|T\|_{L^2_x(\mathbb{R}^n)\to L^2_x(\mathbb{R}^n)}
$$
for any bounded operator $T$ from $L^2$ to $L^2$.
 The Fourier transform and its inverse are defined as follows:
$$
\widehat{f}(\xi) \coloneqq \frac{1}{(2\pi)^{n/2}} \int e^{-ix \cdot \xi} f(x)\ dx
$$
and
$$
f(x) \coloneqq \frac{1}{(2\pi)^{n/2}} \int e^{ix \cdot \xi} \ \widehat{f}(\xi) \ d\xi
$$
for $f \in L^2_x(\mathbb{R}^n)$. Let $p=-i\nabla_x$,  then we define the differential operator $F(p)$ by
    $$F(p) f(x)=\frac{1}{(2\pi)^{n/2}} \int e^{ix \cdot \xi} \ F(\xi)\widehat{f}(\xi) \ d\xi.$$

For $m>0$, let $F_{\geq m}(k)=F(\frac{k}{m})$ denote a smooth  characteristic functions  satisfying
\begin{align}
F(\lambda)=\begin{cases}1,& \text{when }\lambda \geq 1,\\
0, & \text{when }\lambda <\frac{1}{2},\end{cases}
\end{align}
and   $F_{<m}(k)=1-F_{\geq m}(k)$ denote the complement of $F_{\geq m}(k).$ Similarly, we can also define the characteristics functions $F_{>m}(k)$ and $F_{\leq m}(k)$.

\subsection{Strategies}
We take $t_0 = 0$ and write $P_c \equiv P_c(0)$ for simplicity. The case when $t_0\neq 0$ can be treated similarly. Inspired by  \cite{sw20225}, we use the compactness argument, the notion of incoming/outgoing waves (see Definition \ref{def2.1}) and the identity $P_c(t)=\Omega_{+}(t)\Omega^*_{+}(t)$(AC) to prove our theorems. For the reader's convenience, we briefly outline the main argument below.
 We give the M\"oller wave operators as introduced in \cite{SWWZ}: for $\alpha\in(0,\frac12-\frac 2n)$, $n\geq 5$,
\begin{align}\label{WaveO-def}
    \Omega_{ \pm}^*(t)\coloneqq s\text{-}\lim_{v\rightarrow\pm\infty}F_{\leq 1}\left(\frac{|x|}{v^\alpha} \right)e^{ivH_0}U(t+v,t),\quad \mbox{on }\ L^2_x(\mathbb{R}^{n}).
\end{align}
Here, we use notation $\Omega_{ \pm}^*(t)$ since $\Omega_{ \pm}^*(t)$ are independent of $\alpha$; see (1.34) of \cite{SWWZ}. The M\"oller wave operators $\Omega_{ \pm}^*(t)$ satisfy the intertwining property (see Lemma \ref{LemmaA.1} for proof):
   \begin{align}
    \Omega^*_{\pm}(t)U(t,0)=e^{-itH_0}\Omega_{\pm}^*(0),
\end{align} in the weak topology of $L^2_x(\mathbb{R}^n)$.

For the simplicity, let's consider $\psi_0=P_c\psi_0$.

\textbf{Step 1.} By   the  incoming/outgoing decomposition, we split   $\psi(t)=U(t,0)P_c\psi_0$  as:
\begin{align}\label{decomforsolu}
 \psi(t)&=P^+\psi(t)+P^-\psi(t)\nonumber\\
&\coloneqq \psi_{f,1}(t)+\psi_{f,2}(t)+C(t)\psi(t)
,\end{align}
where
$$\psi_{f,1}(t)=P^+e^{-itH_0}\Omega^*_{+}(0)\psi_0, \ \psi_{f,2}(t)=P^-e^{-itH_0}\Omega_{ -}^*(0)\psi_0$$
and
$$C(t)=P^+(1-\Omega^*_{+}(t))+P^-(1-\Omega^*_{-}(t)).$$

We claim that $C(t)$ is compact on  $L^2_x$, see Lemma \ref{comp1}. By estimates $\|P^{\pm}\|_{2\to2}\leq1$ and $\|\Omega_{\pm}^*(t)\|_{2\to2}\leq1$, $C(t)$ satisfies \begin{align}
    \sup_{t\in(0,\infty)}\|C(t)\|_{2\to2}\leq4.
\end{align}

Next, we use the equality $P_c(t)=\Omega_{+}(t)\Omega^*_{+}(t)$ to decompose $C(t) $ further. We define the space-time cutoff as
    \begin{equation}
      F_{M}(x,p)\coloneqq F_{\leq M}(|x|)F_{\geq \frac{1}{M}}(|p|).\label{FM}
    \end{equation}
Combined with $P_b(t)\coloneqq1-P_c(t),$ the operator $C(t)$ admits the composition
$$C(t)=C_M(t)+C_r(t)+C(t)P_b(t),$$
where components are given by
\begin{align}
    C_M(t)\coloneqq C(t)\Omega_{+}(t)F_{M}(x,p) \Omega^*_{+}(t)
\end{align}
and
\begin{align}
    C_r(t)\coloneqq C(t)\Omega_{+}(t)(1-F_{M}(x,p)) \Omega_{+}^*(t).
\end{align}
Since $P_b(t)U(t,0)P_c=P_b(t)P_c(t)U(t,0)=0,$ \eqref{decomforsolu} is equivalent to
\begin{align}\label{formu-for-U}
     \psi(t)=\psi_{f,1}(t)+\psi_{f,2}(t)+C_M(t)\psi(t)+C_r(t)\psi(t).
\end{align}

\textbf{Step 2.} We  prove  that there exists a constant $M_0>1$ such that whenever $M\geq M_0$,
\begin{equation}\label{condon-r}
    \begin{aligned}
    \sup\limits_{t\in (0,\infty)}\|C_r(t)\|_{2\to2}<\frac{1}{2},
\end{aligned}
\end{equation}
see Proposition \ref{smallforC_r} for the detailed proof, which relies on the key result in Proposition \ref{pbcon} that $$s\text{-}\lim\limits_{\mathbf{u}\to \mathbf{s}}\tilde{P_b}(\mathbf{u})=\tilde{P_b}(\mathbf{s}),$$where $\tilde{P_b}(\mathbf{u})$ is defined in \eqref{3.11}.

\textbf{Step 3.} Thanks to \textbf{Step 2,} $(1-C_r(t))^{-1}$ exists. 
Then moving $C_r(t)\psi(t)$ to the left-hand side of \eqref{formu-for-U} and then applying $(1-C_r(t))^{-1}$ to both sides consequently yields
\begin{align}\label{decomp}
    \psi(t)=(1-C_r(t))^{-1}\Big(\psi_{f,1}(t)+\psi_{f,2}(t)\Big)+(1-C_r(t))^{-1}C_M(t)\psi(t).
\end{align}
The absorption properties of $P_b(t)$ allows us to establish that
\begin{equation}\label{eq:111}
    \begin{aligned}
        \int_0^{\infty} \|\langle x\rangle^{-\eta}(1-C_r(t))^{-1}C_M(t)\psi(t)\|^2_{L_x^2 }dt
        \lesssim&_M\|\psi_0\|_{L_x^2}^2
    \end{aligned}
\end{equation}
and for all $\eta>5/2$,
\begin{align}\label{eq:222}
   \int_0^\infty \|\langle x\rangle^{-\eta}(1-C_r(t))^{-1}(\psi_{f,1}(t)+\psi_{f,2}(t))\|^2_{L_x^2}\lesssim\|\psi_0\|_{L_x^2}^2,\quad n\geq14,
\end{align}
see  Proposition \ref{Prop4.16}. Combining \eqref{eq:111} and \eqref{eq:222}, with \eqref{decomp}, gives Theorem \ref{Thm1.1}.

\subsection{Difficulties in the problems and recent advances in dispersion Theory}


For free bi-Laplacian Schr\"odinger equation, Ben-Artzi, Koch and Saut \cite{MR1226935} prove  the $L^1\to L^\infty$ estimate for $e^{-it\Delta^2}$.  For time-independent  bi-Laplacian Schr\"odinger operator $H=H_0+V(x)$,  Feng, Soffer and Yao \cite{FSY2018} derived the asymptotic expansion of the resolvent $R_V(z)=(H-z)^{-1}$ near the zero threshold under the assumption that zero is a regular point of $H$ for $n\geq 5$ and $n=3$. They further proved that in this regular case, the Kato–Jensen time decay estimate of $e^{-itH}$ is $(1+|t|)^{-n/4}$ for $n\geq 5$, while the $L^1\to L^\infty$ decay rate is $O(|t|^{-1/2})$ for $n=3$ (not optimal).
Erdo\u{g}an–Green–Toprak \cite{EGT2021} (for $n=3$) and Green–Toprak \cite{GT2019} (for $n=4$) derived the asymptotic expansion of $R_V(z)$ near zero in the presence of resonance or eigenvalue, established $L^1-L^\infty$ estimates for each type of zero resonance and proved that the time decay rate is $|t|^{-n/4}$ for $n=3,4$ when zero is regular, while zero-energy resonances alter this decay rate.  Subsequently, Li, Soffer and Yao \cite{LSY2023} confirmed the decay estimates in the 2-dimensional case, while Soffer, Wu and Yao \cite{SWY2022} addressed them in the 1-dimensional case.
These dispersive estimates, in turn, naturally yield local decay bounds through standard weighted energy arguments. However, although the methods employed in these works are rather general, they remain limited when applied to potentials
$V(x,t)$ that are spatially localized and time-periodic. To the best of our knowledge, there are no available results on dispersive estimates, local decay estimates, or Strichartz estimates for the bi-Laplacian Schr\"odinger equation with time-dependent potentials.

Recently, Soffer, Wu, Wu and Zhang \cite{SWWZ} studied the existence of wave operators and the asymptotic completeness of solutions for the bi-Laplacian Schr\"odinger equation with time-dependent potentials. Motivated by \cite{sw20225}, our goal is to extend the local decay  theory of the Schr\"odinger equation to the bi-Laplacian Schr\"odinger case, which is far from trivial. When the potential $V(x,t)$ is time-periodic, we employ Floquet theory to transform the problem into one with a time-independent potential (see Appendix \ref{appendixB}). In this work, we adopt the approach introduced in \cite{sw20225} to establish local decay estimates for solutions of the bi-Laplacian Schr\"odinger equation with quasi-periodic time-dependent potentials.

Time-dependent Hamiltonians $H(t)=H_0+V(x,t)$ lie beyond the reach of classical spectral theory: knowing the instantaneous spectrum of $H(t)$ at each fixed $t$ does not by itself determine the long-time behavior of the corresponding Schr\"odinger evolution.
The proof of local decay estimates is not straightforward and requires overcoming three technical difficulties inherent to the bi-Laplacian operator. These challenges also constitute the principal innovations of this paper. In what follows, we explain each in turn.

First, for the bi-Laplacian Schr\"odinger operator, the validity of \begin{align*}
    \|P^{\pm}F_{\geq 1}(|p|)F_{\leq 2}(|p|)e^{\pm it(-\Delta)^2} \langle x\rangle^{-\delta}\|_{L_x^2}\lesssim_{n,\delta}\frac{1}{\langle t\rangle^{\delta}}
\end{align*} had remained an open problem. In this paper, we provide the first rigorous proof of this result via a refined phase-space analysis and microlocal decomposition. Specifically, we decompose the propagator into a near-field and a far-field based on spatial scale. The desired estimates are ultimately estimated by the non-stationary phase method and the explicit evolution property; see Appendix \ref{Lemma2.4} for details.
Second, the classical route to obtain $$s\text{-}\lim\limits_{\mathbf{u}\to \mathbf{s}}\tilde{P_b}(\mathbf{u})=\tilde{P_b}(\mathbf{s})\quad\mbox{and}\quad \dim\tilde{P_b}(\mathbf{s})<\infty $$ for Schr\"odinger operator relies on $\text{Ran}(\Omega_+(t))=\text{Ran}(\Omega_-(t))$, see \cite{J}. For the bi-Laplacian Schr\"odinger operator, $\text{Ran}(\Omega_+(t))=\text{Ran}(\Omega_-(t))$ may not hold. Our work introduces a new method that circumvents this obstacle; see Proposition \ref{pbcon}.
Finally, a further key technical innovation of our work is to establish a higher-order absorption property of the operator $P_b(t)$ associated with the bi-Laplacian Schr\"odinger operator, which—unlike in the Schr\"odinger case—permits absorption of terms up to order $\langle x\rangle^\delta$, for all $\delta \in[0,\min\{\frac{n}{2}-4,4\}]$ when $n\geq 9$.

For the Schr\"{o}dinger operator, asymptotic completeness enables the decomposition of the state space into bound states and scattering states. Bound states correspond to eigenvalues and eigenfunctions, with negative eigenvalues posing minimal difficulty: \cite{SB1981} established that for potentials decaying sufficiently rapidly at infinity, there exist finitely many negative eigenvalues with smooth and rapidly decaying associated eigenfunctions. Crucially, the existence of a zero eigenvalue may give rise to singular behavior. In \cite{NP2025}, Hayashi and Naumkin gave the existence of the singular solutions for the nonlinear bi-Laplacian Schr\"{o}dinger equation. So the existence of the zero eigenvalues even the positive eigenvalues must be taken into account. In \cites{EGT2021,GT2019}, Erdo\u{g}an, Green and Toprak considered the zero eigenvalue and zero energy resonance condition in 3D and 4D. While Beceanu \cite{BM2016} obtained the dispersive estimate for the Schr\"{o}dinger equation under the assumption that the zero energy exists.

For the bound state corresponding to the Schr\"{o}dinger operator $-\Delta+V+m^2$ with additional mass term, there are numerous results. As demonstrated in \cites{SW1990,SW1998,SW1999} by Soffer and Weinstein, such bound states become unstable under nonlinear perturbations, forming metastable states. Regarding the soliton solutions of the nonlinear Schr\"{o}dinger equation in dimensions $d\geq 3$, extensive work addresses asymptotic completeness and long-time dynamics under varying eigenvalue constraints, see
\cites{sol1,sol2,sol3,sol4,sol5,sol6,sol7,sol8,sol9,sol10,sol11}.

Beyond local decay estimates for bi-Laplacian Schr\"odinger operators and results on bound states/solitons of Schr\"{o}dinger-type operators, significant progress has recently been achieved in the analysis of bi-Laplacian Schr\"odinger equations. Accelerated dispersion decay was established in weighted spaces by \cite{GG2022}, while time decay estimates for the bi-Laplacian Schr\"odinger operator in $\mathbb{R}^3$ were obtained in \cite{LWWY2025}. Concurrently, Galtbayar and Yajima \cite{GY2024} demonstrated $L^p$-boundedness of wave operators for the bi-Schr\"{o}dinger operator on $\mathbb{R}^4$.

\section{Phase Space Decompositions}

\subsection{Incoming/Outgoing Waves}
In this subsection, we establish some preliminary tools for decomposing the evolution operator. The proof of Theorem \ref{Thm1.1}  requires the notions of forward/backward propagation waves and new free channel wave operators, which can be found in \cite{sw20226} and \cite{SW20221}, respectively.
First, we present the definition of the incoming/outgoing wave, which inspired by Mourre \cite{Mr19792}.

Let $S^{n-1}$ denote the unit sphere in $\mathbb{R}^n$. We define a class of functions on the $S^{n-1}$, $\{F^{\hat{h}}(\xi)\}_{\hat{h}\in I}$, as a smooth partition of unity with an index set$$I=\{\hat{h}_1,\cdots,\hat{h}_N\}\subseteq S^{n-1}$$for some $N\in \mathbb{N}^+$, satisfying that there exists $c>0$ such that for every $\hat{h}_i\in I$ and $\xi\in S^{n-1}$,\begin{align}\label{cutF1}
    F^{\hat{h}_i}(\xi)=\begin{cases}
        1\quad \text{when}|\xi-\hat{h}_i|<c, \\
        0\quad \text{when}|\xi-\hat{h}_i|>2c.
    \end{cases}
\end{align}
Given $\hat{h}\in I$, for $\xi\in S^{n-1}$, we define $\tilde{F}:S^{n-1}\to \mathbb{R}$, as another smooth cut-off function satisfying \begin{align}\label{cutF2}
    \tilde{F}^{\hat{h}}(\xi)=\begin{cases}
        1\quad \text{when}|\xi-\hat{h}|<4c, \\
        0\quad \text{when}|\xi-\hat{h}|>8c.
    \end{cases}
\end{align}
For $h\in \mathbb{R}^n-\{0\}$, we define $\hat{h}=h/|h|$ and $\hat{h}=0$ when $h=0$. We also assume that $c>0$, define in \eqref{cutF1} and \eqref{cutF2}, is properly chosen such that for all $x,q\in \mathbb{R}^n-\{0\}$, \begin{align}\label{cutest1}
    F^{\hat{h}}(\hat{x})\tilde{F}^{\hat{h}}(\hat{q})|x+q|\geq F^{\hat{h}}(\hat{x})\tilde{F}^{\hat{h}}(\hat{q})\frac{1}{10}(|x|+|q|),
\end{align}and \begin{align}\label{cutest2}
    F^{\hat{h}}(\hat{x})(1-\tilde{F}^{\hat{h}}(\hat{q}))|x-q|\geq F^{\hat{h}}(\hat{x})(1-\tilde{F}^{\hat{h}}(\hat{q}))\frac{1}{10^6}(|x|+|q|).
\end{align}

Now let we define the incoming/outgoing wave:
\begin{definition}[Cutoff on the forward/backward propagation set]\label{def2.1}
    The smooth cutoffs onto the forward and backward propagation sets, in term of $(r,v)\in \mathbb{R}^{n+n}$, are defined as follows:\begin{align}
        P^+(r,v)=\sum_{b=1}^NF^{\hat{h}_b}(\hat{r})\tilde{F}^{\hat{h}_b}(\hat{v}),
    \end{align}and \begin{align}
        P^-(r,v)=1-P^+(r,v),
    \end{align}respectively.
\end{definition}We define the incoming/outgoing waves $P^\pm$ as \begin{align}
    P^\pm\equiv P^\pm(x,2p).
\end{align}

\begin{lemma}\label{BoundP-Lem}
    For all $-\frac{n}{2}<a<\frac{n}{2},$ $\epsilon\in (0,\frac{1}{4})$ and for all $f\in H^a(\mathbb{R}^n)$, the following inequalities hold:\begin{align}\label{boundp}
        \||p|^aP^\pm f\|_{L_x^2}\lesssim\||p|^af\|_{L_x^2}
    \end{align}and \begin{align}\label{boundx}
        \||x|^aP^\pm f\|_{L_x^2}\lesssim\||x|^af\|_{L_x^2}.
    \end{align}
\end{lemma}
\begin{proof}
    We only need to prove that $$\||x|^aP^+f\|_{L_x^2}\lesssim\||x|^af\|_{L^2}.$$Since$$\||x|^aP^-f\|_{L_x^2}\leq \||x|^aP^+f\|_{L_x^2}+\||x|^af\|_{L_x^2}\lesssim\||x|^af\|_{L_x^2},$$the estimate \eqref{boundx} follows once the above inequality is established.

    We use the knowledge related to singular integrals to estimate $\||x|^aP^+f\|_{L^2}$. We can obtain\begin{equation}
        \begin{aligned}
        &\||x|^aP^+f\|_{L_x^2}\leq\sum_{b=1}^N\||x|^aF^{\hat{h}_b}(\hat{x})\tilde{F}^{\hat{h}_b}(\hat{p})f\|_{L_x^2} =\sum_{b=1}^N\|F^{\hat{h}_b}(\hat{x})|x|^a\tilde{F}^{\hat{h}_b}(\hat{p})f\|_{L_x^2} \\
        &\leq \sum_{b=1}^N\||x|^a\tilde{F}^{\hat{h}_b}(\hat{p})f\|_{L_x^2}
    \end{aligned}
    \end{equation}and \begin{equation}
        \begin{aligned}
            \||x|^a\tilde{F}^{\hat{h}_b}(\hat{p})f\|_{L_x^2}=\left(\int |x|^{2a}|\tilde{F}^{\hat{h}_b}(\hat{p})f|^2dx\right)^{\frac{1}{2}}=\|\tilde{F}^{\hat{h}_b}(\hat{p})f\|_{L_x^2(|x|^{2a})}.
        \end{aligned}
    \end{equation}

    Due to \cite[Definition 5.11 and Theorem 7.11]{JDF}, if $T_b=\tilde{F}^{\hat{h}_b}(\hat{p})$ is a Calder\'on-Zygmund operator and $w=|x|^{2a}\in A_2$ holds, we have finished the proof that $\|\tilde{F}^{\hat{h}_b}(\hat{p})f\|_{L_x^2(|x|^{2a})}\lesssim\|f\|_{L_x^2(|x|^{2a})}=\||x|^af\|_{L_x^2}$.

    Following from the definition of the differential operator, we have \begin{equation}
        \begin{aligned}
            \tilde{F}^{\hat{h}_b}(\hat{p})f=\mathcal{F}^{-1}(\tilde{F}^{\hat{h}_b}(\hat{\xi})\hat{f}(\xi))=\int K(x-y)f(y)dy
        \end{aligned}
    \end{equation}where $K(s)=\mathcal{F}^{-1}(\tilde{F}^{\hat{h}_b})(s)$ is the integral kernel. It is easy to verify that $K$ is a standard kernel and has cancellation $\int_{|x|=1}K(x)d\sigma=0$ and for $-\frac{n}{2}<a<\frac{n}{2}$, $|x|^{2a}\in A_2$ satisfies. Thus, $T_b$ is a Calder\'on-Zygmund operator. As for \eqref{boundp}, we only need to repeat the argument similarly.
\end{proof}


\begin{corollary}\label{boundP2}
 For all $\epsilon\in (0,\frac{1}{4})$, $n\geq 4$ and for all $\langle x\rangle^{-1/2-\epsilon}|p|^{\frac32}f\in L^2(\mathbb{R}^n),$ we have
    \begin{align}\label{boundpx}
        \|\langle x\rangle^{-1/2-\epsilon}|p|^{\frac32}P^{\pm}f\|_{L_x^2}\lesssim\|\langle x\rangle^{-1/2-\epsilon}|p|^{\frac32}f\|_{L_x^2}.
    \end{align}
\end{corollary}

\begin{proof}
    Similarly as before, we only need to prove \begin{align}
        \|\langle x\rangle^{-1/2-\epsilon}|p|^{\frac32}F^{\hat{h}_b}(\hat{x})\tilde{F}^{\hat{h}_b}(\hat{p})f\|_{L_x^2}\lesssim\|\langle x\rangle^{-1/2-\epsilon}|p|^{\frac32}f\|_{L_x^2}.
    \end{align}
    By the commutator smoothing effect, we have \begin{equation}
        \begin{aligned}
            &\|\langle x\rangle^{-1/2-\epsilon}|p|^{\frac32}F^{\hat{h}_b}(\hat{x})\tilde{F}^{\hat{h}_b}(\hat{p})f\|_{L_x^2}\\
            \leq &\|\langle x\rangle^{-1/2-\epsilon}F^{\hat{h}_b}(\hat{x})|p|^{\frac32}\tilde{F}^{\hat{h}_b}(\hat{p})f\|_{L_x^2}+\|\langle x\rangle^{-1/2-\epsilon}[|p|^{\frac32},F^{\hat{h}_b}(\hat{x})]\tilde{F}^{\hat{h}_b}(\hat{p})f\|_{L_x^2} \\
            \lesssim&\|\langle x\rangle^{-1/2-\epsilon}|p|^{\frac32}f\|_{L_x^2}+\|\langle x\rangle^{-1/2-\epsilon}|x|^{-1}|p|^{\frac12}\tilde{F}^{\hat{h}_b}(\hat{p})f\|_{L_x^2} \\
            \lesssim&\|\langle x\rangle^{-1/2-\epsilon}|p|^{\frac32}f\|_{L_x^2},
        \end{aligned}
    \end{equation}where the last line employs the weighted Hardy inequality: for $\al>1-\frac{n}{2}$, \begin{align}
        \||x|^{\al-1}f\|_{L^2_x}\lesssim\||x|^\al|\nabla f|\|_{L^2_x}.
    \end{align}
\end{proof}

\begin{lemma}\label{mourre}
    The incoming/outgoing projections $P^{\pm}$ and the free flows $e^{\pm it(-\Delta)^2}$ satisfy the estimate
\begin{align}\label{fund_Estimate}
    \|P^{\pm}F_{\geq 1}(|p|)F_{\leq 2}(|p|)e^{\pm it(-\Delta)^2} \langle x\rangle^{-\delta}\|_{2\to 2}\lesssim_{n,\delta}\frac{1}{\langle t\rangle^{\delta}}
\end{align}
for $\delta\geq 0$, $t\geq0$ and $n\in \N^+$.
\end{lemma}
\begin{proof}
See Appendix \ref{Lemma2.4} for the proof.
\end{proof}

 By applying \eqref{fund_Estimate}, we get the following lemma, which will be essential to prove Theorem \ref{Thm1.3}.
\begin{lemma}\label{allkindsestimates}
    The  operators $P^{\pm} e^{\pm it(-\Delta)^2}$, $t>0$, satisfy the following estimates:
\begin{itemize}
    \item [1.]\textbf{High Energy Estimate} \\For $n\in \N^+$, $\delta>0,$  $c>0,$ $t>0$,
    \begin{align}\label{HEE}
        \|P^{\pm}F_{\geq c}(|p|)e^{\pm it(-\Delta)^2} \langle x\rangle^{-\delta}\|_{2\rightarrow2}\lesssim_{c,n,l,\delta}\frac{1}{\langle t\rangle^{\delta}}.\end{align}
    \item[2.] \textbf{Pointwise Smoothing Estimate}\\ For $n\in \N^+$, $\delta>0,$ $l\in [0,3\delta),$ $M\geq 1$ and $t\geq 0$ with $tM^4\geq 1,$
    \begin{align}\label{PSE}
         \|P^{\pm}F_{\geq M}(|p|)e^{\pm it(-\Delta)^2)} |p|^l\langle x\rangle^{-\delta}\|_{2\rightarrow2}\lesssim_{n,3\delta -l }\frac{1}{ M^{3\delta -l}t^{\delta}}.
    \end{align}
    \item[3.] \textbf{Time Smoothing Estimate}\\ For $n\in \mathbb{N}^+$, $\delta >3$, $l=1,2,3,$ we have
    \begin{align}\label{TSE}
        \int_0^{1}t^l\|P^{\pm}F_{\geq 1}(|p|)e^{\pm it(-\Delta)^2}|p|^{3l}\langle x\rangle^{-\delta}\|_{2\rightarrow2}dt\lesssim_{n,l,\delta}1.
    \end{align}
    \item [4.]\textbf{Near Threshold Estimate}\\
    For $n\in \N^+$, $\delta>0$ and   $\epsilon\in (0,1/4)$, $\alpha\in[0,\frac{n}{2}),$
    \begin{align}\label{NTE}
         \|\frac{|p|^{\alpha}}{\langle p\rangle^{\alpha}}P^{\pm} e^{\pm it(-\Delta)^2} \langle x\rangle^{-\delta}\|_{2\rightarrow2}\lesssim_{ n,\alpha,\epsilon,\delta}  \frac{1}{\langle t\rangle^{\min\{\frac{\delta }{4}+\epsilon, \ (1/4-\epsilon)(\alpha+\min\{\frac{n}{2},\delta\})\} } }.
    \end{align}
    \item[5.] \textbf{Pointwise Local Smoothing Estimate} \\ For $\epsilon\in (0,1/4)$, $\delta\in [0,4]$ and $n\geq 9,$ $l=1,\cdots,12$ and $t\geq 1$,
    \begin{align}\label{PLSE}
         \|\frac{1}{\langle x\rangle^{4-\delta}}P^{\pm} e^{\pm it(-\Delta)^2} p_j^{l}\langle x\rangle^{-(\frac{n}{2}+4+l-\delta))}\|_{2\rightarrow2}\lesssim_{ n, \epsilon }  \frac{1}{\langle t\rangle^{ \frac{n}{8}+\frac{4+l-\delta}{4}-\epsilon }},
    \end{align}where $p_j$ denotes the j-th component of $p$.
\end{itemize}
\end{lemma}
\begin{proof}
    See Appendix \ref{Lemma2.5} for the proof.
\end{proof}

\begin{lemma}\label{FMest}For all $\delta>0,\  M\geq 1$, we have
    \begin{align}
    \|F_M(x,p)e^{itH_0} \langle x\rangle^{-\delta}\|_{2\to2}\lesssim_{M}\frac{1}{\langle t\rangle^\delta}.
\end{align}
\end{lemma}
\begin{proof}We first split the operator into two parts:
 \begin{equation}\label{SPM}
    \begin{aligned}
        \|F_M(x,p)e^{itH_0} \langle x\rangle^{-\delta}\|_{2\to2}
        \leq \|F_{>t/C}(|x|)\langle x\rangle^{-\delta}\|_{2\to 2}+\|F_M(x,p)e^{itH_0} F_{<t/C}(|x|)\langle x\rangle^{-\delta}\|_{2\to 2}.
    \end{aligned}
\end{equation}The first term is supported in the region $|x|>t/C$, which immediately yields a bound of order $\langle  t\rangle^{-\delta}$. We now estimate the second term. Let $T=F_M(x,p)e^{itH_0}F_{<t/C}(|x|)$. Then we claim that for $C=\frac{M^3}{2}>\frac{M^3}{4},N>0$, $$\|T\|_{2\to 2}\lesssim_{M,N}\langle t\rangle^{-N}.$$For simplify, we estimate $\|TT^*\|_{2\to 2}=(\|T\|^2_{2\to 2})$ instead of estimate $\|T\|_{2\to 2}$ directly. Since $F_M(x,p)$ and $F_{<t/C}(|x|)$
are self-adjoint and  $e^{itH_0}$
 is unitary, we have
\begin{equation}
    \begin{aligned}
        TT^*=&F_M(x,p)e^{itH_0}F_{<t/C}(|x|)F_{<t/C}(|x|)e^{-itH_0}F_M(x,p).\\
    \end{aligned}
\end{equation}
Using the functional calculus property $e^{itH_0}f(|x|)e^{-itH_0}=f(|x(t)|)$ with $x(t)=e^{itH_0} xe^{-itH_0}=x+ 4t|p|^2p$, we obtain
\begin{equation}
    \begin{aligned}
TT^*=&F_M(x,p)F_{<t/C}(|x+ 4t|p|^2p|)F_M(x,p).
    \end{aligned}
\end{equation}
From \eqref{FM}, on the support of $F_M(x,p)$, we have \begin{equation}
    \begin{aligned}
     \frac{4t}{M^3}-M\leq4t|p|^3-|x|\leq |x+ 4t|p|^2p|\leq\frac{t}{C}=\frac{2t}{M^3}.
    \end{aligned}
\end{equation}Thus, as $t\geq \frac{M^4}{2}$, the supports of $F_M(x,p)$ and $F_{<t/C}(|x+ 4t|p|^2p|)$ are disjoint, implying $TT^*=0$ and therefore $\|T\|_{2\to 2}=0$.

For $t< \frac{M^4}{2}$, since $T$ is continuous,  for any fixed $M$ and $N$, there exists a constant $C_{M,N}$ such that $\|T\|_{2\to 2}\leq C_{M,N}\langle t\rangle^{-N}$.

Inserting the above bounds into \eqref{SPM}, we have\begin{equation}
    \begin{aligned}
        \|F_M(x,p)e^{itH_0} \langle x\rangle^{-\delta}\|_{2\to2}
        \lesssim \langle t\rangle^{-\delta} + \langle t\rangle^{-N}.
    \end{aligned}
\end{equation}
Choosing $N\geq \delta$ gives the desired estimate.
\end{proof}


Next, we establish the Hardy-Littlewood-Sobolev inequality, which will be utilized frequently throughout this paper.

\begin{lemma}\label{HLSvar}
    For all $0<\alpha<\frac{n}{2}$, we have \begin{equation}\label{HLS}
        \begin{aligned}
        &\||x|^{-\alpha}|p|^{-\alpha}f\|_{L^2_x}\lesssim_{\alpha,n}\|f\|_{L^2_x}, \\
        &\||p|^{-\alpha}|x|^{-\alpha}f\|_{L^2_x}\lesssim_{\alpha,n}\|f\|_{L^2_x}. \\
    \end{aligned}
    \end{equation}
\end{lemma}

\begin{proof}
We only prove the first inequality; the second one can be handled analogously. We begin by recalling Pitt’s inequality \cite[Theorem 2]{Bec1995}
\begin{equation}\label{pitt}
        \int_{ \mathbb{R}^n}|\xi|^{-a}|\hat{f}(\xi)|^2d\xi\lesssim_{a,n}\int_{ \mathbb{R}^n}|x|^{a}|f(x)|^2dx,
    \end{equation}where $0\leq a<n$.
Since $|\xi|^{-a}$ being the radial function, Pitt's inequality also holds for the inverse Fourier  transform, namely
\begin{equation}\label{pitt1}
        \int_{ \mathbb{R}^n}|x|^{-a}|\check{f}(x)|^2dx\lesssim_{a,n}\int_{ \mathbb{R}^n}|\xi|^{a}|f(\xi)|^2d\xi.
    \end{equation}
    Therefore, \begin{equation}
        \begin{aligned}
            \||x|^{-\alpha}|p|^{-\alpha}f\|_{L^2_x}^2&=\int_{ \mathbb{R}^n}|x|^{-2\alpha}(|p|^{-\alpha}f(x))^2dx \\
            &=\int_{ \mathbb{R}^n}|x|^{-2\alpha}\mathcal{F}^{-1}(|\xi|^{-\alpha}\hat{f}(\xi))^2dx \\
            &\lesssim_{\al,n}\int_{ \mathbb{R}^n}|\xi|^{2\al}(|\xi|^{-\alpha}\hat{f}(\xi))^2d\xi \\
            &=\|\hat{f}(\xi)\|_{L^2_\xi}^2=\|f\|_{L^2_x}^2,
        \end{aligned}
    \end{equation}
where the parameter satisfies $0\leq 2\al<n$, i.e. $0<\al<\frac{n}{2}$. Hence, the estimate \eqref{HLS}$_1$ is proved. The proof of \eqref{HLS}$_2$ follows by applying the same argument in the dual space.
\end{proof}


\subsection{Quasi-Periodic Evolution Operators}
We give some properties related to the quasi-periodicity of the operators $\Omega_{ \pm}^*(t)$, $\Omega_{ \pm}(t)$,   $P_c(t)$, $P_b(t)$ and $C(t)$.
\begin{definition}
    We say a function $f(t)$ is quasi-periodic in $t$ if
    \begin{align}
        f(t)=\tilde{f}(s_1,\cdots,s_N),\quad \mbox{for some} \ \tilde{f}, 
    \end{align}
where $t\equiv s_j\,(mod\, T_j)$ for $s_j\in [0,T_j)$.
\end{definition}

\begin{lemma}\label{Lemma2.9}
If Assumptions \ref{asp:1} and \ref{asp:2} hold, then $\Omega_{ \pm}^*(t)$, $\Omega_{ \pm}(t)$,   $P_c(t)$, $P_b(t)$ and $C(t)$  are bounded operators on $L^2_x(\mathbb{R}^n)$ and  quasi-periodic in $t$  with the same type as $V$.
\end{lemma}

\begin{proof}Recall that
   \begin{align}
    \Omega_{ \pm}^*(t)\coloneqq s\text{-}\lim_{v\rightarrow\pm\infty} F_{\leq1}(\frac{|x|}{v^\alpha} )e^{ivH_0}U(t+v,t),\quad \mbox{on }\ L^2_x(\mathbb{R}^{n}),
\end{align}
\begin{align}
    P_c(t)\coloneqq s\text{-}\lim_{v\rightarrow\infty}U(t,t+v)e^{-ivH_0}F_{\leq1}(\frac{|x|}{v^\alpha} )e^{ivH_0}U(t+v,t),
\end{align}
\begin{align}
    P_b(t)=1-P_c(t)
\end{align}
and
\begin{align}
   C(t)=P^+(1-\Omega^*_{+}(t))+P^-(1-\Omega^*_{-}(t)).
\end{align}
The boundedness of $ P_c(t)$ has been established in   \cite[Theorem 1.5]{SWWZ}. Since $ P_b(t)=1-P_c(t)$, its boundedness follows directly from that of $P_c(t)$. The boundedness of $C(t)$ on $L^2$
is then a consequence of the existence of	$\Omega_{ \pm}^*(t)$  and the estimate $\|P^\pm\|_{2\to2}\leq 1.$
Let $\Omega(t,s)=U(t,t+s)e^{-isH_0}$ and $\Omega^*(t,s)=e^{isH_0}U(t+s,t)$.
By Duhamel’s formula, we have
\begin{align}\label{OmeD}
U(t,t+s)e^{-isH_0}
= 1 + i\int_0^s U(t,t+u) V(x,t+u) e^{-iuH_0}du.
\end{align}
Iterating this formula infinitely many times yields
\begin{align}
    \Omega(t,s)=1+\sum_{j=1}^{\infty} I^{(j)}(t,s),
\end{align}
where
\begin{align}
     I^{(j)}(t,s)=(i)^j \int_0^sdu_1 \int_0^{u_1}du_2 \cdots \int_0^{u_{j-1}}du_j \mathcal{K}_t(V(t+u_j))\times\cdots\times\mathcal{K}_t(V(t+u_1))
\end{align}
with
\begin{align}
    \mathcal{K}_t(V(t+u ))=e^{iuH_0}V(t+u)e^{-iuH_0}.
\end{align}
We can decompose  $\mathcal{K}_t(V(t+u ))$ as
\begin{align}
    \mathcal{K}_t(V(t+u ))=e^{iuH_0}V_0(x)e^{-iuH_0}+\sum_{j=1}^{N}e^{iuH_0}V_j(x,t+u)e^{-iuH_0}.
\end{align}
Hence, it is quasi-periodic in $t$. Consequently, the operators $\Omega(t,s)$, $\Omega^*(t,s)$ and $P_c(t)$ are all quasi-periodic in $t$.
\end{proof}



\section{Preprocessing of the operator \texorpdfstring{$C(t)$}{(C(t))} }

\subsection{Compactness of \texorpdfstring{$C(t)$}{(C(t))} and its decomposition}
Recall that
$$C(t)=C^+(t)+C^-(t)=P^+(1-\Omega^*_{+}(t))+P^-(1-\Omega^*_{-}(t)).$$

By Duhamel expansion
\begin{align}
    \Omega^*_{+}(t)=1+(-i)\int_0^{\infty}dv e^{ivH_0}V(x,t+v)U(t+v,t),
\end{align}
we have
\begin{align}\label{A11}
    C^\pm(t)=i\int_0^{\infty}P^\pm e^{\pm isH_0}V(x,t\pm s)U(t\pm s,t)ds.
\end{align}
We first prove the compactness of $C(t)$ in this subsection.
\begin{lemma}\label{comp1}
    If $\langle x\rangle^5 V(x,t)\in L^\infty_{t,x}(\mathbb{R}^{n+1})$, then $C(t)$ is compact on $L^2_x(\mathbb{R}^n)$ for all $t\in \mathbb{R}.$
\end{lemma}
\begin{proof}
    It suffices to show the compactness of $C^\pm (t)). $  By $\langle x\rangle^5 V(x,t)\in L^\infty_{t,x}(\mathbb{R}^{n+1})$   and Assumption \ref{asp:2}, similarly to \cite[Lemma 4.1]{sw20225}, we obtain the compactness of
\begin{equation}\label{finite}
    \begin{aligned}
        i\int_0^{M}P^\pm e^{\pm isH_0}V(x,t\pm s)U(t\pm s,t)ds.
    \end{aligned}
\end{equation}
Thus, together with estimate
\begin{equation}
    \begin{aligned}
       & \|\int_M^{\infty}P^\pm e^{\pm isH_0}V(x,t\pm s)U(t\pm s,t)ds\|_{2\to2}\\
        \lesssim&\int_M^{\infty}\|P^\pm e^{\pm isH_0}\langle x\rangle^{-5}\|_{2\to2}\|\langle x\rangle^{5} V(x,t\pm s)\|_{L^\infty_{t,x}}ds\\
        \lesssim& \int_M^{\infty}\frac{1}{\langle s\rangle^{\frac{9}{8}-\epsilon}}\|\langle x\rangle^{5} V(x,t\pm s)\|_{L^\infty_{t,x}}ds\to0,
    \end{aligned}
\end{equation}
as $M\to\infty$ where we use \eqref{NTE} as $\alpha=0$ for all $n\geq 9$, yields the compactness of $P^\pm(1-\Omega^*_{\pm}(t))$.
\end{proof}

\begin{corollary}\label{comp2}
    If $\langle x\rangle^5V(x,t)\in L^\infty_{t,x}$, then $P^-(1-e^{-itH_0}U(0,t))$ is compact on $L^2_x( \mathbb{R}^n)$ for all $t\in \mathbb{R}$.
\end{corollary}
\begin{proof}
    By the Duhamel's formula, we have \begin{equation}
        \begin{aligned}
        P^-(1-e^{-itH_0}U(0,t))&=i\int_0^tP^-e^{-isH_0}V(-s)U(0,s)ds.
    \end{aligned}
    \end{equation}
    For a fixed finite t, the operator is obviously compact by \eqref{finite}. As for the sufficient large t, we can also prove the compactness of the operator similarly as Lemma \ref{comp1}.
\end{proof}

We now prove the decomposition property of $C(t)$. We note that $C(t)$ can be expressed as the sum of three operators
\begin{align}
    C(t)=C_M(t)+C_r(t)+C(t)P_b(t)
\end{align}deriving from the equality $P_c(t)=\Omega_+(t)\Omega_+^*(t)$ in \cite{SWWZ}, where the operators $C_M (t)$ and $C_r(t) $ are given by \begin{align}
    C_M(t)\coloneqq C(t)\Omega_{+}(t)F_{M}(x,p) \Omega^*_{+}(t)
\end{align}
and
\begin{align}
    C_r(t)\coloneqq C(t)\Omega_{+}(t)(1-F_{M}(x,p)) \Omega^*_{+}(t)
\end{align}
for $M>0$ and $F_M$ given in (\ref{FM}).
For the operator $C_M (t)$ acting on the function $U(t,0)f$, by Lemma \ref{LemmaA.1} (intertwining property), it also can be rewritten as
\begin{align}\label{C_M3}
    C_M(t)U(t,0)f=C(t)\Omega_{+}(t)F_{M}(x,p)e^{-itH_0} \Omega_{+}^*f.
\end{align}
However, for the operator $C_r (t)$, we give an analysis on its smallness in the next subsection.
\subsection{  A Smallness Analysis on \texorpdfstring{$C_r(t)$}{(C r(t))} }
\begin{proposition}\label{smallforC_r}
If Assumptions \eqref{asp:1} and \eqref{asp:2} are satisfied, then there exists $ M_0>0$ such that whenever $M\geq M_0$,
 \begin{align}\label{C_r0}   \sup_{t\in\mathbb{R}}\|C_r(t)\|_{2\to2}<\frac{1}{2}.
    \end{align}
\end{proposition}
The proof of Proposition \ref{smallforC_r} relies on Lemma \ref{Lemma2.9}. By Lemma \ref{Lemma2.9},  $P_b(t)$, $U(t+T,t)$, $\Omega_{ \pm}(t)$ and $\Omega_{ \pm}^*(t)$   are  quasi-periodic in $t$  with the same type as $V$. So we could use finitely many parameters $s_j\in  \mathbb{T}_j$, $j=1,\ldots,N$ to express these operators:
for $\mathbf{s}=(s_1,\cdots,s_N)\in\mathbb{T}_1\times\cdots\times\mathbb{T}_N $, we define
\begin{equation}\label{3.11}
    \begin{aligned}
        &\tilde{P}_b(\mathbf{s})\coloneqq P_b(t),\quad\tilde{U}_{\mathbf{s}}(T,0)\coloneqq U(t+T,t),\quad\tilde{\Omega}_{+}(\mathbf{s})\coloneqq\Omega_{+}(t),\\
        &\tilde{\Omega}^*_{+}(\mathbf{s})\coloneqq\Omega^*_{+}(t),
    \end{aligned}
\end{equation}
where $t\equiv s_j$ (mod $T_j$), $j=1,\ldots,N$.

\begin{proposition}\label{pbcon}
    For all $\mathbf{s}\in \mathbb{T}_1\times\cdots\times\mathbb{T}_N$,  we have
    \begin{align*}
        s\text{-}\lim_{\mathbf{u}\to \mathbf{s}} \tilde{P}_b(\mathbf{u})=\tilde{P}_b(\mathbf{s}),
        \ \textrm{ on }L^2_x(\mathbb{R}^n).
    \end{align*}
\end{proposition}

We first claim the following corollary which is the directly result of Proposition \ref{pbcon} and will be essential for proving Proposition \ref{smallforC_r}.

\begin{corollary}\label{Omeu}For all $\mathbf{s}\in \mathbb{T}_1\times\cdots\times\mathbb{T}_N$, we have
    \begin{align}
          \lim_{\mathbf{u}\to \mathbf{s}}\|\tilde{\Omega}^\ast_+(\mathbf{u})\tilde{P_c}(\mathbf{s})f\|_{L_x^2}=\|\tilde{P_c}(\mathbf{s})f\|_{L_x^2}.
    \end{align}
\end{corollary}
\begin{proof}Using the definition of $\tilde{P_c}(\mathbf{s})$, we get
    \begin{equation}
    \begin{aligned}
        \|\tilde{\Omega}^\ast_+(\mathbf{s})\tilde{P_c}(\mathbf{s})f\|_{L_x^2}^2=&\langle\tilde{\Omega}^\ast_+(\mathbf{s})\tilde{P_c}(\mathbf{s})f,\tilde{\Omega}^\ast_+(\mathbf{s})\tilde{P_c}(\mathbf{s})f\rangle\\
        =&\langle\tilde{\Omega}_+(\mathbf{s})\tilde{\Omega}^\ast_+(\mathbf{s})\tilde{P_c}(\mathbf{s})f,\tilde{P_c}(\mathbf{s})f\rangle \\
        =&\|\tilde{P_c}(\mathbf{s})f\|_{L_x^2}^2.
    \end{aligned}
    \end{equation}Thus,\begin{equation}
        \begin{aligned}
         \lim_{\mathbf{u}\to \mathbf{s}}\|\tilde{\Omega}^\ast_+(\mathbf{u})\tilde{P_c}(\mathbf{s})f\|_{L_x^2}^2=& \lim_{\mathbf{u}\to \mathbf{s}}\|\tilde{\Omega}^\ast_+(\mathbf{u})\tilde{P_c}(\mathbf{u})\tilde{P_c}(\mathbf{s})f\|_{L_x^2}^2\\
        =& \lim_{\mathbf{u}\to \mathbf{s}}\|\tilde{P_c}(\mathbf{u})\tilde{P_c}(\mathbf{s})f\|_{L_x^2}^2\\
        =&\|\tilde{P_c}(\mathbf{s})f\|_{L_x^2}^2,
    \end{aligned}
    \end{equation}
    where we use the equality $\tilde{\Omega}^\ast_+(\mathbf{u})\tilde{P_c}(\mathbf{u})=\tilde{\Omega}^\ast_+(\mathbf{u})$ and Proposition \ref{pbcon}, see also the claim \eqref{p1}.
\end{proof}

 The proof of Proposition   \ref{pbcon} requires the following two lemmas.
 \begin{lemma}\label{pbfin}
   For all $\mathbf{s},\mathbf{u}\in \mathbb{T}_1\times\cdots\times\mathbb{T}_N$, we have
   \begin{align}
       \dim \tilde{P}_b(\mathbf{u})=\dim \tilde{P}_b(\mathbf{s})<\infty.
   \end{align}
 \end{lemma}
\begin{proof}
Firstly, by the results of \cite{MB1980} and \cite{Dispersion2012}, the discrete spectrum of the general Schr\"{o}dinger operator is finite dimension.

We observe that
$\dim P_b(t)=\dim \tilde{P}_b(\mathbf{s})$, $t\equiv s_j$ (mod $T_j$), $j=1,\ldots,N$ and note that
   \begin{equation}
       \begin{aligned}
        P_b(t) &=1-\Omega_+(t)\Omega_+^*(t) \\
        &=U(t,0)(1-\Omega_+(0)\Omega_+^*(0))U(0,t)\\
        &=U(t,0) P_b(0) U(0,t)
    \end{aligned}
   \end{equation}
for all $t \in \mathbb{R}.$
Hence, it suffices to prove that
$$
\dim P_b(0) < \infty.
$$
Assume to the contrary that
  $$
\dim P_b(0) = \infty.
$$

Let $f\in \mathrm{Ran}(P_b(0))$ satisfying $\|f\|_{L^2}=1$. Consider $U(t,0) f$ and apply the incoming/outgoing decomposition:
\begin{equation}\label{App1}
    \begin{aligned}
        U(t,0)f
        &= P^+ U(t,0)f + P^- U(t,0)f  \\
        &= P^+ \Omega_+^*(t) U(t,0)f + P^+(1 - \Omega_+^*(t)) U(t,0)f  \\
        &\quad + P^- e^{-itH_0} f + P^-(U(t,0) - e^{-itH_0})f  \\
        &= P^+(1 - \Omega_+^*(t)) U(t,0)f + P^-(U(t,0) - e^{-itH_0}) f + P^- e^{-itH_0} f  \\
        &\coloneqq \tilde{C}(t) U(t,0) f + P^- e^{-itH_0} f,
    \end{aligned}
\end{equation}
    where
\begin{align}
    \tilde{C}(t)\coloneqq P^+(1-\Omega_+^*(t))+P^-(1-e^{-itH_0}U(0,t)).
\end{align}
In the third equality above, we have used that $\Omega_+^*(t)U(t,0)f= e^{-itH_0}\Omega_+^*(0)P_b(0)f$, with the identity $\Omega_+^*(0)P_b(0)=\Omega_+^*(0)(1-\Omega_+(0)\Omega_+^*(0))=0$.

Next, decompose $\tilde{C}(t) U(t,0) f$ by means of smooth cutoff functions:
\begin{align}
    \tilde{C}(t)U(t,0)f=&F_{\leq M}(|x|)F_{\leq M}(|p|)\tilde{C}(t)U(t,0)f\nonumber\\
    &+(1-F_{\leq M}(|x|)F_{\leq M}(|p|))\tilde{C}(t)U(t,0)f.
\end{align}
For the second term, using Lemma \ref{comp1} and Corollary \ref{comp2}, we know that the operator $\tilde{C}(t)$ is compact. Combined with the fact  $$\lim\limits_{M\rightarrow\infty}\|\left(1-F_{\leq M}(|x|)F_{\leq M}(|p|)\right)f\|_{L^2}=0,  \ \forall f\in L^2,$$we have $$s\text{-}\lim\limits_{M\rightarrow\infty}1-F_{\leq M}(|x|)F_{\leq M}(|p|)=0, \ \textrm{ on }\mathrm{Ran}(\tilde{C}(t)). $$
Hence, the uniform bound
\begin{align}\label{smallforsupp}
    \sup_{t}\|(1-F_{\leq M}(|x|)F_{\leq M}(|p|))\tilde{C}(t)U(t,0)f\|_{L^2_x(\mathbb{R}^n)}\leq\frac{1}{100}
\end{align}holds for all $M\geq M_0$, where $M_0$ is  fixed.
Obviously, $F_{\leq M}(|x|)F_{\leq M}(|p|)$ is a time-independent compact operator which can be approximated by a finite rank operator. Thus, for some $\tilde{N}$,
\begin{align}\label{app2}
     \tilde{C}(t)U(t,0)f=&\sum_{l=1}^{\tilde{N}}[(\phi_l(x),\cdot)_{L^2_x}\psi_{l}(x)]\tilde{C}(t)U(t,0)f+\tilde{C}_{r1}(t)U(t,0)f\nonumber\\
     &+(1-F_{\leq M}(|x|)F_{\leq M}(|p|))\tilde{C}(t)U(t,0)f
\end{align}
with \begin{align*}
    \sup_{t}\|C_{r1}(t)\|_{2\to2 }<\frac{1}{100}.
\end{align*}
Substituting \eqref{app2} into \eqref{App1} yields
\begin{align}
U(t,0)f=\sum_{l=1}^{\tilde{N}}[(\phi_l(x),\cdot)_{L^2_x}\psi_{l}(x)]\tilde{C}(t)U(t,0)f+\tilde{C}_r(t)U(t,0)f+P^-e^{-itH_0}f,
\end{align}
where the combined remainder is
    \begin{align}
    \tilde{C}_r(t)\coloneqq\tilde{C}_{r1}(t)+(1-F_{\leq M}(|x|)F_{\leq M}(|p|))\tilde{C}(t).
\end{align}
 For sufficiently large $M$, we obtain
\begin{align*}
    \sup_{t}\|\tilde{C}_{r}(t)\|_{2\to2}<\frac{1}{10}.
\end{align*}
Hence,
\begin{equation}
    \begin{aligned}
    U(t,0)f&=\sum_{l=1}^{\tilde{N}}C_{l}(t,f)(1-\tilde{C}_r(t))^{-1}\psi_l(x)+(1-\tilde{C}_r(t))^{-1}P^-e^{-itH_0}f,
\end{aligned}
\end{equation}
where $C_{l}(t,f)= (\phi_l(x),\tilde{C}(t)U(t,0)f)_{L^2_x}$.

Let $m=2\tilde{N}$. Then there exists an orthonormal set $\{f_1, \ldots, f_m\} \subset \mathrm{Ran}(P_b(0))$ with $\|f_j\|_{L^2_x} = 1$ and $f_j \perp f_l$ for $j \neq l$.
Following Shubin \cite{SMA2001}, the eigenfunctions of the bi-Laplacian Schr\"odinger operator exhibit polynomial decay of arbitrary order. Hence, for each  $j=1,\ldots, m$,
\begin{align}
    \|\langle x\rangle^{ \delta}f_j\|_{L^2_x(\mathbb{R}^n)}\lesssim 1,\quad \forall \delta>0.
\end{align}

By Lemma~\ref{allkindsestimates}, and in particular from \eqref{NTE}, we have
$\|P^-e^{-itH_0}f_j\|_{L^2_x}$ approaches zero, as t goes to infinity.

Choose   $T=T(m,f_1,\cdots,f_m)>0$ large enough, such that, for all  $t\geq T, $
\begin{align*}
    \sup_{1\leq j\leq m}\|P^-e^{-itH_0}f_j\|_{L^2_x(\mathbb{R}^n)}<\frac{1}{\tilde{N}+100},
\end{align*}
then
\begin{align*}
    \sup_{1\leq j\leq m}\|(1-\tilde{C}_r(t))^{-1}P^-e^{-itH_0}f_j\|_{L^2_x(\mathbb{R}^n)}<\frac{2}{\tilde{N}+100}, \ \mbox{for} \ t\geq T.
\end{align*}
Fix $t=2T>T$ and $\dim(span\{(1-\tilde{C}_r(t))^{-1}\psi_1(x),\cdots,(1-\tilde{C}_r(t))^{-1}\psi_{\tilde{N}}(x)\})\leq \tilde{N}.$
Then there exists a function $f\in span\{f_1,\cdots , f_m\}$ with $\|f\|_{L^2_x(\mathbb{R}^n)}=1$, such that $U(t,0)f\perp (1-\tilde{C}_r(t))^{-1}\psi_l(x)$ for all $l=1,\cdots,\tilde{N}.$
Let
\begin{align*}
    f(x)=\sum_{l=1}^{m}c_lf_l(x).
\end{align*}
Then we have
\begin{align*}
    (&U(t,0)f,U(t,0)f)_{L^2_x}=\left(U(t,0)f,\sum_{l=1}^{m}c_lU(t,0)f_l(x)\right)_{L^2_x}\nonumber\\
    &=\left(U(t,0)f,\sum_{l=1}^{m}c_l\sum_{k=1}^{\tilde{N}}C_{k}(t, f)(1-\tilde{C}_r(t))^{-1}\psi_k(x)+\sum_{l=1}^{m}c_l(1-\tilde{C}_r(t))^{-1}
    P^-e^{-itH_0}f_l\right)_{L^2_x}\nonumber\\
    &=\left(U(t,0)f,\sum_{l=1}^{m}c_l(1-\tilde{C}_r(t))^{-1}P^-e^{-itH_0}f_l\right)_{L^2_x}\nonumber\\
    &\leq \|f\|_{L^2_x}\times\sum_{l=1}^{m}|c_l| \times\frac{2}{\tilde{N}+100}\nonumber\\
    &\leq 1\times (\sum_{l=1}^{m}|c_l|^2)^{1/2}N^{1/2}\times \frac{2}{\tilde{N}+100}\nonumber\\
    &\leq\frac{2\sqrt{2}}{\sqrt{\tilde{N}+100}}<1,
\end{align*}
which contradicts the unitarity of $U(t,0)$, i.e. $\|U(t,0)f\|_{L^2_x} = 1$. Therefore, we conclude that $\dim P_b(0)<\infty.$\end{proof}

\begin{lemma}\label{Lemma3.7}
    If Assumptions \eqref{asp:1} and \eqref{asp:2} hold, then
    \begin{equation}
        \begin{aligned}
            \lim_{\mathbf{u}\to \mathbf{s}} \|(\tilde{\Omega}_{+}(\mathbf{s})-\tilde{\Omega}_{+}(\mathbf{u}))F_M(x,p)\|_{2\to2}=\lim_{u\to \mathbf{s}} \|F_M(x,p)(\tilde{\Omega}^*_{+}(\mathbf{s})-\tilde{\Omega}^*_{+}(\mathbf{u}))\|_{2\to2}=0.
        \end{aligned}
    \end{equation}
\end{lemma}
This lemma can be proved similarly as the argument in \cite[Lemma4.3]{sw20225} and we omit the details. Then, we prove Proposition \ref{pbcon} as follows.

\begin{proof}[\textbf{Proof of Proposition \ref{pbcon}}]
    We first claim that for all $\mathbf{s},\mathbf{u}\in \mathbb{T}_1\times\cdots\times\mathbb{T}_N$,
    \begin{align}\label{p1}
         s\text{-}\lim_{\mathbf{u}\to \mathbf{s}} \tilde{P}_c(\mathbf{u})\tilde{P}_c(\mathbf{s})=\tilde{P}_c(\mathbf{s})
    \end{align}
on $L^2_x$, namely, $\tilde{P}_c(\mathbf{s})$ is continuous in $\mathbb{T}_1\times\cdots\times\mathbb{T}_N$. Claim  \eqref{p1} implies
\begin{align*}
     s\text{-}\lim_{\mathbf{u}\to \mathbf{s}} \tilde{P}_b(\mathbf{u})\tilde{P}_c(\mathbf{s})= s\text{-}\lim_{\mathbf{u}\to \mathbf{s}} (1-\tilde{P}_c(\mathbf{u}))\tilde{P}_c(\mathbf{s})=0
\end{align*}
on $L^2_x$, which in turn implies
\begin{align*}
     s\text{-}\lim_{\mathbf{u}\to \mathbf{s}} (\tilde{P}_b(\mathbf{u})-\tilde{P}_b(\mathbf{u})\tilde{P}_b(\mathbf{s}))=0
\end{align*}
on $L^2_x$. By Lemma \ref{pbfin}, we obtain
\begin{align}\label{3.36}
     s\text{-}\lim_{\mathbf{u}\to \mathbf{s}} (\tilde{P}_b(\mathbf{u})-\tilde{P}_b(\mathbf{s}))=0
\end{align}
on $L^2_x$.

 Now we prove Claim  \eqref{p1}. For any $f\in L^2_x$, we get
\begin{align*}
    (\tilde{P}_c(\mathbf{u})-1)\tilde{P}_c(\mathbf{s})f=\mathcal{A}_1(\mathbf{u},\mathbf{s})f+\mathcal{A}_2(\mathbf{u},\mathbf{s})f,
\end{align*}
where operators $\mathcal{A}_1(\mathbf{u},\mathbf{s})$ and $\mathcal{A}_2(\mathbf{u},\mathbf{s})$ are given by
\begin{align*}
    \mathcal{A}_1(\mathbf{u},\mathbf{s})f\coloneqq(\tilde{P}_c(\mathbf{u})-1) \tilde{\Omega}_{+}(\mathbf{s})F_M(x,p)\tilde{\Omega}_{+}^*(\mathbf{s})f
\end{align*}
and
\begin{align*}
\mathcal{A}_2(\mathbf{u},\mathbf{s})f\coloneqq(\tilde{P}_c(\mathbf{u})-1) \tilde{\Omega}_{+}(\mathbf{s})(1-F_M(x,p))\tilde{\Omega}_{+}^*(\mathbf{s})f.
\end{align*}
First,   considering $\mathcal{A}_1(\mathbf{u},\mathbf{s})f$,  using the equation
\begin{align*}
(\tilde{P}_c(\mathbf{u})-1)\tilde{\Omega}_{+}(\mathbf{u})=0,
\end{align*}
we rewrite $\mathcal{A}_1(\mathbf{u},\mathbf{s})f$ as
\begin{align*}
    \mathcal{A}_1(\mathbf{u},\mathbf{s})f=(\tilde{P}_c(\mathbf{u})-1)(\tilde{\Omega}_{+}(\mathbf{s})-\tilde{\Omega}_{+}(\mathbf{u}))F_M(x,p)\tilde{\Omega}_{+}^*(\mathbf{s})f.
\end{align*}
By  Lemma \ref{Lemma3.7}, we obtain
\begin{align}\label{p4}
    \lim_{\mathbf{u}\to \mathbf{s}}\|\mathcal{A}_1(\mathbf{u},\mathbf{s})f\|_{L_x^2}=0.
\end{align}
Now we consider $\mathcal{A}_2(\mathbf{u},\mathbf{s})f$. For any $\epsilon>0, $ by the definition of the cutoff function, there exists $M=M_1(\epsilon)>0$ such that
\begin{align}\label{p2}
     \|(1-F_M(x,p))\tilde{\Omega}_{+}^*(\mathbf{s})f\|_{L_x^2}<\frac{\epsilon}{4}.
\end{align}
Combining \eqref{p2}, using  estimates $\|\tilde{P}_b(\mathbf{u})\|\leq1$ and $\|\tilde{\Omega}_{+}(\mathbf{u})\|\leq 1$, we have
\begin{align*}
    \|\mathcal{A}_2(\mathbf{u},\mathbf{s})f\|_{L_x^2}<\frac{\epsilon}{2}.
\end{align*}
This, together with \eqref{p4}, yields \eqref{p1}.
\end{proof}

\begin{proof}[\textbf{\textit{Proof of Proposition \ref{smallforC_r}}}]
It suffices to prove that for each $f\in L^2_x(\mathbb{R}^n),$
\begin{align*}
       \lim_{M\to \infty} \sup_{t\in \mathbb{R}}\|(1-F_M(x,p))\Omega^*_{+}(t)f\|_{L^2_x}=0.
    \end{align*}
By Lemma \ref{Lemma2.9}, this is equivalent to
\begin{align}\label{3.48}
     \lim_{M\to \infty} \sup_{\mathbf{s}\in \mathbb{T}_1\times\cdots\times\mathbb{T}_N}\|(1-F_M(x,p))\tilde{\Omega}^*_{+}(\mathbf{s})f\|_{L^2_x}=0.
\end{align}
We note that for each $\mathbf{s}\in \mathbb{T}_1\times\cdots\times\mathbb{T}_N$,
\begin{align}\label{limit1}
    \lim_{M\to \infty} \|(1-F_M(x,p))\tilde{\Omega}^*_{+}(\mathbf{s})f\|_{L^2_x}=0.
\end{align}
By Lemma \ref{Lemma3.7} and $\tilde{\Omega}^*_{+}(\mathbf{s})\tilde{P}_c(\mathbf{s})=\tilde{\Omega}^*_{+}(\mathbf{s})$, we obtain
\begin{align*}
     \lim_{M\to \infty}\lim_{\mathbf{u}\to \mathbf{s}}\|F_M(x,p)\tilde{\Omega}^*_{+}(\mathbf{u})\tilde{P}_c(\mathbf{s})f\|_{L^2_x}=\lim_{M\to \infty}\|F_M(x,p)\tilde{\Omega}^*_{+}(\mathbf{s})f\|_{L^2_x}=\|\tilde{\Omega}^*_{+}(\mathbf{s})f\|_{L^2_x}.
\end{align*}
By direct computation, we have
\begin{equation*}
    \begin{aligned}
     &\|(1-F_M(x,p))\tilde{\Omega}^*_{+}(\mathbf{u})\tilde{P}_c(\mathbf{s})f\|^2_{L^2_x}\\
     =&\|\tilde{\Omega}^*_{+}(\mathbf{u})\tilde{P}_c(\mathbf{s})f\|_{L^2_x}^2+\|F_M(x,p)\tilde{\Omega}^*_{+}(\mathbf{u})\tilde{P}_c(\mathbf{s})f\|_{L^2_x}^2\\
     &-2\langle\tilde{\Omega}^*_{+}(\mathbf{u})\tilde{P}_c(\mathbf{s})f,F_M(x,p)\tilde{\Omega}^*_{+}(\mathbf{u})\tilde{P}_c(\mathbf{s})f\rangle\\
     \coloneqq&I_1+I_2,
\end{aligned}
\end{equation*}where $$I_1=\|\tilde{\Omega}^*_{+}(\mathbf{u})\tilde{P}_c(\mathbf{s})f\|_{L^2_x}^2+\|F_M(x,p)\tilde{\Omega}^*_{+}(\mathbf{u})\tilde{P}_c(\mathbf{s})f\|_{L^2_x}^2$$ and $$I_2=-2\langle\tilde{\Omega}^*_{+}(\mathbf{u})\tilde{P}_c(\mathbf{s})f,F_M(x,p)\tilde{\Omega}^*_{+}(\mathbf{u})\tilde{P}_c(\mathbf{s})f\rangle.$$Combining with Corollary \ref{Omeu}, we get$$\lim_{M\to \infty}\lim_{\mathbf{u}\to \mathbf{s}}I_1=2\|\tilde{P}_c(\mathbf{s})f\|_{L^2_x}^2.$$Using  Lemma \ref{Lemma3.7}, we have
\begin{equation*}
\begin{aligned}
        \lim_{M\to \infty}\lim_{\mathbf{u}\to \mathbf{s}}I_2=&-2\lim_{M\to \infty}\lim_{\mathbf{u}\to \mathbf{s}}\langle\tilde{P}_c(\mathbf{s})f,\tilde{\Omega}_{+}(\mathbf{u})F_M(x,p)\tilde{\Omega}^*_{+}(\mathbf{s})\tilde{P}_c(\mathbf{s})f\rangle\\
    =&-2\lim_{M\to \infty}\lim_{\mathbf{u}\to \mathbf{s}}\langle\tilde{P}_c(\mathbf{s})f,\tilde{\Omega}_{+}(\mathbf{s})F_M(x,p)\tilde{\Omega}^*_{+}(\mathbf{s})\tilde{P}_c(\mathbf{s})f\rangle\\
    =&-2\lim_{M\to \infty}\lim_{\mathbf{u}\to \mathbf{s}}\langle\tilde{\Omega}^*_{+}(\mathbf{s})\tilde{P}_c(\mathbf{s})f,F_M(x,p)\tilde{\Omega}^*_{+}(\mathbf{s})\tilde{P}_c(\mathbf{s})f\rangle\\
    =&-2\|\tilde{\Omega}^*_{+}(\mathbf{s})\tilde{P}_c(\mathbf{s})f\|_{L^2_x}^2=-2\|\tilde{P}_c(\mathbf{s})f\|_{L^2_x}^2.
\end{aligned}
\end{equation*}
Thus we have proved that $$\lim_{M\to \infty}\lim_{\mathbf{u}\to \mathbf{s}}\|(1-F_M(x,p))\tilde{\Omega}^*_{+}(\mathbf{u})\tilde{P}_c(\mathbf{s})f\|^2_{L^2_x}=0.$$
This, together with \eqref{limit1}, yields
\begin{align*}
     \lim_{M\to \infty}\lim_{\mathbf{u}\to \mathbf{s}}\|(1-F_M(x,p))(\tilde{\Omega}^*_{+}(\mathbf{u})-\tilde{\Omega}^*_{+}(\mathbf{s}))\tilde{P}_c(\mathbf{s})f\|_{L^2_x}=0,
\end{align*}
where we use $\tilde{\Omega}^*_{+}(\mathbf{s})\tilde{P}_c(\mathbf{s})=\tilde{\Omega}^*_{+}(\mathbf{s})$. This, combined with (\ref{p1}) and Proposition \ref{pbcon}, yields
\begin{equation}
    \begin{aligned}\label{3.531}
    &\lim_{M\to \infty}\lim_{\mathbf{u}\to \mathbf{s}}\|(1-F_M(x,p))\tilde{\Omega}^*_{+}(\mathbf{u})f-(1-F_M(x,p))\tilde{\Omega}^*_{+}(\mathbf{s})f\|_{L^2_x} \\
    \leq&\lim_{M\to \infty}\lim_{\mathbf{u}\to \mathbf{s}}\|(1-F_M(x,p))(\tilde{\Omega}^*_{+}(\mathbf{u})-\tilde{\Omega}^*_{+}(\mathbf{s}))\tilde{P}_c(\mathbf{s})f\|_{L^2_x} \\
    &+\lim_{M\to \infty}\lim_{\mathbf{u}\to \mathbf{s}}\|(1-F_M(x,p))\tilde{\Omega}^*_{+}(\mathbf{u})(\tilde{P}_c(\mathbf{u})-\tilde{P}_c(\mathbf{s}))f\|_{L^2_x}=0.
\end{aligned}
\end{equation}
For any $\mathbf{s}\in \mathbb{T}_1\times\cdots\times\mathbb{T}_N$ and $\varepsilon>0$, from (\ref{limit1}) and (\ref{3.531}), there exist constants $M_{1\mathbf{s}},M_{2\mathbf{s}},\delta_{\mathbf{s}}$ such that
\begin{equation}
    \begin{aligned}\label{3.53}
    &\|(1-F_M(x,p))\tilde{\Omega}^*_{+}(\mathbf{s})f\|_{L^2_x}< \varepsilon,\ \forall M\geq M_{1\mathbf{s}},\\
    &\lim_{\mathbf{u}\to \mathbf{s}}\|(1-F_M(x,p))(\tilde{\Omega}^*_{+}(\mathbf{u})-\tilde{\Omega}^*_{+}(\mathbf{s}))f\|_{L^2_x}< \varepsilon, \ \forall M\geq M_{1\mathbf{s}},\\
    &\|(1-F_M(x,p))(\tilde{\Omega}^*_{+}(\mathbf{u})-\tilde{\Omega}^*_{+}(\mathbf{s}))f\|_{L^2_x}< 2\varepsilon,\  \forall M\geq M_{1\mathbf{s}},\ \forall|\mathbf{u}-\mathbf{s}|< \delta_{\mathbf{s}}.
\end{aligned}
\end{equation}Since $\mathbb{T}_1\times\cdots\times\mathbb{T}_N$ is compact and the set family $\{\mathbf{B}(\mathbf{s},\frac{\delta_{\mathbf{s}}}{2})\}$ covers the space, we have $$\mathbb{T}_1\times\cdots\times\mathbb{T}_N\subset \bigcup_{j=1}^{m}\mathbf{B}(\mathbf{s}_j,\frac{\delta_{\mathbf{s}_j}}{2}).$$
Let $M_0=\max\limits_{j=1,2\cdots m}\{M_{1\mathbf{s}_j}\vee M_{2\mathbf{s}_j}\}$. From \eqref{3.53}, for any $\varepsilon>0$,  $M>M_0$ and $\mathbf{u}\in \mathbb{T}_1\times\cdots\times\mathbb{T}_N$, we have $$\|(1-F_M(x,p))\tilde{\Omega}^*_{+}(\mathbf{u})f\|_{L^2_x}< 3\varepsilon.$$
Thus we get \eqref{3.48} and finish the proof.
\end{proof}

\section{Elaborate estimate for each part of the operator}

\subsection{Momentum-Position Weight Absorption Mechanism of   the operator \texorpdfstring{$C_M(t)$}{(C M(t))}}
Recall that \begin{align*}
   C_M(t)\coloneqq C(t)\Omega_{+}(t)F_{M}(x,p) \Omega^*_{+}(t).
\end{align*}
We will use two distinct Duhamel expansion with respect to $\Omega^*_{+}(t)$. The first is given by
\begin{equation}\label{D1}
    \begin{aligned}
    \Omega^*_{+}(t)&=\Omega^*_{+}(t)\Omega_{+}(t)+\Omega^*_{+}(t)(1-\Omega_{+}(t))\\
    &=1+(-i)\int_0^\infty e^{iuH_0}\Omega^*_{+}(t+u)V(x,t+u)e^{-iuH_0}du,
\end{aligned}
\end{equation}
using the $\Omega^*_+(t)U(t,t+u)=e^{iuH_0}\Omega^*_+(t+u)$ (see Lemma \ref{LemmaA.1} in Appendix) and the Duhamel principle for the $\Omega_{+}(t)$, while the other is
\begin{equation}\label{D2}
    \begin{aligned}
    \Omega^*_{+}(t)=1+(-i)\int_0^\infty e^{iuH_0}V(x,t+u)U(t+u,t)du,
\end{aligned}
\end{equation}which is directly follow from the Duhamel principle and the definition of the $\Omega^*_{+}(t)$.

 Through careful observation, the operator
\begin{align*}
    O(u)= \langle x\rangle^{-\frac{5}{2}}e^{-iuH_0}|p|^{-\frac32} \langle x\rangle^{\frac{1}{2}+\epsilon}
\end{align*}
 enjoys a time-weighted $L^2_x$ decay property, as formalized in the following lemma.
\begin{lemma}\label{Lemma4.1}
For all $\epsilon\in (0,\tfrac14),n\geq 5$, $O(t)$ satisfies
\begin{align*}
    \int_{0}^{\infty}\langle t\rangle^{-2}\|O(t)f\|_{L^2_x}^2dt\lesssim_{n,\epsilon}\|f\|_{L^2_x}^2.
\end{align*}
\end{lemma}
\begin{proof}For the $O(t)F_{\leq1}(|x|)$, using Lemma \ref{HLSvar}, the following is obviously
    \begin{equation}\label{eq:43}
        \begin{aligned}
        &\|O(t)F_{\leq1}(|x|)\|_{2\to2}\\
        \leq &\|\langle x\rangle^{-\frac{5}{2}}|x|^{\frac{3}{2}}\|_{L_x^\infty}\||x|^{-\frac{3}{2}}|p|^{-\frac32}\|_{2\to2}\|e^{-itH_0}\|_{2\to2}\|\langle x\rangle^{\frac{1}{2}+\epsilon}F_{\leq 1}(|x|)\|_{L_x^\infty}\\
        \lesssim&1,
    \end{aligned}
    \end{equation}where $n\geq4$.
  Then, we estimate $O(t)F_{>1}(|x|)$.  We decompose  $O(t)F_{>1}(|x|)$ into $n$ pieces
  \begin{align}\label{eq:4.4}
       O(t)F_{>1}(|x|)=\sum_{j=1}^n O_j(t),
  \end{align}
  where operators $O_j(t)$, $j=1,\cdots,n$ are defined as
  \begin{align*}
      O_j(t)\coloneqq O(t)F_{>1}(|x|)F_j(x),
  \end{align*}
with $\{F_j(x)\}_{j=1}^n$ a smooth partition of unity $\{F_j\}_{j=1}^{j=n}$ such that
\begin{align*}
    |x_j|\geq \frac{|x|}{n},\quad \mbox{for all } x\in \supp(F_j).
\end{align*}
We establish estimates solely for $O_1(t)$, as the analogous bounds for $O_j(t)$, $j=2,\cdots,n$ follow through identical arguments.
We now use the identity
\begin{align*}
    e^{-itH_0}x_1=(x_1-4tp_1|p|^2)e^{-itH_0},
\end{align*}
which follows from the elementary commutator identities calculations, to decompose $O_1$  into three parts
\begin{align*}
    O_1(t)=\sum_{j=1}^3 O_{1j}(t),
\end{align*}
where   operators  $O_{1j}(t)$, $j=1,2,3,$ are given by
\begin{align*}
    O_{11}(t)\coloneqq\langle x\rangle^{-\frac{5}{2}}x_1e^{-itH_0}|p|^{-\frac32} \langle x\rangle^{\frac{1}{2}+\epsilon}\frac{F_1(x)}{x_1}F_{>1}(|x|),
\end{align*}
\begin{align*}
    O_{12}(t)\coloneqq-4t\langle x\rangle^{-\frac{5}{2}}p_1|p|^{\frac12}e^{-itH_0}  \langle x\rangle^{\frac{1}{2}+\epsilon}\frac{F_1(x)}{x_1}F_{>1}(|x|)
\end{align*}
and
\begin{align*}
    O_{13}(t)\coloneqq\langle x\rangle^{-\frac{5}{2}}[|p|^{-\frac32},x_1]e^{-itH_0}  \langle x\rangle^{\frac{1}{2}+\epsilon}\frac{F_1(x)}{x_1}F_{>1}(|x|)
\end{align*}
respectively. Using Lemma \ref{HLSvar}, the unitarity of $e^{-itH_0}$ and the calculation of the commutators, we obtain
\begin{equation}\label{4.15}
    \begin{aligned}
        \|O_{11}(t)\|_{2\to2}&\leq \|\langle x\rangle^{-\frac{5}{2}}x_1|x|^{\frac32}\|_{L_x^\infty}\||x|^{-\frac32}|p|^{-\frac32}\|_{2\to2}\|e^{-itH_0}\|_{2\to2}\|\langle x\rangle^{\frac{1}{2}+\epsilon}\frac{F_1(x)}{x_1}F_{>1}(|x|)\|_{L_x^\infty}\\
        &\lesssim_n1,
    \end{aligned}
\end{equation}
\begin{equation}
    \begin{aligned}
        \|O_{13}(t)\|_{2\to2}&\leq \|\langle x\rangle^{-\frac{5}{2}}|x|^{2+\epsilon}\|_{L_x^\infty}\||x|^{-2-\epsilon}|p|^{-2-\epsilon}\|_{2\to2}\||p|^{2+\epsilon}[|p|^{-\frac32},x_1]|p|^{\frac12-\epsilon}\|_{2\to2}\\
        &\qquad\times \|e^{-itH_0}\|_{2\to2}\||p|^{-\frac12+\epsilon}|x|^{-\frac12+\epsilon}\|_{2\to2}\|\langle x\rangle\frac{F_1(x)}{x_1}F_{>1}(|x|)\|_{L_x^\infty}\\
        &\lesssim_{n,\epsilon}1,
    \end{aligned}
\end{equation}where $n\geq 5$.
By $L^2$ local smoothing estimate \eqref{localsmooth}, $O_{12}(t)$
satisfies
\begin{equation}\label{4.17}
    \begin{aligned}
      \int_{0}^\infty \langle t\rangle^{-2}  \|O_{12}(t)f\|_{L_x^2}^2dt&\lesssim\int_{0}^\infty\|\langle x\rangle^{-\frac{5}{2}}p_1|p|^{\frac12}e^{-itH_0}  \langle x\rangle^{\frac{1}{2}+\epsilon}\frac{F_1(x)}{x_1}F_{>1}(|x|)f\|_{L_x^2}^2dt\\
      &\lesssim \|\frac{p_1}{|p|}\langle x\rangle^{\frac{1}{2}+\epsilon}\frac{F_1(x)}{x_1}F_{>1}(|x|)f\|_{L^2_x}^2\\
      &\lesssim_{n,\epsilon}\|f\|_{L_x^2}^2
    \end{aligned}
\end{equation}
for all $f\in L^2_x(\mathbb{R}^n)$. From \eqref{4.15}-\eqref{4.17}, we get \begin{equation}\label{O1}
    \begin{aligned}
        \int_{0}^\infty \langle t\rangle^{-2}  \|O_{1}(t)f\|_{L^2_x}^2dt\lesssim_{n,\epsilon}\|f\|_{L^2_x}^2.
    \end{aligned}
\end{equation}Similarly, we can also prove that \eqref{O1} is valid with $O_1(t)$ replaced by $O_j(t)$, $j=2,\cdots,n$. Following from \eqref{eq:4.4}, we have
\begin{equation*}
    \begin{aligned}
        \int_{0}^\infty \langle t\rangle^{-2}  \|O(t)F_{>1}(|x|)f\|_{L_x^2}^2dt\lesssim_{n,\epsilon}\|f\|_{L_x^2}^2.
    \end{aligned}
\end{equation*}
Combining with \eqref{eq:43}, we have finished the proof of this lemma.
\end{proof}

On the other hand, we can prove the following lemma.
\begin{lemma}\label{Lemma4.2}
  Under the condition $\langle x\rangle^{\alpha}V\in L^{\infty}_{t,x}(\mathbb{R}^{n+1})$ with $0<\alpha<\frac{n}{2}$, $n\geq 6$, we have  \begin{equation*}
    \begin{aligned}
        \|(U(t+s,t)-e^{-isH_0})\frac{\langle p\rangle^{\alpha}}{|p|^\alpha}\|_{2\to2}\leq s \|\langle x\rangle^{\alpha}V(x,t)\|_{L^\infty_{t,x}}.
    \end{aligned}
\end{equation*}
\end{lemma}
\begin{proof}
By using the Duhamel formula
\begin{align}\label{D3}
    U(t+s,t)=e^{-isH_0}+(-i)\int_0^sU(t+s,t+u)V(x,t+u)e^{-iuH_0}du
\end{align}
and the unitarity of
$U(t+s,t+u)$, we obtain
    \begin{equation*}
    \begin{aligned}
        \|(U(t+s,t)-e^{-isH_0})&\frac{\langle p\rangle^{\alpha}}{|p|^\alpha}\|_{2\to2}\leq \int_0^s\|U(t+s,t+u)V(x,t+u)e^{-iuH_0}\frac{\langle p\rangle^{\alpha}}{|p|^\alpha}\|_{2\to2}du\\
        \lesssim& s \sup_{u}\|U(t+s,t+u)\|_{2\rightarrow2}\|\langle x\rangle^{\alpha}V(x,t)\|_{L^\infty_{t,x}}\|\langle x\rangle^{-\alpha}\frac{\langle p\rangle^{\alpha}}{|p|^\alpha}\|_{2\to 2}\|e^{-iuH_0}\|_{2\rightarrow2} \\
        \lesssim & s \|\langle x\rangle^{\alpha}V(x,t)\|_{L^\infty_{t,x}},
    \end{aligned}
\end{equation*}
where we use Lemma \ref{HLSvar} to prove $\|\langle x\rangle^{-\alpha}\frac{\langle p\rangle^{\alpha}}{|p|^\alpha}\|_{2\to 2}\lesssim1$, $0<\al< \frac{n}{2}, \ n\geq6$.
This completes the proof.
\end{proof}

Next, based on Lemmas \ref{Lemma4.1} and \ref{Lemma4.2}, we have the following  Momentum-Position weight absorption mechanism for $C_M(t)$.
\begin{lemma}\label{Lemma4.3}
  Under the condition $\langle x\rangle^{\frac{9}{2}}V\in L^{\infty}_{t,x}(\mathbb{R}^{n+1})$,  we have
    \begin{equation}\label{C_M1}
        \begin{aligned}
            \sup_{t\in\mathbb{R}}\|C_M(t)|p|^{-\frac{3}{2}}\langle x\rangle^{\frac12+\epsilon}\|_{2\to2}\lesssim_{M,\epsilon}1
        \end{aligned}
    \end{equation}
    and
\begin{equation}\label{C_M2}
    \begin{aligned}
        \sup_{t\in\mathbb{R}}\|C_M(t) \frac{\langle p\rangle^{\alpha}}{|p|^\alpha}\|_{2\to2}\lesssim_{M,\alpha}1,
    \end{aligned}
\end{equation}
for all $M\geq 1$, $\epsilon\in (0,1/4)$ and $0<\alpha<\frac{n}{2}$ when $n\geq 6$.
\end{lemma}
\begin{proof}
We estimate \eqref{C_M1} first.
Revisit that
\begin{align*}
   C_M(t)\coloneqq C(t)\Omega_{+}(t)F_{M}(x,p) \Omega^*_{+}(t).
\end{align*}
By using Duhamel principle \eqref{D1}, $ C_M(t)$ reads as
\begin{align}\label{F4}
     C_M(t)|p|^{-\frac{3}{2}}\langle x\rangle^{\frac12+\epsilon}= C_{M,1}(t)+C_{M,2}(t),
\end{align}
where
\begin{align*}
    C_{M,1}(t)\coloneqq C(t)\Omega_{+}(t)F_{M}(x,p)|p|^{-\frac{3}{2}}\langle x\rangle^{\frac12+\epsilon}
\end{align*}
and
\begin{align*}
    C_{M,2}(t)\coloneqq(-i)\int_0^\infty C(t)\Omega_{+}(t)F_{M}(x,p)e^{iuH_0}\Omega_{+}^*(t+u)V(x,t+u)e^{-iuH_0}|p|^{-\frac{3}{2}}\langle x\rangle^{\frac12+\epsilon}du.
\end{align*}
By the definition of $F_M(x,p)$ \eqref{FM} and the boundedness of $C(t)$ and $\Omega_+(t)$ in Lemma \ref{Lemma2.9}, $C_{M,1}(t)$ satisfies
\begin{align}\label{F3}
    \|C_{M,1}(t)\|_{2\to2}\lesssim\|F_{M}(x,p)|p|^{-\frac{3}{2}}\langle x\rangle^{\frac12+\epsilon}\|_{2\to2}\lesssim_{M,\epsilon}1.
\end{align}
Then, we estimate $C_{M,2}(t)$. We claim that
\begin{align}\label{F1}
    \|F_M(x,p)e^{iuH_0}\Omega_{+}^*(t+u)\langle x\rangle^{-2}\|_{2\to2}\lesssim_{M}\frac{1}{\langle u\rangle^2}.
\end{align}
Using Lemma \ref{FMest} and the unitary of $U(t,0)$, we obtain
\begin{equation*}
    \begin{aligned}
       &\| \int_0^\infty F_M(x,p)e^{i(u+s)H_0}V(x,t+u+s)U(t+u+s,t+u)ds\|_{2\to2}\\
       &\lesssim\int_0^\infty \frac{1}{\langle u+s\rangle^3}\|\langle x\rangle^3 V\|_{L^\infty_{t,x}} \|U(t+u+s,t+u)\|_{2\rightarrow2}ds\\
       &\lesssim_M\frac{1}{\langle u\rangle^2}\|\langle x\rangle^3 V\|_{L^\infty_{t,x}}.
    \end{aligned}
\end{equation*}
This, together with
\begin{align*}
    \|F_M(x,p)e^{iuH_0} \langle x\rangle^{-2}\|_{2\to2}\lesssim_{M}\frac{1}{\langle u\rangle^2}
\end{align*}
in Lemma \ref{FMest} and Duhamel principle \eqref{D2} to expand $\Omega_{+}^*(t+u)$, yields \eqref{F1}.
By using the boundedness of $C(t)$ and $\Omega_{+}(t)$ in Lemma \ref{Lemma2.9} and \eqref{F1}, we get
\begin{equation}\label{F2}
    \begin{aligned}
       \|C_{M,2}(t)\|_{2\to2}&\leq\int_0^\infty \|F_M(x,p)e^{iuH_0}\Omega_{+}^*(t+u)\langle x\rangle^{-2}\|_{2\to2} \|\langle x\rangle^{\frac92} V\|_{L^\infty_{u,x}}\|O(u)\|_{2\to2}du\\
       &\lesssim_M \int_0^\infty\frac{1}{\langle u\rangle^2}\|\langle x\rangle^{\frac92} V\|_{L^\infty_{u,x}}\|O(u)\|_{2\to2}du.
    \end{aligned}
\end{equation}
Applying Lemma \ref{Lemma4.1} and H\"{o}lder inequality to estimate \eqref{F2}, we have
\begin{equation*}
    \begin{aligned}
    \|C_{M,2}(t)\|_{2\to2}&\lesssim_{M,\epsilon}\left(\int_0^\infty\frac{1}{\langle u\rangle^2}du\right)^{\frac12} \left(\int_0^\infty\frac{1}{\langle u\rangle^2}\|\langle x\rangle^{\frac92} V\|_{L^\infty_{u,x}}^2\|O(u)\|_{2\to2}^2du\right)^{\frac12}\\
    &\lesssim_{M,\epsilon}\|\langle x\rangle^{\frac92} V\|_{L^\infty_{t,x}}.
\end{aligned}
\end{equation*}
Combining the result with estimate \eqref{F3} and \eqref{F4} yields
\begin{align*}
    \sup_{t\in\mathbb{R}}\|C_M(t)|p|^{-\frac{3}{2}}\langle x\rangle^{\frac12+\epsilon}\|_{2\to2}\lesssim_{M,\epsilon}1.
\end{align*}

To derive \eqref{C_M2},
we invoke the Duhamel principle \eqref{D2}, yielding the following:
\begin{equation*}
    \begin{aligned}
        C_M(t) \frac{\langle p\rangle^{\alpha}}{|p|^\alpha}=C(t)\Omega_{+}(t)F_{M}(x,p) \Omega^*_{+}(t)\frac{\langle p\rangle^{\alpha}}{|p|^\alpha}=C_{M,3}(t)+C_{M,4}(t),
    \end{aligned}
\end{equation*}
where the operators $C_{M,3}(t)$ and $C_{M,4}(t)$ are defined as
\begin{align*}
   C_{M,3}(t)\coloneqq C(t)\Omega_{+}(t)F_{M}(x,p)  \frac{\langle p\rangle^{\alpha}}{|p|^\alpha}
\end{align*}
and
\begin{align}\label{C_M_2}
    C_{M,4}(t)=(-i)\int_0^\infty C(t)\Omega_{+}(t)F_{M}(x,p)e^{isH_0}V(x,t+s)U(t+s,t)\frac{\langle p\rangle^{\alpha}}{|p|^\alpha}ds.
\end{align}
We observe that
\begin{align}\label{F5}
    \sup_{t\in\mathbb{R}}\|C_{M,3}(t)\|_{2\to2}\lesssim_{M,\alpha}1,
\end{align}
which follows from the fact that $\frac{\langle p\rangle^{\alpha}}{|p|^\alpha}$ is uniformly bounded on the support of $F_{M}(x,p)$. Actually, we can provide a more precise estimate for $\frac{\langle p\rangle^{\alpha}}{|p|^\alpha}$. Specifically,
\begin{align*}
    \frac{\langle p\rangle^{\alpha}}{|p|^\alpha}\lesssim_{C_0}\begin{cases}
        1, & \mbox{if} \quad|p|>1, \\
        \frac{1}{|p|^{\alpha}}, & \mbox{if}\quad |p|\leq 1.
    \end{cases}
\end{align*}

To estimate $C_{M,4}(t)$, applying Lemma \ref{Lemma4.2} in conjunction with the equation \eqref{C_M_2}, Lemma \ref{FMest} and Lemma \ref{HLSvar} yields
\begin{equation}\label{F6}
    \begin{aligned}
        \|C_{M,4}(t)\|_{2\to2}
        \leq&\int_0^\infty \|C(t)\Omega_{+}(t)\|_{2\to2} \|F_{M}(x,p)e^{isH_0}\langle x\rangle^{-\frac{5}{4}}\|_{2\to2}\\
        &\times\|\langle x\rangle^{\alpha+\frac{5}{4}}V(x,t+s)\|_{L^\infty_{t,x}}\|\langle x\rangle^{-\alpha}e^{-isH_0}\frac{\langle p\rangle^{\alpha}}{|p|^\alpha}\|_{2\to2}ds\\
        &+\int_0^\infty\|C(t)\Omega_{+}(t)\|_{2\to2} \|F_{M}(x,p)e^{isH_0}\langle x\rangle^{-\frac{9}{4}}\|_{2\to2}\\
        &\times\|\langle x\rangle^{\frac{9}{4}}V(x,t+s)\|_{L^\infty_{t,x}}\|(U(t+s,t)-e^{-isH_0})\frac{\langle p\rangle^{\alpha}}{|p|^\alpha}\|_{2\to2}ds\\
        \lesssim&\int_0^\infty \frac{1}{\langle s\rangle^{\frac{5}{4}}}\|\langle x\rangle^{\alpha+\frac{5}{4}}V(x,t+s)\|_{L^\infty_{t,x}}ds \\
        &+\int_0^\infty \frac{s}{\langle s\rangle^{\frac{9}{4}}}\|\langle x\rangle^{\frac{9}{4}}V(x,t+s)\|_{L^\infty_{t,x}}\|\langle x\rangle^{\alpha}V(x,t+s)\|_{L^\infty_{t,x}}ds\\
        \lesssim& (\|\langle x\rangle^{\alpha+\frac{5}{4}}V(x,t+s)\|^2_{L^\infty_{t,x}}+1)^2
    \end{aligned}
\end{equation}
for all $0<\alpha<\frac{n}{2}.$ Combining \eqref{F5} and \eqref{F6}, we obtain $$ \sup_{t\in\mathbb{R}}\|C_M(t) \frac{\langle p\rangle^{\alpha}}{|p|^\alpha}\|_{2\to2}\lesssim_{M,\alpha}1.$$
\end{proof}

\subsection{Weight-Reduction Mechanism of the operator \texorpdfstring{$C^{\pm}(t)$}{(C±(t))}}
Recall that
\begin{align*}
    C^\pm(t)=i\int_0^{\infty}P^\pm e^{\pm isH_0}V(x,t\pm s)U(t\pm s,t)ds.
\end{align*}
Since the sufficient spatial decay of the potential $V(x,t)$, it is expected that the weighted evolution operator $U(t\pm s,t)\frac{\langle p\rangle^{\alpha}}{|p|^\alpha}$ admits some weight absorption. This gain, nevertheless, comes at the cost of a $\frac{\alpha}{4}$-order (linear) time growth.
\begin{lemma}\label{Lemma4.4}
    For all $0<\alpha<\frac{n}{2},$ and $n\geq 9$,
    \begin{align}\label{B}
        \|\langle x\rangle^{-\alpha} U(t+s,t)\frac{\langle p\rangle^{\alpha}}{|p|^\alpha}\|_{2\to2}\lesssim_{n} \langle s\rangle^{\alpha/4}
    \end{align}
    holds true.
\end{lemma}
\begin{proof}
Let $B(t,s)=\langle x\rangle^{-\alpha} U(t+s,t)\frac{\langle p\rangle^{\alpha}}{|p|^\alpha}$. We only consider the case $s\geq 0$.
While the case $s < 0$ can be treated analogously. We decompose $B(t,s)$ into two parts:
\begin{align*}
    B(t,s)=B_1(t,s)+B_2(t,s),
\end{align*}
where the operators $B_j(t,s)$, $j=1,2,$ are defined as
\begin{align*}
    B_1(t,s) &= B(t,s)F_{> 1/\langle s\rangle^{1/4}}\left(|p| \right), \\
    B_2(t,s) &= B(t,s)F_{\leq1/\langle s\rangle^{1/4}}\left(|p| \right).
\end{align*}
For $B_1(t,s)$, using the estimate
\begin{align*}
    \|\frac{\langle p\rangle^{\alpha}}{|p|^\alpha}F_{>1/\langle s\rangle^{1/4}}\left(|p|\right)\|_{2\to 2}\lesssim\langle s\rangle^{\alpha/4},
\end{align*}
together with the unitarity of $U(t+s,t)$, yields
\begin{align}\label{B_1}
    \|B_1(t,s)\|_{2\to 2}\lesssim\langle s\rangle^{\alpha/4}.
\end{align}
Next, we estimate $B_2(t,s)$. For $\alpha<\frac{n}{2}$, application of Lemma \ref{HLSvar} and the unitary of $e^{-iuH_0}$ yields
\begin{align*}
    \|\langle x\rangle^{-\alpha} \frac{\langle p\rangle^{\alpha}}{|p|^\alpha}e^{-iuH_0}F_{\leq1/\langle s\rangle^{1/4}}\left(|p|\right)\|_{2\to 2}\lesssim1.
\end{align*}
Then, by the Duhamel formula \eqref{D3}, the unitarity of  $U(t +s, t +u)$ and $e^{-iuH_0}$ and Lemma \ref{HLSvar}, we deduce that
\begin{equation} \label{B_2}
    \begin{aligned}
        \|B_2(t,s)\|_{2\to 2}
        &\leq \left\| \langle x\rangle^{-\alpha} \frac{\langle p\rangle^{\alpha}}{|p|^\alpha} e^{-iuH_0} F_{\leq 1/\langle s\rangle^{1/4}}\left(|p|  \right) \right\|_{2\to 2} \\
        &\quad + \int_0^s \left\| \langle x\rangle^{-\alpha} U(t+s, t+u) V(x, t+u) e^{-iuH_0} \frac{\langle p\rangle^{\alpha}}{|p|^\alpha} F_{\leq 1/\langle s\rangle^{1/4}}\left(|p|  \right) \right\|_{2\to 2} \, du \\
        &\lesssim 1 + \int_0^s \|\langle x\rangle^4 V(x,t)\|_{L^\infty_{t,x}} \cdot \| |x|^{-4} |p|^{-4} \|_{2\to 2} \cdot \left\| \langle p\rangle^{\alpha} |p|^{4 - \alpha} F_{ \leq 1/\langle s\rangle^{1/4}}\left(|p| \right) \right\|_{2\to 2} \, du \\
        &\lesssim \langle s\rangle^{\alpha/4}.
    \end{aligned}
\end{equation}
Combining \eqref{B_1} and \eqref{B_2}, we obtain the desired estimate \eqref{B}.\end{proof}
The following proposition establishes a `Weight-Reduction Mechanism', which will be essential for implementing the Neumann series expansion in a finite summation.
\begin{proposition} \label{Prop4.5}Under the condition $\langle x\rangle^{2\al+\frac{9}{2}}V(x)\in L^\infty_{t,x}(\mathbb{R}^{n+1})$, for all $\beta\in (0,\frac{1}{16}]$, $0<\alpha<\frac{n}{2},$ $n\geq 9$,  we have
    \begin{align}\label{eq:424}
        \left\|\frac{|p|^{(\alpha-\beta)_+}}{\langle p\rangle^{(\alpha-\beta)_+}}C^{\pm}(t)\frac{\langle p\rangle^{\alpha}}{|p|^\alpha}\right\|_{2\to2}\lesssim_{n,\epsilon} 1,
    \end{align}
    where $(\alpha-\beta)_+\coloneqq\max\{\alpha-\beta,0\}$.
\end{proposition}
\begin{remark}
This proposition   indeed embodies a bootstrap-like or iterative property: it allows one to reduce the exponent of the $|p|$- or $\langle p\rangle$-
weight by a fixed amount $\beta$
in each application, ultimately reaching a manageable (low or zero) weight regime in finitely many steps. This mechanism is fundamental for closing weighted operator estimates in time-dependent spectral and dispersive analyses.
\end{remark}
\begin{proof} By Lemma \ref{Lemma4.4} and the estimate \eqref{NTE}, we obtain
 \begin{eqnarray*}
        \left\|\frac{|p|^{(\alpha-\beta)_+}}{\langle p\rangle^{(\alpha-\beta)_+}}\, C^{\pm}(t)\, \frac{\langle p\rangle^{\alpha}}{|p|^\alpha}\right\|_{2\to2}
       & \lesssim & \int_0^{\infty} \left\| \frac{|p|^{(\alpha-\beta)_+}}{\langle p\rangle^{(\alpha-\beta)_+}}\, P^+ e^{isH_0} V(x,t+s)\, U(t+s,t)\, \frac{\langle p\rangle^{\alpha}}{|p|^\alpha} \right\|_{2\to2}  \, ds \\
       & \lesssim & \int_0^{\infty} \left\| \frac{|p|^{(\alpha-\beta)_+}}{\langle p\rangle^{(\alpha-\beta)_+}}\, P^+ e^{isH_0} \langle x\rangle^{-(\alpha + \tfrac{9}{2})} \right\|_{2\to2} \\
       & &   \times \left\| \langle x\rangle^{2\alpha + \tfrac{9}{2}} V(x,t) \right\|_{L^{\infty}_{t,x}}
        \left\| \langle x\rangle^{-\alpha}\, U(t+s,t)\, \frac{\langle p\rangle^{\alpha}}{|p|^\alpha} \right\|_{2\to2}  \, ds \\
       & \lesssim & \int_0^{\infty} \frac{1}{\langle s\rangle^{(\frac{1}{4}-\epsilon)((\alpha-\beta)_+ + \frac{9}{2})}}\, \langle s\rangle^{\frac{\alpha}{4}}\,
        \left\| \langle x\rangle^{2\alpha + \tfrac{9}{2}} V(x,t) \right\|_{L^{\infty}_{t,x}} \, ds \\
       & \lesssim & \left\| \langle x\rangle^{2\alpha + \tfrac{9}{2}} V(x,t) \right\|_{L^{\infty}_{t,x}}
        \int_0^{\infty} \frac{1}{\langle s\rangle^{\frac{9}{8} + \frac{(\alpha-\beta)_+}{4} - \frac{\alpha}{4} - \epsilon((\alpha-\beta)_+ + \frac{9}{2})}} \, ds.
\end{eqnarray*}
    When  $\beta\in(0,\frac{1}{16})$ and $\epsilon\in (0,\frac{1}{100})$, it follows that
    \begin{equation*}
        \begin{aligned}
            \frac{9}{8}+\frac{(\alpha-\beta)_+}{4}-\frac{\alpha}{4}-\epsilon((\alpha-\beta)_+ + \frac{9}{2}) > \frac{9}{8}-\frac{1}{64}-\epsilon(\frac{n}{2}+\frac{9}{2})>1.
        \end{aligned}
    \end{equation*}
    Hence, for all   $\epsilon\in(0,\frac{1}{100})$, we can conclude that
   \begin{align*}
        \left\|\frac{|p|^{(\alpha-\beta)_+}}{\langle p\rangle^{(\alpha-\beta)_+}}C^{\pm}(t)\frac{\langle p\rangle^{\alpha}}{|p|^\alpha}\right\|_{2\to2}\lesssim_{n,\epsilon} 1.
    \end{align*} This completes the proof of \eqref{eq:424}.
\end{proof}

\begin{lemma}\label{Lemma4.7}
     For all $n\geq 9$, we have
      \begin{align}\label{op1}
          \|\langle x\rangle^{-4} |p|^{-\frac{3}{2}}\langle x \rangle^{\frac{5}{2}}\|_{2\to2}\lesssim_n1.
      \end{align}
\end{lemma}
\begin{proof}
Notice that
    \begin{align*}
        \|\langle x\rangle^{-4} |p|^{-\frac{3}{2}}\langle x \rangle^{\frac{5}{2}}\|_{2\to2}\leq \|\langle x\rangle^{-4}[|p|^{-\frac32},\langle x\rangle^{4}]\langle x\rangle^{-\frac32}\|_{2\to2}+\||p|^{-\frac32}\langle x\rangle^{-\frac32}\|_{2\to2}.
    \end{align*}
By Hardy–Littlewood–Sobolev inequality, we have
\begin{align*}
    \||p|^{-\frac32}\langle x\rangle^{-\frac32}\|_{2\to2}\leq \||p|^{-\frac32}|x|^{-\frac32}\|_{2\to2}\||x|^{\frac32}\langle x\rangle^{-\frac32}\|_{L^\infty_{t,x}}\lesssim1,
\end{align*}
for all $n\geq 4.$ To estimate $\|\langle x\rangle^{-4}[|p|^{-\frac32},\langle x\rangle^{4}]\langle x\rangle^{-\frac32}\|_{2\to2}$,
we decompose the commutator as follows:
   \begin{equation*}
    \begin{aligned}
        [\langle x\rangle^{4},|p|^{-\frac32}]&=\langle x\rangle^{2}[\langle x\rangle^{2},|p|^{-\frac32}]+[\langle x\rangle^{2},|p|^{-\frac32}]\langle x\rangle^{2}\\
        &= 2\langle x\rangle^{2}[\langle x\rangle^{2},|p|^{-\frac32}]-[\langle x\rangle^{2},[\langle x\rangle^{2},|p|^{-\frac32}]].
    \end{aligned}
\end{equation*}
So we only need to compute the commutators $[\langle x\rangle^{2},|p|^{-\frac32}]$  and   $[\langle x\rangle^{2},[\langle x\rangle^{2},|p|^{-\frac32}]]$. For the first one, we write
    \begin{align*}
        [\langle x\rangle^{2},|p|^{-\frac32}]&=x[ x,|p|^{-\frac32}]+[ x,|p|^{-\frac32}]x\nonumber\\
        &=-\tfrac{3i}{2}x\cdot p|p|^{-\frac{7}{2}}-\tfrac{3i}{2}|p|^{-\frac{7}{2}}p\cdot x\nonumber\\
        &=-3ix\cdot p |p|^{-\frac{7}{2}}+\sum_{j=1}^n\tfrac{3i}{2}[x_j,p_j|p|^{-\frac{7}{2}}]\nonumber\\
        &=-3ix\cdot p |p|^{-\frac{7}{2}}- \sum_{j=1}^n\tfrac{3}{2}\left( |p|^{-\frac{7}{2}} - \tfrac{7}{2}  |p|^{-\frac{11}{2}} p_j^2 \right)\\
        &=\sum_{j=1}^n\left( -3ix_j p_j |p|^{-\frac{7}{2}}-\tfrac{3}{2}|p|^{-\frac{7}{2}} +\tfrac{21}{4}  |p|^{-\frac{11}{2}} p_j^2 \right)\nonumber\\
        &=-3ix\cdot p|p|^{-\frac{7}{2}}-\frac{3}{2}(n-\frac{7}{2})|p|^{-\frac{7}{2}} \nonumber\\
        &\coloneqq A^0\nonumber.
    \end{align*}
For the second commutator, we compute
\begin{equation*}
    \begin{aligned}
     [\langle x\rangle^{2},[\langle x\rangle^{2},|p|^{-\frac32}]]&=\sum_{j=1}^n\left([x^2,-3ix_j p_j |p|^{-\frac{7}{2}}]\right)+[x^2,-\tfrac{3}{2}(n-\frac{7}{2})|p|^{-\frac{7}{2}}] \\
     &=\sum_{j=1}^n\sum_{k=1}^n\left([x_k^2,-3ix_j p_j |p|^{-\frac{7}{2}}]\right)+[x^2,-\tfrac{3}{2}(n-\frac{7}{2})|p|^{-\frac{7}{2}}]   \\
     &\coloneqq\sum_{j=1}^n\sum_{k=1}^n (A_{jk}^1)+A^2 \\
     &=A^1+A^2.
    \end{aligned}
\end{equation*}
$A_{jk}^1$ is computed as
\begin{equation*}
    \begin{aligned}
        A_{jk}^1&=x_jx_k[x_k,-3ip_j|p|^{-\frac{7}{2}}]+x_j[x_k,-3ip_j|p|^{-\frac{7}{2}}]x_k\\
        &=2x_jx_k[x_k,-3ip_j|p|^{-\frac{7}{2}}]-x_j[x_k, [x_k,-3ip_j|p|^{-\frac{7}{2}}]].
    \end{aligned}
\end{equation*}
Then we compute $[x_k,-3ip_j|p|^{-\frac{7}{2}}]$ and $x_j[x_k, [x_k,-3ip_j|p|^{-\frac{7}{2}}]]$ as
\begin{equation*}
    \begin{aligned}
       [x_k,-3ip_j|p|^{-\frac{7}{2}}]= 3\delta_{kj}|p|^{-\frac{7}{2}}-\tfrac{21}{2}p_jp_k|p|^{-\frac{11}{2}}
    \end{aligned}
\end{equation*}
and
\begin{equation*}
    \begin{aligned}
        \sum_{k=1}^n[x_k, [x_k,-3ip_j|p|^{-\frac{7}{2}}]]&=\sum_{k=1}^n\left([x_k,3\delta_{kj}|p|^{-\frac{7}{2}}]-[x_k,\tfrac{21}{2}p_jp_k|p|^{-\frac{11}{2}}]\right)\\
        &=-\tfrac{21i}{2}p_j|p|^{-\frac{11}{2}}-\frac{21}{2}i (1+n) p_j |p|^{-\frac{11}{2}} + \tfrac{231i}{4} p_j|p|^{-\frac{11}{2}}, \\
        &=\frac{21}{2}i(\frac{17-2n}{2})p_j|p|^{-\frac{11}{2}}.
    \end{aligned}
\end{equation*}
Thus, we have
\begin{equation*}
    \begin{aligned}
        \sum_{j=1}^n\sum_{k=1}^nA_{jk}^1=&6|x|^2|p|^{-\frac{7}{2}}-\sum_{j=1}^n\sum_{k=1}^n(21x_jx_kp_jp_k|p|^{-\frac{11}{2}}) \\
        &-\frac{21(17-2n)}{4}ix\cdot p|p|^{-\frac{11}{2}}.
    \end{aligned}
\end{equation*}
The computation of $A^2$ proceeds analogously to the preceding derivation:
\begin{equation*}
    \begin{aligned}
        A^2&=\tfrac{21i}{4}(n-\frac{7}{2})\sum_{k=1}^n \left( x_k p_k|p|^{-\frac{11}{2}}  + |p|^{-\frac{11}{2}} p_k x_k \right)\\
        &=\tfrac{21i}{4}(n-\frac{7}{2})\sum_{k=1}^n  \left( 2x_k p_k|p|^{-\frac{11}{2}}  - [x_k,|p|^{-\frac{11}{2}} p_k ] \right)\\
        &= \tfrac{21i}{4}(n-\frac{7}{2})\sum_{k=1}^n  \left( 2x_k p_k|p|^{-\frac{11}{2}} +\tfrac{11i}{2} |p|^{-\frac{15}{2}}p_k^2 - i |p|^{-\frac{11}{2}} \right) \\
        &=\frac{21}{4}i(n-\frac{7}{2})\left(2x\cdot p |p|^{-\frac{11}{2}}+(\frac{11}{2}-n)i|p|^{-\frac{11}{2}}\right)
    \end{aligned}
\end{equation*}
Therefore, we obtain the estimate
\begin{equation*}
    \begin{aligned}
        \|\langle x\rangle^{-4}[|p|^{-\frac32},\langle x\rangle^{4}]\langle x\rangle^{-\frac32}\|_{2\to 2}&\lesssim\|\langle x\rangle^{-2}A^0\langle x\rangle^{-\frac32}\|_{2\to 2} + \|\langle x\rangle^{-4}(A^1+A^2)\langle x\rangle^{-\frac32}\|_{2\to 2}.
    \end{aligned}
\end{equation*}
To bound these terms-exemplified by the estimate
\begin{equation*}
    \begin{aligned}
        \|\langle x\rangle^{-2}A^0\langle x\rangle^{-\frac32}\|_{2\to 2}\lesssim1,
    \end{aligned}
\end{equation*}
invoking the definition of $A^0$, the uniform bounds $\|\frac{p_j^{r}}{|p|^r}\|_{2\to 2}\leq1$ and $\|\frac{|x|^{r}}{\langle x\rangle^r}\|_{2\to 2}\leq1$ for all $r>0 $ and the Hardy–Littlewood–Sobolev inequality in Lemma \ref{HLSvar}, we establish
\begin{equation*}
    \begin{aligned}
        &\|\langle x\rangle^{-2}x_j p_j |p|^{-\frac{7}{2}}\langle x\rangle^{-\frac32}\|_{2\to 2} \\
        \leq  &\|\langle x\rangle^{-2}x_j |x|\|_{L^\infty_x}\||x|^{-1}|p|^{-1}\|_{2\to 2} \|\frac{p_j}{|p|}\|_{2\to 2}\||p|^{-\frac32}|x|^{-\frac{3}{2}}\|_{2\to 2}\||x|^{\frac{3}{2}}\langle x\rangle^{-\frac32}\|_{L^\infty_x}\\
        \lesssim&1
    \end{aligned}
\end{equation*}
and
\begin{equation*}
    \begin{aligned}
        \|\langle x\rangle^{-2}|p|^{-\frac{7}{2}}\langle x\rangle^{-\frac32}\|_{2\to 2}\leq\|\langle x\rangle^{-2}|p|^{-2}\|_{2\to 2}\||p|^{-\frac32}\langle x\rangle^{-\frac32}\|_{2\to 2}\lesssim1.
    \end{aligned}
\end{equation*}
Similarly, we can estimate the remaining terms $\|\langle x\rangle^{-4}A^1\langle x\rangle^{-\frac32}\|$ and $\|\langle x\rangle^{-4}A^2\langle x\rangle^{-\frac32}\|$ via identical techniques under the dimension constraint $n \geq 9$. This completes the proof of \eqref{op1}.
\end{proof}

\begin{lemma}\label{Lemma4.8}
     For all $n\geq 6$, we have
     \begin{align}
       \|\langle x\rangle^{-\frac{5}{2}} |p|^{-\frac{3}{2}} x_j  \|_{2\to2}\lesssim_n 1,\ j=1,\cdots,n.
     \end{align}
\end{lemma}
\begin{proof}
    We observe that
    \begin{equation*}
        \begin{aligned}
          \|\langle x\rangle^{-\frac{5}{2}} |p|^{-\frac{3}{2}} x_j  \|_{2\to2} \leq&  \|x_j \langle x\rangle^{-\frac{5}{2}} |p|^{-\frac{3}{2}}  \|_{2\to2}+ \|\langle x\rangle^{-\frac{5}{2}}  [x_j,|p|^{-\frac{3}{2}}] \|_{2\to2}.
        \end{aligned}
    \end{equation*}
Using the identity $[x_j,|p|^{-\frac{3}{2}}] =-\frac{3i}{2}\frac{p_j}{|p|^{\frac{7}{2}}}$, we apply Lemma \ref{HLSvar} and the bounds $\|\frac{p_j}{|p|}\|_{2\to2}\leq 1$ to obtain
\begin{equation*}
    \begin{aligned}
        \|\langle x\rangle^{-\frac{5}{2}} |p|^{-\frac{3}{2}} x_j  \|_{2\to2} \leq&\|x_j \langle x\rangle^{-\frac{5}{2}}|x|^{\frac{3}{2}}\|_{2\to2}\||x|^{-\frac{3}{2}}|p|^{-\frac{3}{2}}\|_{2\to2}\\
        +&\|\langle x\rangle^{-\frac{5}{2}} |x|^{\frac{5}{2}}\|_{2\to2}\||x|^{-\frac{5}{2}}|p|^{-\frac{5}{2}}\|_{2\to2}\|p_j/|p|\|_{2\to2}\\
        \lesssim&_n 1.
    \end{aligned}
\end{equation*}
\end{proof}

\begin{lemma}\label{Lemma4.9}
     For all $n\geq 9$ and $\epsilon\in (0,\frac14)$, we have
     \begin{align}
         \int_0^{\infty}\langle s\rangle^{-1-2\epsilon}\|\langle x\rangle^{-4}e^{-isH_0}|p|^{-\frac{3}{2}}\langle x\rangle^{\frac{1}{2}+\epsilon}\|_{2\to 2}ds\lesssim_{n,\epsilon}1.
     \end{align}
\end{lemma}
\begin{proof}
Application of Lemma \ref{Lemma4.7}, the endpoint Strichartz estimates for the bi-Laplacian free Schr\"odinger propagator in \cite[Theorem 1.4]{MY202410} and H\"older's inequality yields
\begin{equation*}
    \begin{aligned}
       & \int_0^{\infty}\langle s\rangle^{-1-2\epsilon}\|\langle x\rangle^{-4}e^{-isH_0}|p|^{-\frac32}\langle x\rangle^{\frac12+\epsilon}F_{\leq \langle s\rangle}(|x|)\|_{2\to 2}ds\\
       \lesssim&\int_0^{\infty} \langle s\rangle^{-1-2\epsilon}\|\langle x\rangle^{-4}|p|^{-\frac32} \langle x\rangle^{\frac{5}{2}}\|_{2\to 2}\|\langle x\rangle^{-\frac{5}{2}}e^{-isH_0}\|_{2\to 2}\langle s \rangle^{\frac12+\epsilon} ds\\
       \lesssim&\left(\int_0^{\infty} \langle s\rangle^{-1-2\epsilon}ds\right)^{1/2}\|e^{-isH_0}\|_{L^2_x\to L^2_sL_x^{\frac{2n}{n-4}}}\\
       \lesssim&_{n,\epsilon}1,
    \end{aligned}
\end{equation*}
for $n\geq 9$ and $\epsilon\in(0,1/4)$. Next, we consider the term
\begin{equation*}
    \begin{aligned}
        Q\coloneqq\int_0^\infty \langle s\rangle^{-1-2\epsilon}\|\langle x\rangle^{-4}e^{-isH_0}|p|^{-\frac32}\langle x\rangle^{\frac12+\epsilon}F_{> \langle s\rangle}(|x|)\|_{2\to 2}ds.
    \end{aligned}
\end{equation*}
To estimate $Q,$  we introduce a smooth partition of unity $\{F_j\}_{j=1}^{j=n}$ such that
\begin{align*}
    |x_j|\geq \frac{|x|}{n},\quad \mbox{for all } x_j\in \supp\{F_j\}.
\end{align*}
This implies that
\begin{align*}
    F_jF_{> \langle s\rangle}(|x|)=F_jF_{> \langle s\rangle}(|x|)F_{>|x|/n}(|x_j|)
\end{align*}
and
\begin{align}\label{Fj1}
    \|F_j\langle x_j\rangle^{-1/2-\epsilon}\langle x\rangle^{1/2+\epsilon}F_{>1}(|x|)\|_{2\to 2}\lesssim_{n}1.
\end{align}
Using the triangle inequality and \eqref{Fj1}, we reduce $Q$ into the following finite sum
\begin{align*}
    Q\lesssim_n\sum_{j=1}^n Q_j,
\end{align*}
where each $Q_j,$ is defined as
\begin{align*}
    Q_j\coloneqq\int_0^\infty \langle s\rangle^{-1-2\epsilon}\|\langle x\rangle^{-4}e^{-isH_0}|p|^{-3/2}\langle x\rangle^{1/2+\epsilon}F_{>\frac{|x|}{n}}(|x_j|)F_{>\langle s\rangle}(|x|)\|_{2\to 2}ds.
\end{align*}
We now use the identity
\begin{align*}
    e^{-iH_0}x_j=(x_j-4sp_j|p|^2)e^{-isH_0}
\end{align*}
to decompose $Q_j$ as
\begin{align*}
    Q_j\leq Q_{j1}+Q_{j2}+Q_{j3},
\end{align*}
where the individual terms are given by
\begin{align*}
    Q_{j1}\coloneqq \int_0^\infty \langle s\rangle^{-1-2\epsilon}\|\langle x\rangle^{-4}x_j|p|^{-\frac32}e^{-isH_0}\frac{\langle x\rangle^{\frac12+\epsilon}}{x_j}F_{>\frac{|x|}{n}}(|x_j|)F_{>\langle s\rangle}(|x|)\|_{2\to 2}ds,
\end{align*}
\begin{align*}
    Q_{j2}\coloneqq 4\int_0^\infty s\langle s\rangle^{-1-2\epsilon}\|\langle x\rangle^{-4}e^{-isH_0}|p|^{-\frac32}p_j|p|^2\frac{\langle x\rangle^{\frac{1}{2}+\epsilon}}{x_j}F_{>\frac{|x|}{n}}(|x_j|)F_{>\langle s\rangle}(|x|)\|_{2\to 2}ds
\end{align*}
and
\begin{align*}
    Q_{j3}\coloneqq \int_0^\infty \langle s\rangle^{-1-2\epsilon}\|\langle x\rangle^{-4}[|p|^{-\frac32},x_j]e^{-isH_0}\frac{\langle x\rangle^{\frac12+\epsilon}}{x_j}F_{>\frac{|x|}{n}}(|x_j|)F_{>\langle s\rangle}(|x|)\|_{2\to 2}ds.
\end{align*}
For $Q_{j1}$, we have
\begin{equation*}
    \begin{aligned}
       Q_{j1}\leq&  \int_0^\infty \langle s\rangle^{-1-2\epsilon}\|\langle x\rangle^{-4}x_j|p|^{-\frac32}\|_{2\to2}\|\frac{\langle x\rangle^{\frac12+\epsilon}}{x_j}F_{>\frac{|x|}{n}}(|x_j|)F_{>\langle s\rangle}(|x|)\|_{2\to 2}ds \\
       \lesssim& \int_0^\infty \langle s\rangle^{-1-2\epsilon}\frac{1}{\langle s\rangle^{\frac12-\epsilon}}ds\lesssim_{n,\epsilon}1.
    \end{aligned}
\end{equation*}
For $Q_{j2}$, using  the fact $\|\frac{p_j}{|p|}\|_{2\to2}\leq 1$, H\"older's inequality and the $L^2$ local smoothing estimate \eqref{localsmooth}, we get
\begin{eqnarray*}
       Q_{j2}&\leq & 4\int_0^\infty s\langle s\rangle^{-1-2\epsilon}\|\langle x\rangle^{-4}e^{-isH_0}|p|^{\frac32}\|_{2\to 2}\|\frac{p_j|p|^2}{|p|^3}\|_{2\to 2}\\
       &&\times\|\frac{\langle x\rangle^{\frac12+\epsilon}}{x_j}F_{>\frac{|x|}{n}}(|x_j|)F_{>\langle s\rangle}(|x|)\|_{2\to 2}ds  \\
      &\lesssim&\int_0^\infty \langle s\rangle^{-\frac12-\epsilon}\|\langle x\rangle^{-4}e^{-isH_0}|p|^{\frac32}\|_{2\to 2}ds\\
       &\lesssim&\left(\int_0^{\infty} \langle s\rangle^{-1-2\epsilon}ds\right)^{\frac12} \left(\int_0^{\infty} \|\langle x\rangle^{-4}|p|^{\frac32}e^{-isH_0}\|^2_{2\to 2} ds\right)^{\frac12}\\
      & \lesssim&_{n,\epsilon}1.
\end{eqnarray*}
For $Q_{j3},$ using the commutator identity $[x_j,|p|^{-\frac32}] =-\frac{3i}{2}\frac{p_j}{|p|^{\frac72}}$ and the uniform bounds $\|\frac{p_j}{|p|}\|_{2\to2}\leq 1$, we establish
\begin{equation*}
    \begin{aligned}
         Q_{j3}\lesssim&
         \int_0^\infty \langle s\rangle^{-1-2\epsilon}\|\langle x\rangle^{-4}e^{-isH_0}p_j|p|^{-\frac72}\frac{\langle x\rangle^{\frac12+\epsilon}}{x_j}F_{>\frac{|x|}{n}}(|x_j|)F_{>\langle s\rangle}(|x|)\|_{2\to 2}ds\\
         \lesssim&\int_0^\infty \langle s\rangle^{-1-2\epsilon}\|\langle x\rangle^{-4}|p|^{-2-\epsilon}\|_{2\to 2}\|\frac{p_j}{|p|}\|_{2\to 2}\||p|^{-\frac12+\epsilon}\langle x\rangle^{-\frac12+\epsilon}\|_{2\to 2}ds\\
         \lesssim&_{n,\epsilon}1.
    \end{aligned}
\end{equation*}
Synthesizing the preceding estimates, we have shown that for each $j$, $Q_j\lesssim_{n,\epsilon}1$ holds and whence $Q\lesssim_{n,\epsilon}1$. Together with the established bound for the spatial region $F_{\leq \langle s\rangle}(|x|)$, we conclude that
 \begin{align*}
         \int_0^{\infty}\langle s\rangle^{-1-2\epsilon}\|\langle x\rangle^{-4}e^{-isH_0}|p|^{-\frac32}\langle x\rangle^{\frac12+\epsilon}\|_{2\to 2}ds\lesssim_{n,\epsilon}1,
     \end{align*}
as desired.\end{proof}

\begin{lemma}\label{Lemma4.10}
 If $\langle x\rangle^{\sigma}V(x,t)\in L^\infty_{x,t}(\mathbb{R}^{n+1})$ for some $\sigma \geq 4,$ then for all $n\geq 9,$ and $\epsilon\in (0,\frac13)$,
 \begin{align*}
         \int_0^{\infty}\langle s\rangle^{-2-3\epsilon}\|\langle x\rangle^{-4}U(t+s,t)|p|^{-\frac32}\langle x \rangle^{\frac12+\epsilon}\|_{2\to2}ds\lesssim_{n,\epsilon}\|\langle x\rangle^{4}V(x,t)\|_{L^{\infty}_{x,t}}+1.
     \end{align*}
\end{lemma}
\begin{proof}
Let
\begin{align*}
    L\coloneqq\int_0^{\infty}\langle s\rangle^{-2-3\epsilon}\|\langle x\rangle^{-4}U(t+s,t)|p|^{-\frac32}\langle x \rangle^{\frac12+\epsilon}\|_{2\to2}ds.
\end{align*}
By using the Duhamel formula \eqref{D3}, together with Lemma \ref{Lemma4.9}, we have
    \begin{equation*}
        \begin{aligned}
        &\| (U(t+s,t)-e^{-isH_0})|p|^{-3/2}\langle x \rangle^{1/2+\epsilon}\|_{2\to2}\\\leq&  \int_0^s  \|U(t+s,t+u)V(x,t+u)e^{-iuH_0}|p|^{-3/2}\langle x \rangle^{1/2+\epsilon}\| _{2\to2}du\\
        \leq&\|\langle x \rangle^{4} V(x,t)\|_{L^\infty_{t,x}}\langle s\rangle^{1+2\epsilon}\int_0^s\langle u\rangle^{-1-2\epsilon}\|\langle x\rangle^{-4}e^{-iuH_0}|p|^{-3/2}\langle x \rangle^{1/2+\epsilon}\| _{2\to2}du\\
        \lesssim&_{n,\epsilon} \|\langle x \rangle^{4} V(x,t)\|_{L^\infty_{t,x}}\langle s\rangle^{1+2\epsilon}.
        \end{aligned}
    \end{equation*}
    Hence, we conclude that
\begin{equation*}
    \begin{aligned}
  \|L\|_{2\to2}\lesssim &\int_0^{\infty}\langle s\rangle^{-2-3\epsilon}\|\langle x\rangle^{-4}e^{-isH_0}|p|^{-\frac32}\langle x \rangle^{\frac12+\epsilon}\|_{2\to2}ds \\
  &+\int_0^{\infty}\langle s\rangle^{-2-3\epsilon}\|\langle x\rangle^{-4}(U(t+s,t)-e^{-isH_0})|p|^{-\frac32}\langle x \rangle^{\frac12+\epsilon}\|_{2\to2}ds. \\
  \lesssim&_{n,\epsilon}\int_0^{\infty}\langle s\rangle^{-1-2\epsilon}\|\langle x\rangle^{-4}e^{-isH_0}|p|^{-\frac32}\langle x \rangle^{\frac12+\epsilon}\|_{2\to2}ds\\
  &+\int_0^\infty\langle s\rangle^{-1-\epsilon}ds\|\langle x \rangle^{4} V(x,t)\|_{L^\infty_{t,x}} \\
  \lesssim&_{n,\epsilon} 1+\|\langle x \rangle^{4} V(x,t)\|_{L^\infty_{t,x}}.
    \end{aligned}
\end{equation*}
\end{proof}

\begin{proposition}\label{Prop4.11}
    Let $n\geq9$ and $\frac{n}{2}>\alpha >\frac72$  be two fixed numbers.  If  $ \langle x\rangle^{\tfrac{17}{2}+\alpha}V(x,t)\in L^{\infty}_{t,x}(\mathbb{R}^{n+1})$, then for all  $\epsilon\in(0,\min\{-\tfrac78+\frac{\alpha}{4},\tfrac12\})$,
    \begin{align}\label{C_pm1}
        \left\|\frac{|p|^\alpha}{\langle p\rangle^{\alpha}} C^{\pm}(t) |p|^{-\frac32}\langle x \rangle^{\frac12+\epsilon}\right\|_{2\to2}\lesssim_{n,\epsilon} \max_{j=1,2}\|\langle x\rangle^{\alpha +\frac{19}{2}}V(x,t)\|^j_{L^{\infty}_{x,t}}.
    \end{align}
\end{proposition}
\begin{proof}
We only estimate $$\left\|\frac{|p|^\alpha}{\langle p\rangle^{\alpha}} C^{+}(t) |p|^{-\frac32}\langle x \rangle^{\frac12+\epsilon}\right\|_{2\to2}$$ and  note that the term related to $ C^{-}(t)$ can be treated analogously.  By applying estimate \eqref{NTE}, we obtain the following bound, for all $\epsilon_1>0$ satisfying   $\frac{9}{8}+\frac{\alpha}{4}-\epsilon_1>2$:
\begin{equation*}
    \begin{aligned}
        &\left\|\frac{|p|^\alpha}{\langle p\rangle^{\alpha}} C^{+}(t) |p|^{-\frac32}\langle x \rangle^{\frac12+\epsilon}\right\|_{2\to2}\\
        \leq&\int_0^\infty\left\|\frac{|p|^\alpha}{\langle p\rangle^{\alpha}} P^{+} e^{isH_0} \langle x\rangle^{-(\frac{9}{2}+\alpha)}
 \right\|_{2\to2}\\
 &\times\|\langle x\rangle^{\frac{17}{2}+\alpha}V(x,u)\|_{L^\infty_{t,x}}\|\langle x\rangle^{-4}U(t+s,t)|p|^{-\frac32}\langle x \rangle^{\frac12+\epsilon} \|_{2\to2} ds\\
 \lesssim&\int_0^\infty \frac{1}{\langle s \rangle^{\frac{9}{8}+\frac{\alpha}{4}-\epsilon_1}}\|\langle x\rangle^{-4}U(t+s,t)|p|^{-\frac32}\langle x \rangle^{\frac12+\epsilon} \|_{2\to2}ds\|\langle x\rangle^{\frac{19}{2}+\alpha}V(x,u)\|_{L^\infty_{t,x}}.
    \end{aligned}
\end{equation*}
Combining this with Lemma \ref{Lemma4.10} (by choosing $\epsilon=\frac{1}{3}(-\frac78+\frac{\alpha}{4}-\epsilon_1)>0$), we obtain the desired estimate  \eqref{C_pm1}.\end{proof}

\subsection{Weight Absorption Property of \texorpdfstring{$P_b(t)$}{(P b(t))}}
The eigenfunctions corresponding to bound states exhibit sufficient spatial decay. So that, after weighting by $\langle x\rangle^{\delta}$, they remain in $L^2_x$
with uniformly bounded norm.

\begin{proposition}\label{Prop4.12}
    If Assumptions \eqref{asp:1} and \eqref{asp:2} are satisfied, then for all $\delta\in [0,\min\{\frac{n}{2}-4,4\}]$, $n\geq 9$,
    \begin{align}\label{P_b1}
        \sup_{t\in\mathbb{R}}\|P_b(t)\langle x\rangle^{\delta}\|_{2\to2}\lesssim_{\delta,n}1.
    \end{align}
\end{proposition}
\begin{proof}
    It suffices to fix some $\delta\in [\max\{0,\frac{n}{2}+16-\sigma\},\min\{\frac{n}{2}-4,4\}]$, where the set is nonempty under Assumption \eqref{asp:1} and the restraint $n\geq 9$. To show \eqref{P_b1} is equal to show for $j=1,\cdots,n$
    \begin{align}\label{x4}
        \sup_{t\in\mathbb{R}}\|P_b(t)\frac{x_j^4}{\langle x\rangle^{4-\delta}}\|_{2\to2}\lesssim_{\delta,n}1
    \end{align}
    and
    \begin{align}\label{x2}
        \sup_{t\in\mathbb{R}}\|P_b(t)\frac{x_j^2}{\langle x\rangle^{4-\delta}}\|_{2\to2}\lesssim_{\delta,n}1.
    \end{align}
We estimate the case when $j=1$ and the case $j=2,\cdots,n$ can be treated in same way. We then decompose $P_b(t)\frac{x_1^4}{\langle x\rangle^{4-\delta}}$  into two parts via the incoming/outgoing wave decompositions
\begin{equation*}
    \begin{aligned}
        P_b(t)\frac{x_1^4}{\langle x\rangle^{4-\delta}}=P_b(t)x_1^4P^+\frac{1}{\langle x\rangle^{4-\delta}}+P_b(t)x_1^4P^-\frac{1}{\langle x\rangle^{4-\delta}}.
    \end{aligned}
\end{equation*}
Notice that $P_b(t)\Omega_{+}(t)=0$, by Duhamel principle \eqref{OmeD}, the operator $P_b(t)x_1^4P^-\frac{1}{\langle x\rangle^{4-\delta}}$ can be written as
\begin{equation*}
    \begin{aligned}
        P_b(t)x_1^4P^-\frac{1}{\langle x\rangle^{4-\delta}}&=P_b(t)(1-\Omega_{+}(t))x_1^4P^-\frac{1}{\langle x\rangle^{4-\delta}}\\
        &=(-i)\int_0^\infty P_b(t)U(t,t+s)V(x,t+s)e^{-isH_0}x_1^4 P^-\frac{1}{\langle x\rangle^{4-\delta}}ds.
    \end{aligned}
\end{equation*}
In view of
\begin{equation}\label{decompofx4}
    \begin{aligned}
        e^{-isH_0}x_1^4 =&(x_1^4+s(-16x_1^3p_1|p|^2+R_1(x,p))+s^2(96x_1^2p_1^2|p|^4+R_2(x,p)) \\
        &+s^3(-256x_1p_1^3|p|^6+R_3(x,p))+256s^4p_1^4|p|^8)e^{-isH_0}, \\
    \end{aligned}
\end{equation}where, respectively, each $R_j(x)$ is the error term with the lower order relative to the main term, we obtain
    \begin{align}
         P_b(t)x_1^4P^-\frac{1}{\langle x\rangle^{3-\delta}}=&(-i)\int_0^\infty P_b(t)U(t,t+s)V(x,t+s)x_1^4e^{-isH_0} P^-\frac{1}{\langle x\rangle^{4-\delta}}ds\nonumber\\
      &+16i\int_0^\infty sP_b(t)U(t,t+s)V(x,t+s)x_1^3p_1|p|^2e^{-isH_0} P^-\frac{1}{\langle x\rangle^{4-\delta}}ds\nonumber\\
         &-96i\int_0^\infty s^2P_b(t)U(t,t+s)V(x,t+s)x_1^2 p_1^2|p|^4e^{-isH_0} P^-\frac{1}{\langle x\rangle^{4-\delta}}ds\nonumber\\
         &+256i\int_0^\infty s^3 P_b(t)U(t,t+s)V(x,t+s)  x_1p_1^3|p|^6e^{-isH_0} P^-\frac{1}{\langle x\rangle^{4-\delta}}ds\\
         &-256i\int_0^\infty s^4 P_b(t)U(t,t+s)V(x,t+s)  p_1^4|p|^8e^{-isH_0} P^-\frac{1}{\langle x\rangle^{4-\delta}}ds\nonumber\\
         &-i\sum_{j=1}^3\int_0^\infty  s^jP_b(t)U(t,t+s)V(x,t+s)  R_j(x,p)e^{-isH_0} P^-\frac{1}{\langle x\rangle^{4-\delta}}ds\nonumber\\
           \coloneqq&\sum_{k=1}^{6} I_k(t,x)\nonumber.
    \end{align}
For $I_1$, we have
\begin{equation}\label{estI1}
    \begin{aligned}
        \|I_1(t,x)\|_{2\to2}\lesssim&\int_0^\infty \|\langle x\rangle^9 V(x,t)\|_{L^\infty_{t,x}}\|\langle x\rangle^{-5}e^{-isH_0}P^-\|_{2\to 2}ds\\
        \lesssim& \int_0^\infty \frac{1}{\langle s\rangle^{\frac{9}{8}-\epsilon}}ds\|\langle x\rangle^9 V(x,t)\|_{L^\infty_{t,x}}.
    \end{aligned}
\end{equation}
 For $I_2(t,x)$, applying the fact $\|P_b(t)\|_{2\to2}\leq 1$,  \eqref{TSE}, \eqref{PLSE} and the unitarity of $U(t,t+s)$ yields
\begin{equation}\label{estI2}
    \begin{aligned}
        \|I_2(t,x)\|_{2\to2}\lesssim&\int_0^\infty s\|\langle x\rangle^{\frac{n}{2}+4+6-\delta} V(x,t)\|_{L^\infty_{t,x}}\|\langle x\rangle^{-(\frac{n}{2}+4+6-\delta)}x_1^3p_1|p|^2e^{-isH_0}P^-\frac{1}{\langle x\rangle^{4-\delta}}\|_{2\to 2}ds\\
        \lesssim& \int_0^\infty \frac{s}{\langle s\rangle^{\frac{\frac{n}{2}+4+3-\delta}{4}-\epsilon}} ds \lesssim_{\delta,n}1.
    \end{aligned}
\end{equation}
Thus, for $\delta\in [\max\{0,\frac{n}{2}+16-\sigma\},\min\{\frac{n}{2}-4,4\})$, we can similarly get the following estimates for $I_3(t,x)$, $I_4(t,x)$ and $I_5(t,x)$
\begin{equation}\label{estI3}
    \begin{aligned}
        \|I_3(t,x)\|_{2\to2}\lesssim&\int_0^\infty s^2\|\langle x\rangle^{\frac{n}{2}+4+8-\delta} V(x,t)\|_{L^\infty_{t,x}}\|\langle x\rangle^{-(\frac{n}{2}+4+8-\delta)}x_1^2p_1^2|p|^4e^{-isH_0}P^-\frac{1}{\langle x\rangle^{4-\delta}}\|_{2\to 2}ds \\
        \lesssim&\int_0^\infty \frac{s^2}{\langle s\rangle^{\frac{\frac{n}{2}+4+6-\delta}{4}-\epsilon}} ds\lesssim\int_0^\infty \frac{1}{\langle s\rangle^{\frac{\frac{n}{2}+2-\delta}{4}-\epsilon}} ds\lesssim1,
    \end{aligned}
\end{equation}
\begin{equation}\label{estI4}
    \begin{aligned}
        \|I_4(t,x)\|_{2\to2}\lesssim&\int_0^\infty s^3\|\langle x\rangle^{\frac{n}{2}+4+10-\delta} V(x,t)\|_{L^\infty_{t,x}}\|\langle x\rangle^{-(\frac{n}{2}+4+10-\delta)}x_1p_1^3|p|^6e^{-isH_0}P^-\frac{1}{\langle x\rangle^{4-\delta}}\|_{2\to 2}ds \\
        \lesssim&\int_0^\infty \frac{s^3}{\langle s\rangle^{\frac{\frac{n}{2}+4+9-\delta}{4}-\epsilon}} ds\lesssim\int_0^\infty \frac{1}{\langle s\rangle^{\frac{\frac{n}{2}+1-\delta}{4}-\epsilon}} ds\lesssim1
    \end{aligned}
\end{equation}
and
\begin{equation}\label{estI5}
    \begin{aligned}
        \|I_5(t,x)\|_{2\to2}\lesssim&\int_0^\infty s^4\|\langle x\rangle^{\frac{n}{2}+4+12-\delta} V(x,t)\|_{L^\infty_{t,x}}\|\langle x\rangle^{-(\frac{n}{2}+4+12-\delta)}p_1^4|p|^8e^{-isH_0}P^-\frac{1}{\langle x\rangle^{4-\delta}}\|_{2\to 2}ds \\
        \lesssim&\int_0^\infty \frac{s^4}{\langle s\rangle^{\frac{\frac{n}{2}+4+12-\delta}{4}-\epsilon}} ds\lesssim\int_0^\infty \frac{1}{\langle s\rangle^{\frac{\frac{n}{2}-\delta}{4}-\epsilon}} ds\lesssim1.
    \end{aligned}
\end{equation}

Combining \eqref{estI1}-\eqref{estI5}, we get
\begin{align}\label{P_b+}
    \sup_{t\in \mathbb{R}}\| P_b(t)x_1^4P^-\frac{1}{\langle x\rangle^{4-\delta}}\|_{2\to2}\lesssim_{\delta,n}\|\langle x\rangle^{\frac{n}{2}+16-\delta} V(x,t)\|_{L^\infty_{t,x}}.
\end{align}
Hence, similarly, we have
\begin{align}\label{P_b-}
    \sup_{t\in \mathbb{R}}\| P_b(t)x_1^4P^+\frac{1}{\langle x\rangle^{4-\delta}}\|_{2\to2}\lesssim_{\delta,n}\|\langle x\rangle^{\frac{n}{2}+16-\delta} V(x,t)\|_{L^\infty_{t,x}}.
\end{align}
Combining \eqref{P_b+} with \eqref{P_b-} yields
\begin{align*}
    \sup_{t\in \mathbb{R}}\| P_b(t)\frac{x_1^4 }{\langle x\rangle^{4-\delta}}\|_{2\to2}\lesssim_{\delta,n}\|\langle x\rangle^{\frac{n}{2}+16-\delta} V(x,t)\|_{L^\infty_{t,x}}.
\end{align*}
By the same way, we have
\begin{align*}
     \sup_{t\in \mathbb{R}}\| P_b(t)\frac{x_j^4 }{\langle x\rangle^{4-\delta}}\|_{2\to2}\lesssim_{\delta,n}\|\langle x\rangle^{\frac{n}{2}+16-\delta} V(x,t)\|_{L^\infty_{t,x}},\quad j=2,\cdots,n,
\end{align*} and \eqref{x4} holds. And for \eqref{x2}, similarly, we have \begin{equation}\label{decompx2}
    \begin{aligned}
        e^{-isH_0}x_1^2=\left(x_1^2-8sx_1p_1|p|^2+16s^2p_1^2|p|^4+4is(|p|^2+2p_1^2)\right)e^{-isH_0}.
    \end{aligned}
\end{equation}Repeating the above argument and using \eqref{decompx2} instead of \eqref{decompofx4}, we can get the same conclusion for \eqref{x2}.
Therefore, we conclude \eqref{P_b1}.
\end{proof}

\section{Proof of Theorems \ref{Thm1.1} and \ref{Thm1.3}}

\subsection{Proof of Theorems \ref{Thm1.1} and \ref{Thm1.3}}
Synthesizing the preceding lemmas and propositions, we establish the following proposition.
\begin{proposition}\label{Prop4.16}
Under Assumptions \ref{asp:1} and \ref{asp:2}, we have
    \begin{align}
        \int_0^\infty \|\langle x\rangle^{-\eta}(1-C_r(t))^{-1}P^{+}e^{-itH_0}\Omega_{+}^{*}f\|^2_{L_x^2}dt\lesssim_{n,\eta}\|f\|_{L_x^2}^2
    \end{align}
    and
      \begin{align}
        \int_0^\infty \|\langle x\rangle^{-\eta}(1-C_r(t))^{-1}P^{-}e^{-itH_0}\Omega_{-}^{*}f\|_{L_x^2}^2dt\lesssim_{n,\eta}\|f\|_{L_x^2}^2
    \end{align}
    for all $\eta>\frac52$, $n\geq 14$ and $f\in L^2_x(\mathbb{R}^n)$.
\end{proposition}
\begin{proof}
    Let $f\in L^2_x(\mathbb{R}^n)$. We estimate $\langle x\rangle^{-\eta}(1-C_r(t))^{-1}P^{+}e^{-itH_0}\Omega_{+}^{*}f$, while
    $\langle x\rangle^{-\eta}(1-C_r(t))^{-1}P^{-}e^{-itH_0}\Omega_{-}^{*}f$ can be treated similarly. We write $\langle x\rangle^{-\eta}(1-C_r(t))^{-1}P^{+}e^{-itH_0}\Omega_{+}^{*}f$  as
\begin{align*}
     (1-C_r(t))^{-1}P^{+}e^{-itH_0}\Omega_{+}^{*}f =f_1(t)+f_2(t),
\end{align*}
where $f_j(t)$, $j=1,2,$ are given by
\begin{align*}
    f_1(t)\coloneqq P^{+}e^{-itH_0}\Omega_{+}^{*}f
\end{align*}
and
\begin{align*}
    f_2(t)\coloneqq(1-C_r(t))^{-1}C_r(t)P^{+}e^{-itH_0}\Omega_{+}^{*}f.
\end{align*}
Invoking Lemma \ref{BoundP-Lem}, Corollary \ref{boundP2}, the $L^2$ local decay of the free flow \eqref{localdecay} and the unitary for $\Omega_{+}^{*}$ yields
\begin{equation*}
    \begin{aligned}
        \left(\int_0^\infty \|\langle x\rangle^{-\eta}f_1(t)\|_{L_x^2}^2dt\right)^{1/2}\leq&   \left(\int_0^\infty \|\langle x\rangle^{-\eta}P^{+}\langle x\rangle^{5/2}\|_{2\to 2}^2\|\langle x\rangle^{-5/2}e^{-itH_0}\Omega_{+}^{*}f\|_{L_x^2}^2dt\right)^{1/2} \\
        \lesssim& \|f\|_{L_x^2}^2
    \end{aligned}
\end{equation*}
for all $\eta>\frac{5}{2}.$ For $f_2(t),$
we write
\begin{equation*}
    \begin{aligned}
        f_2(t)&=(1-C_r(t))^{-1}\frac{\langle p\rangle^{\alpha}}{|p|^\alpha}\times\frac{|p|^\alpha}{\langle p\rangle^{\alpha}}C_r(t)|p|^{-3/2}\langle x\rangle^{1/2+\epsilon}\\
        &\qquad\qquad\qquad\qquad\qquad
        \times\langle x\rangle^{-1/2-\epsilon}|p|^{3/2} P^{+}e^{-itH_0}\Omega_{+}^{*}f.
    \end{aligned}
\end{equation*}
So proving
\begin{equation*}
    \begin{aligned}
        \left(\int_0^\infty \|\langle x\rangle^{-\eta}f_2(t)\|_{L_x^2}^2dt\right)^{1/2}\leq   \|f\|_{L_x^2}
    \end{aligned}
\end{equation*}
is equal to show
\begin{align}\label{C_r1}
    \|\langle x\rangle^{-\eta}(1-C_r(t))^{-1}\frac{\langle p\rangle^{\alpha}}{|p|^\alpha} \|_{2\to 2}\lesssim_{n,M,\eta,\alpha}1,
\end{align}
\begin{equation}\label{C_r2}
    \begin{aligned}
      \sup_{t\in\mathbb{R}} \left\|\frac{|p|^\alpha}{\langle p\rangle^{\alpha}} C_r(t) |p|^{-3/2}\langle x \rangle^{1/2+\epsilon}\right\|_{2\to2}
      \lesssim_{n,\epsilon,M}1
    \end{aligned}
\end{equation}
and
\begin{equation}\label{Localsmooth}
    \left(\int_0^\infty\|\langle x\rangle^{-1/2-\epsilon}|p|^{3/2}P^+e^{-itH_0}\Omega_{+}^{*}f\|_{L_x^2}^2dt\right)^{1/2}\lesssim\|f\|_{L_x^2}.
\end{equation}

The inequality \eqref{Localsmooth} can be obtained by
using Corollary \ref{boundP2}, the $L^2$ local smoothing estimate \eqref{localsmooth} and the unitary of $\Omega_{+}^{*}$.

For \eqref{C_r1}, by Proposition \ref{Prop4.5}, for all $\alpha\in (0,\frac{n}{2})$ and
$n\geq 9, $ with $(\alpha-\frac{1}{16})_+=\max\{0,\alpha-\frac{1}{16}\}$, we have
\begin{align}
    \left\|\frac{|p|^{(\alpha-\frac{1}{16})_+}}{\langle p\rangle^{(\alpha-\frac{1}{16})_+}}C(t)\frac{\langle p\rangle^{\alpha}}{|p|^\alpha}\right\|_{2\to2}\lesssim1.
\end{align}
This together with Proposition \ref{Prop4.12}, estimate \eqref{C_M2}
and the bounds $\|\frac{|p|^\alpha}{\langle p\rangle^{\alpha}}\|\leq 1,$ implies
\begin{equation}\label{C_r3}
    \begin{aligned}
        &\left\|\frac{|p|^{(\alpha-\frac{1}{16})_+}}{\langle p\rangle^{(\alpha-\frac{1}{16})_+}}C_r(t)\frac{\langle p\rangle^{\alpha}}{|p|^\alpha}\right\|_{2\to 2} \\ \leq   & \left\|\frac{|p|^{(\alpha-\frac{1}{16})_+}}{\langle p\rangle^{(\alpha-\frac{1}{16})_+}}C(t)\frac{\langle p\rangle^{\alpha}}{|p|^\alpha}\right\|_{2\to 2}+   \left\|\frac{|p|^{(\alpha-\frac{1}{16})_+}}{\langle p\rangle^{(\alpha-\frac{1}{16})_+}}C_M(t)\frac{\langle p\rangle^{\alpha}}{|p|^\alpha}\right\|_{2\to 2} + \left\|\frac{|p|^{(\alpha-\frac{1}{16})_+}}{\langle p\rangle^{(\alpha-\frac{1}{16})_+}}C(t)P_b(t)\frac{\langle p\rangle^{\alpha}}{|p|^\alpha}\right\|_{2\to 2}\\
        \lesssim& 1 +\left\|C_M(t)\frac{\langle p\rangle^{\alpha}}{|p|^\alpha}\right\|_{2\to 2}+\|P_b(t)\langle x\rangle^{\alpha}\|_{2\to 2} \left\|\langle x\rangle^{-\alpha}\frac{\langle p\rangle^{\alpha}}{|p|^\alpha}\right\|_{2\to 2}\lesssim_{n,M,\alpha}1,
    \end{aligned}
\end{equation}
where we use Hardy-Littlewood-Sobolev inequality in Lemma \ref{HLSvar}.
This result, together with
Hardy–Littlewood–Sobolev inequality in Lemma \ref{HLSvar} again and Neumann series, establishes \eqref{C_r1}, for all $\eta>\frac{5}{2}$ and $\alpha\in(\frac52,\min\{\eta, \frac{n}{2}\}).$ We fix $\alpha$ satisfying the condition. Then we take a fixed $n\in \N$ such that $0<\frac{\alpha}{n}<\frac{1}{16}$. Thus we have the following finite Neumann series
    \begin{align*}
        \langle x\rangle^{-\eta}(1-C_r(t))^{-1}\frac{\langle p\rangle^{\al}}{|p|^\al} &=\langle x\rangle^{-\eta}\frac{\langle p\rangle^{\al}}{|p|^\al}+\langle x\rangle^{-\eta}C_r(t)\frac{\langle p\rangle^{\al}}{|p|^\al} \\
        &\quad\cdots+ \langle x\rangle^{-\eta}(C_r(t))^{n-1}\frac{\langle p\rangle^{\al}}{|p|^\al}\\
       &\quad +\langle x\rangle^{-\eta}(1-C_r(t))^{-1}(C_r(t))^n\frac{\langle p\rangle^{\al}}{|p|^\al}.
    \end{align*}

We use the Hardy-Littlewood-Sobolev inequality in Lemma \ref{HLSvar} and \eqref{C_r3} to prove the \eqref{C_r1} by the following
\begin{align*}
    &\|\langle x\rangle^{-\eta}(1-C_r(t))^{-1}\frac{\langle p\rangle^{\al}}{|p|^\al}\|_{2\to 2}  \\
        \leq&\|\langle x\rangle^{-\eta}\frac{\langle p\rangle^{\al}}{|p|^\al}\|_{2\to 2}+\|\langle x\rangle^{-\eta}C_r(t)\frac{\langle p\rangle^{\al}}{|p|^\al}\|_{2\to 2} \\
        &\cdots+\|\langle x\rangle^{-\eta}(C_r(t))^{n-1}\frac{\langle p\rangle^{\al}}{|p|^\al}\|_{2\to 2} \\
        &+ \|\langle x\rangle^{-\eta}(1-C_r(t))^{-1}(C_r(t))^n\frac{\langle p\rangle^{\al}}{|p|^\al}\|_{2\to 2}. \\
        \lesssim&1+\|\langle x\rangle^{-\eta}\frac{\langle p\rangle^{\al-\frac{\al}{n}}}{|p|^{\al-\frac{\al}{n}}}\|_{2\to 2}\|\frac{|p|^{\al-\frac{\al}{n}}}{\langle p\rangle^{\al-\frac{\al}{n}}}C_r(t)\frac{\langle p\rangle^{\al}} {|p|^\al}\|_{2\to 2} \\
        & \cdots+\|\langle x\rangle^{-\eta}\frac{\langle p\rangle^{\frac{\al}{n}}}{|p|^{\frac{\al}{n}}}\|_{2\to 2}\|\frac{|p|^{\frac{\al}{n}}}{\langle p\rangle^{\frac{\al}{n}}}C_r(t)\frac{\langle p\rangle^{\frac{2\al}{n}}}{|p|^{\frac{2\al}{n}}}\|_{2\to 2}\cdots\|\frac{|p|^{\al-\frac{\al}{n}}}{\langle p\rangle^{\al-\frac{\al}{n}}}C_r(t)\frac{\langle p\rangle^{\al}} {|p|^\al}\|_{2\to 2} \\
        &+\|(1-C_r(t))^{-1}\|_{2\to2}\|\frac{|p|^{0}}{\langle p\rangle^{0}}C_r(t)\frac{\langle p\rangle^{\frac{\al}{n}}}{|p|^{\frac{\al}{n}}}\|_{2\to 2} \\
        &\times\|\frac{|p|^{\frac{\al}{n}}}{\langle p\rangle^{\frac{\al}{n}}}C_r(t)\frac{\langle p\rangle^{\frac{2\al}{n}}}{|p|^{\frac{2\al}{n}}}\|_{2\to 2}\cdots\|\frac{|p|^{\al-\frac{\al}{n}}}{\langle p\rangle^{\al-\frac{\al}{n}}}C_r(t)\frac{\langle p\rangle^{\al}} {|p|^\al}\|_{2\to 2} \lesssim 1.
\end{align*}

For \eqref{C_r2}, by applying  Proposition \ref{Prop4.11},
\begin{align*}
   \sup_{t\in\mathbb{R}} \left\|\frac{|p|^\alpha}{\langle p\rangle^{\alpha}} C^{\pm}(t) |p|^{-\frac32}\langle x \rangle^{\frac12+\epsilon}\right\|_{2\to2}\lesssim_{n,\epsilon} 1
\end{align*}
holds true for all $n\geq 9$, $\epsilon\in(0,\frac{1}{4})$ and $\alpha\in(\frac{5}{2},\frac{n}{2}).$ This together with \eqref{C_M1}
and estimates $\|\frac{|p|^\alpha}{\langle p\rangle^{\alpha}}\|\leq 1,$ yields
    \begin{align*}
        &\sup_{t\in\mathbb{R}} \left\|\frac{|p|^\alpha}{\langle p\rangle^{\alpha}} C_r(t) |p|^{-\frac32}\langle x \rangle^{\frac12+\epsilon}\right\|_{2\to2} \\
        \leq&\sup_{t\in\mathbb{R}} \left\|\frac{|p|^\alpha}{\langle p\rangle^{\alpha}} C(t) |p|^{-\frac32}\langle x \rangle^{\frac12+\epsilon}\right\|_{2\to2}+\sup_{t\in\mathbb{R}} \left\|\frac{|p|^\alpha}{\langle p\rangle^{\alpha}} C_M(t) |p|^{-\frac32}\langle x \rangle^{\frac12+\epsilon}\right\|_{2\to2}\\
        &+\sup_{t\in\mathbb{R}} \left\|\frac{|p|^\alpha}{\langle p\rangle^{\alpha}} C(t)P_b(t) |p|^{-\frac32}\langle x \rangle^{\frac12+\epsilon}\right\|_{2\to2}\\
        \lesssim&_{n,\epsilon} 1+\sup_{t\in\mathbb{R}} \left\|C_M(t) |p|^{-\frac32}\langle x \rangle^{\frac12+\epsilon}\right\|_{2\to2}+\sup_{t\in\mathbb{R}} \left\| P_b(t) |p|^{-\frac32}\langle x \rangle^{\frac12+\epsilon}\right\|_{2\to2}\\
        \lesssim&_{n,\epsilon,M}1+ \sup_{t\in\mathbb{R}}\left\| P_b(t) |p|^{-\frac32}\langle x \rangle^{1/2+\epsilon}\right\|_{2\to2}.
    \end{align*}

This, together with Proposition \ref{Prop4.12} and Lemma \ref{Lemma4.8}, yields, for $n\geq 14$,
\begin{equation}\label{5.17}
    \begin{aligned}
      \sup_{t\in\mathbb{R}} \left\|\frac{|p|^\alpha}{\langle p\rangle^{\alpha}} C_r(t) |p|^{-\frac32}\langle x \rangle^{\frac12+\epsilon}\right\|_{2\to2}   &\lesssim_{n,\epsilon,M}1+ \sup_{t\in\mathbb{R}}\| P_b(t)\langle x\rangle^{\frac52}\|\|\langle x\rangle^{-\frac52} |p|^{-\frac32}\langle x \rangle^{\frac12+\epsilon}\|_{2\to2}\\
      &\lesssim_{n,\epsilon,M}1.
    \end{aligned}
\end{equation}
We fix $M = M_0$. Combined with estimates \eqref{C_r1}, the result establishes
\begin{equation}
    \sup_{t\in\mathbb{R}}\|\langle x\rangle^{-\eta}(1-C_r(t))^{-1}C_r(t)|p|^{-\frac32}\langle x \rangle^{\frac12+\epsilon}\| \lesssim_{n,\eta}1
\end{equation}
for all $\epsilon\in (0,\frac14)$ and $\eta>\frac52$. Therefore, by \eqref{Localsmooth}, we conclude
\begin{align*}
     \int_0^\infty \|\langle x\rangle^{-\eta}(1-C_r(t))^{-1}P^{+}e^{-itH_0}\Omega_{+}^{*}f\|^2_2dt\lesssim_{n,\eta}\|f\|^2_{L^2_x}.
\end{align*}
Similarly, we have
 \begin{align*}
        \int_0^\infty \|\langle x\rangle^{-\eta}(1-C_r(t))^{-1}P^{-}e^{-itH_0}\Omega_{-}^{*}f\|_{L^2_x}^2dt\lesssim_{n,\eta}\|f\|_{L^2_x}^2.
    \end{align*}
\end{proof}

With the above proposition in hand, we can immediately proceed to prove Theorem \ref{Thm1.1}.

\begin{proof} [Proof of Theorem \ref{Thm1.1}.]
By \eqref{decomp} and Proposition \ref{smallforC_r}, we obtain that there exists $ M_0 > 0$
such that, whenever $M \geq M_0$,
\begin{equation*}
   \begin{aligned}
     U(t,0)P_c\psi_0&=(1-C_r(t))^{-1}P^{+}e^{-itH_0}\Omega_{+}^{*}\psi_0+(1-C_r(t))^{-1}P^{-}e^{-itH_0}\Omega_{-}^{*}\psi_0\\
     &\quad+(1-C_r(t))^{-1}C_M(t)U(t,0)P_c\psi_0
\end{aligned}
\end{equation*}
for $\psi_0\in L^2_x(\mathbb{R}^n)$ and $\|(1-C_r(t))^{-1}\|_{2\to2}\leq2$ for all $t\in \mathbb{R}$. This, together with Proposition \ref{Prop4.16}, \eqref{C_M3}, the unitary for $\Omega_{+}^{*}$ and
\begin{equation*}
    \begin{aligned}
        \left(\int_0^\infty \|F_M(x,p)e^{-itH_0}\psi_0\|_{L^2_x(\mathbb{R}^n)}^2dt\right)^{\frac12}\lesssim_M \|\psi_0\|_{L^2_x(\mathbb{R}^n)},\quad \forall \psi_0\in L^2_x,
        \end{aligned}
\end{equation*}
yields
\begin{equation*}
    \begin{aligned}
      \left(\int_0^\infty \|\langle x\rangle^{-\eta}U(t,0)P_c\psi_0\|_{L^2_x}^2dt\right)^{\frac12}
      \lesssim& \|\psi_0\|_{L^2_x(\mathbb{R}^n)}
       +\left(\int_0^\infty \|F_M(x,p)  e^{-itH_0}P_c\psi_0\|_{L^2_x}^2dt\right)^{\frac12} \\
       \lesssim&_{\eta,n,M} \|\psi_0\|_{L^2_x(\mathbb{R}^n)}
    \end{aligned}
\end{equation*}
for all $\eta>\frac52$ and $\psi_0\in L^2_x(\mathbb{R}^n)$ and $n\geq 14.$
\end{proof}

Invoking Assumption \ref{asp:3} reduces the spatial dimension requirement to 9, improving Theorem \ref{Thm1.1}. We first establish Theorem \ref{Prop4.16}'.
\begin{theoremprime}{5.1}\label{Prop4.16'}
    Under Assumptions \ref{asp:1}, \ref{asp:2} and \ref{asp:3}, we have
    \begin{align*}
        \int_0^\infty \|\langle x\rangle^{-\eta}(1-C_r(t))^{-1}P^{+}e^{-itH_0}\Omega_{+}^{*}f\|^2_{L_x^2}dt\lesssim_{n,\eta}\|f\|_{L_x^2}^2
    \end{align*}
    and
      \begin{align*}
        \int_0^\infty \|\langle x\rangle^{-\eta}(1-C_r(t))^{-1}P^{-}e^{-itH_0}\Omega_{-}^{*}f\|_{L_x^2}^2dt\lesssim_{n,\eta}\|f\|_{L^2_x}^2
    \end{align*}
    for all $\eta>\frac52$, $n\geq 9$ and $f\in L^2_x(\mathbb{R}^n)$.
\end{theoremprime}
\begin{proof}
    We prove the theorem using an argument identical to the proof of Theorem \ref{Prop4.16}. The only difference is that we use Assumption \ref{asp:3} to prove \eqref{5.17} instead of Proposition \ref{Prop4.12}. Therefore, we can obtain the same results and reduce the constraint from $n\geq 14$ to $n\geq 9$.
\end{proof}

We now establish Theorem \ref{Thm1.3}.

\begin{proof}[Proof of Theorem \ref{Thm1.3}.]
Following the same argument as in Theorem~\ref{Thm1.1}, the proof proceeds identically, except that Theorem~\ref{Prop4.16} is replaced by Theorem~\ref{Prop4.16'}.
\end{proof}

\subsection{Applications}

First, we recall the Strichartz estimates for the free flow, established in \cite{PB2007},\cite{KPV1991},\cite{MY202410}. For the admissible pair $(q,r,n), 2\leq q,r\leq \infty$ satisfy \begin{align}\label{admis}
    \frac{4}{q}+\frac{n}{r}=\frac{n}{2},
\end{align}we have \begin{itemize}
    \item \textbf{Homogeneous Strichartz Estimates}\begin{align}\label{homo}
        \|e^{-itH_0}P_c\psi_0\|_{L_t^qL_x^r}\lesssim\|\psi_0\|_{L^2_x},
    \end{align}
    \item  \textbf{Inhomogeneous Strichartz Estimates}\begin{align}\label{inhomo}
        \|\int_0^te^{-i(t-s)H_0}fds\|_{L_t^qL_x^r}\lesssim\|f\|_{L_t^{q'}L_x^{r'}}.
    \end{align}
\end{itemize}
In particular, the endpoint Strichartz estimates holds, that is, \eqref{homo} and \eqref{inhomo} remain valid as $(q,r)=(2,\frac{2n}{n-4})$. The endpoint Strichartz estimates for \eqref{inhomo} will be utilized in the subsequent proof.

\renewcommand{\proofname}{Proof of the Theorem \ref{Thm1.5}}
\begin{proof}
It suffices to establish the endpoint Strichartz estimates corresponding to $(q,r)=(2,\frac{2n}{n-4})$. By the Duhamel's principle, we have
    \begin{equation*}
        \begin{aligned}
        U(t,0)P_c(0)f&=e^{-itH_0}f+(-i)\int_0^tdse^{-i(t-s)H_0}V(x,s)U(s,0)P_c(0)f \\
        &\coloneqq\psi_1+\psi_2.
    \end{aligned}
    \end{equation*}
The first term   $\psi_1$ represents the free flow and clearly satisfies the endpoint Strichartz estimates. For the second term $\psi_2$, we apply the inhomogeneous Strichartz estimates for the free propagator together with \eqref{dis: eq}, obtaining
    \begin{equation*}
        \begin{aligned}
        \|\psi_2\|_{L^2_t L^{\frac{2n}{n-4}}_x }&\lesssim\|V(x,t)U(t,0)P_c(0)f\|_{L^{2}_t L^{\frac{2n}{n+4}}_x } \\
        &\lesssim\|V(x,t)\langle x\rangle^3\|_{L_t^\infty L_x^{n/2}}\|\langle x\rangle^{-3}U(t,0)P_c(0)f\|_{L_{t,x}^2} \\
        &\lesssim\|f\|_{L^2_x(\mathbb{R}^n)}.
    \end{aligned}
    \end{equation*}
 Combining the above bound with that of $\psi_1$, we conclude the desired endpoint Strichartz estimates, thereby completing the proof of Theorem \ref{Thm1.5}.
\end{proof}
\renewcommand{\proofname}{Proof}
\begin{corollary}
    By   Theorem \ref{Thm1.5}, we also can prove the inhomogeneous Strichartz estimate for the bi-Laplacian Schr\"odinger operator with quasi-periodic potential
    \begin{equation}\label{inhom}
    \|\int_0^\infty U(t,s)P_c(s)fds\|_{L_t^qL_x^r}\lesssim\|f\|_{L_t^{\rho'}L_x^{\gamma'}}
    \end{equation}where $(q,r), (\rho,\gamma)$ are both admissible pair satisfy \eqref{admis}.
\end{corollary}
\begin{proof}
    The argument reduces to establishing the endpoint Strichartz estimates. Given the validity of standard Strichartz estimates, we claim that their dual counterparts remain valid under quasi-periodic conditions: \begin{equation}\label{dualest}
        \|\int_0^\infty U(0,t)P_c(t)f(t)dt\|_{L^2}\lesssim\|f\|_{L_t^{q'}L_x^{r'}}.
    \end{equation}

    To prove the claim, we define the operator $A(t)=U(t,0)P_c(0)$, then we have $$\|A(t)\psi_0\|_{L_t^q L_x^r }\leq C\|\psi_0\|_{L^2_x(\mathbb{R}^n)}.$$Hence, the dual Strichartz estimates follows from duality in the sense of the space-time inner product. For $f\in L_t^{q'} L_x^{r'} , \psi_0\in L_x^2$, we obtain
    \begin{equation*}
        \begin{aligned}
            \langle A(t)\psi_0,f\rangle_{L_t^q L_x^r ,L_t^{q'} L_x^{r'} } =&\langle \psi_0,A^*f\rangle_{L_x^2,L_x^2},
        \end{aligned}
    \end{equation*}
where $A^*f=\int_0^\infty P_c(0)U(0,t)f(t)dt=\int_0^\infty U(0,t)P_c(t)f(t)dt$. By the adjoint relation and standard functional analytic arguments, it follows that
$$\|A^*\|_{L_t^{q'}L_x^{r'}\to L^2}\leq C, $$ which establishes the same estimate as for $A$.

Thus,
\begin{equation*}
    \begin{aligned}
     \|\int_0^\infty U(t,s)P_c(s)fds\|_{L_t^qL_x^r}&\lesssim\int_0^\infty(\int_0^\infty \|U(t,s)P_c(s)f\|^q_{L_x^r}dt)^{\frac{1}{q}} ds \\
    &\lesssim\int_0^\infty\|U(t,s)P_c(s)f\|_{L_t^qL_x^r} ds \\
    &\lesssim\int_0^\infty\|f\|_{L^2_x}ds\lesssim\|f\|_{L_t^1L_x^2}.
    \end{aligned}
\end{equation*}
Moreover, we prove another endpoint Strichartz estimates. By the \eqref{dualest}, we have
\begin{equation*}
    \begin{aligned}
      \|\int_0^\infty U(t,s)P_c(s)fds\|_{L^2_x}&=\|U(t,0)\int_0^\infty U(0,s)P_c(s)fds\|_{L^2_x} \\
    &\lesssim\|\int_0^\infty U(0,s)P_c(s)fds\|_{L^2_x} \\
&\lesssim\|f\|_{L_t^{q'}L_x^{r'}}.
    \end{aligned}
\end{equation*}
Therefore, we have
\begin{align*}
    \|\int_0^\infty U(t,s)P_c(s)fds\|_{L_t^\infty L_x^2}=\sup\limits_{t}\|\int_0^\infty U(t,s)P_c(s)fds\|\lesssim\|f\|_{L_t^{q'}L_x^{r'}},
\end{align*}
where we apply the Christ–Kiselev lemma to justify the time-ordering of the integral and ensure the rigor of the argument.

Finally, by interpolating between the two endpoint Strichartz estimates, we complete the proof of \eqref{inhom}.
\end{proof}

\appendix

\section[Appendix~\thesection: Proof of the estimates for free flow]{Proof of the estimates for free flow}\label{appendixA}

\subsection{Intertwining property of time-dependent wave operators }
\begin{lemma}\label{LemmaA.1}
The wave operator $\Omega^*_+(t)$ satisfies the following time-dependent intertwining property:
    \begin{align}
        \Omega^*_+(t)U(t,0)=e^{-itH_0}\Omega^*_+(0)
    \end{align}
    and
    \begin{align}\label{eqpm}
    \Omega^*_{\pm}(t)U(t,0)=e^{-itH_0}\Omega_{\pm}^*(0)
\end{align}in the weak topology of $L^2_x(\mathbb{R}^n)$.
\end{lemma}
\begin{proof}
To prove this lemma, we first introduce the operator $$\Omega^*(t)\psi(t)=w\text{-}\lim\limits_{t\to \infty}e^{isH_0}\psi(t+s).$$ We have \begin{equation}
        \begin{aligned}
            \Omega^*(t)\psi(t)=e^{-itH_0}\Omega^*(0)\psi(0),
        \end{aligned}
    \end{equation}
which have proved in \cite[Proposition 4.6]{SWWZ}. Since $$s\text{-}\lim\limits_{t\to \infty}\left(1-F_{\leq 1}(\frac{|x|}{t^\alpha})\right)f=0, \ \forall f\in  L^2_x(\mathbb{R}^n)$$for all $\alpha\in (0,1)$, for fixed $g\in L^2_x(\mathbb{R}^n)$, we have \begin{equation}
    \begin{aligned}
            0=&\limsup\limits_{s\to \infty}\left|\left((1-F_{\leq 1}(\frac{|x|}{s^\alpha}))g,e^{isH_0}\psi(t+s)\right)\right| \\
            =&\limsup\limits_{s\to \infty}\left|\left(g,e^{isH_0}\psi(t+s)\right)-\left(g,F_{\leq 1}(\frac{|x|}{s^\al})e^{isH_0}\psi(t+s)\right)\right|.
        \end{aligned}
    \end{equation}Thus, we have that $$\Omega^*(t)\psi(t)=\Omega^*_+(t)\psi(t)$$ holds in the weak topology in $L^2_x(\mathbb{R}^n)$. Furthermore, $$\Omega^*_+(t)U(t,0)\psi(0)=\Omega^*(t)U(t,0)\psi(0)=e^{-itH_0}\Omega^*(0)\psi(0)=e^{-itH_0}\Omega^*_+(0)\psi(0)$$holds in the weak topology of $L^2_x(\mathbb{R}^n)$. As for the \eqref{eqpm}, we can prove it similarly. This completes the proof.
\end{proof}

\subsection{Proof of Lemma \ref{mourre}}\label{Lemma2.4}

\begin{proof}Since it is easy to prove as time is short, we can assume that $t>1$. We only consider $P^+$, and the operator $P^-$ can be treated similarly. We have
\begin{equation}
    \begin{aligned}
        \|P^+F_{\geq 1}(|p|)F_{\leq 2}(|p|)e^{it(-\Delta)^2}\langle x\rangle^{-\delta}f\|_{L^2_x}\leq I_{near} +I_{far},
    \end{aligned}
\end{equation}
where
\begin{align}
    I_{near}=\|P^+F_{\geq 1}(|p|)F_{\leq 2}(|p|)e^{it(-\Delta)^2}F_{\leq t/10}(|x|)\langle x\rangle^{-\delta}f\|_{L^2_x}
\end{align}
and
\begin{align}
    I_{far}=\|P^+F_{\geq 1}(|p|)F_{\leq 2}(|p|)e^{it(-\Delta)^2}F_{\geq t/10}(|x|)\langle x\rangle^{-\delta}f\|_{L^2_x}.
\end{align}
First, it is straightforward to get
\begin{align}
     I_{far}\lesssim\langle t\rangle^{-\delta}\|f\|_{L^2_x}.
\end{align}
Then we only deal with $I_{near}$,

\begin{equation}
    \begin{aligned}
        I_{near}\leq \sum_{b=1}^N\|F^{\hat{h}_b}(\hat{x})\tilde{F}^{\hat{h}_b}(\hat{p})F_{\geq 1}(|p|)F_{\leq 2}(|p|)e^{it(-\Delta)^2}F_{\leq t/10}(|x|)\langle x\rangle^{-\delta}f\|_{L^2_x}.
    \end{aligned}
\end{equation}
Using the method of non-stationary phase, we obtain
\begin{equation}
    \begin{aligned}
        &\|F^{\hat{h}_b}(\hat{x})\tilde{F}^{\hat{h}_b}(\hat{p})F_{\geq 1}(|p|)F_{\leq 2}(|p|)e^{it(-\Delta)^2}\langle x\rangle^{-\delta}F_{\leq t/10}(|x|)f\|_{L^2_x} \\
        =&\|\iint e^{i\Phi_{t}(x,y,\xi)}F^{\hat{h}_b}(\hat{x})\tilde{F}^{\hat{h}_b}(\hat{\xi})F_{\geq 1}(|\xi|)F_{\leq 2}(|\xi|)\langle y\rangle^{-\delta}F_{\leq t/10}(|y|)f(y)dyd\xi\|_{L^2_x} \\
        =&\|\iint e^{i\Phi_{t}(x,y,\xi)}F^{\hat{h}_b}(\hat{x})(L^*)^M[\tilde{F}^{\hat{h}_b}(\hat{\xi})F_{\geq 1}(|\xi|)F_{\leq 2}(|\xi|)]\langle y\rangle^{-\delta}F_{\leq t/10}(|x|)f(y)dyd\xi\|_{L^2_x} \\
        \lesssim&\|\frac{1}{(|x|+t)^M}\int_{|y|\leq t/10}\langle y\rangle^{-\delta}f(y)dy\|_{L^2_x} \\
        \lesssim &\langle t\rangle^{-\delta}\|f\|_{L^2_x},
    \end{aligned}
\end{equation}where $\Phi_{t}(x,y,\xi)=(x-y)\cdot\xi+t|\xi|^4$, $|\nabla_{\xi}\Phi_{t}(x,y,\xi)|=|(x-y) +4t|\xi|^2\xi|\gtrsim |x+t\xi|$ and $e^{i\Phi_{t}(x,y,\xi)}=\frac{-i\nabla_{\xi}\Phi_{t}(x,y,\xi)}{|\nabla_{\xi}\Phi_{t}(x,y,\xi)|^2}\cdot\nabla_{\xi}e^{i\Phi_{t}(x,y,\xi)}\coloneqq L(e^{i\Phi_{t}(x,y,\xi)}).$

Combining the above process, we complete the proof.
\end{proof}

\subsection{Proof of Lemma \ref{allkindsestimates}}\label{Lemma2.5}

\begin{proof}[\bfseries\itshape Proof of  \eqref{HEE}:]
   When $t\in(0,1),$ by the definition of $\langle t\rangle$, we know that it is sufficient to prove the left of \eqref{HEE} is bounded. As the boundedness is immediate, we restrict to $t\geq 1$ and employ a circular decomposition:
   \begin{align}
   P^{\pm} F_{>2}(|p|)e^{\pm itH_0} \langle x\rangle^{-\delta}=\sum_{j=1}^{\infty}f_{\pm,j}(t)
\end{align}
where $f_{\pm,j}(t)= P^{\pm}F_{2^j}(|p|)  e^{\pm itH_0} \langle x\rangle^{-\delta}$
and
    \begin{align}
        F_{2^j}(|p|)\coloneqq  F_{\geq 2^j}(|p|)F_{<2^{j+1}}(|p|),\ j=0,1,2,\cdots,
    \end{align}
besides $F_{2^j}(|p|)$ satisfies
\begin{align}
    \sum_{j=1}^{\infty}F_{2^j}(|p|)=F_{>2}(|p|).
\end{align}

  Using dilation transformation and \eqref{fund_Estimate}, one has
\begin{equation}\label{circular_decomp}
    \begin{aligned}
    \|P^{\pm}F_{2^j}(|p|) & e^{\pm itH_0} \langle x\rangle^{-\delta}\|_{2\rightarrow2}=\|P^{\pm}F_{1}(|p|)  e^{\pm i2^{4j}tH_0}\langle x/2^j\rangle^{-\delta}\|_{2\rightarrow2}\nonumber\\
    &\leq  \|P^{\pm}F_{1}(|P|)  e^{\pm i2^{4j}tH_0}\langle x\rangle^{-\delta}\|_{2\rightarrow2} \|\langle x \rangle^{\delta} \langle x/2^j\rangle^{-\delta}\|_{2\rightarrow2}\nonumber\\
    &\lesssim_n   \frac{1}{\langle 2^{4j}t\rangle^\delta}2^{j\delta}\lesssim_n \frac{1}{2^{3j\delta }}\frac{1}{\langle t\rangle^\delta},
\end{aligned}
\end{equation}
then summing it we get
\begin{align}
  \|P^{\pm} F_{>2}(|p|)e^{\pm itH_0} \langle x\rangle^{-\delta}\|_{2\rightarrow2}&\leq \sum_{j=1}^{\infty}\|P^{\pm}F_{2^j}(|p|)  e^{\pm itH_0}\langle x\rangle^{-\delta}\|_{2\to2}\nonumber\\
  &\lesssim\sum_{j=1}^{\infty} \frac{1}{2^{3j\delta }}\frac{1}{\langle t\rangle^\delta}\lesssim_{n,l,\delta}\frac{1}{\langle t\rangle^\delta}.
\end{align}
This finishes the proof for $c=2$. The case $c\not=2, $ follows in a similar way by using dilation transformation.

\noindent\textbf{\textit{Proof of \eqref{PSE}}}: Similar to the proof of \eqref{circular_decomp}, we get
\begin{equation}
    \begin{aligned}
         &\|P^{\pm}F_{\geq M}(|p|)e^{\pm it|p|^4)} |p|^l\langle x\rangle^{-\delta}\|_{2\rightarrow2}\\
         =&M^l \|P^{\pm}F_{\geq 1}(|p|)e^{\pm itM^4|p|^4)} |p|^l\langle x/M\rangle^{-\delta}\|_{2\rightarrow2}\\
         \leq&\sum_{j=0}^{\infty}M^l \|P^{\pm}F_{2^j}(|p| )e^{\pm itM^4|p|^4)} |p|^l\langle x/M\rangle^{-\delta}\|_{2\rightarrow2}\\
         \lesssim&\sum_{j=0}^{\infty}M^l 2^{jl}\frac{1}{2^{3j\delta}}\frac{1}{\langle M^4 t\rangle^\delta} M^\delta
         \lesssim_{n,3\delta-l}\frac{1}{ M^{3\delta -l}t^{\delta}}.
    \end{aligned}
\end{equation}

\noindent\textbf{\textit{{Proof of \eqref{TSE}}}}: Here, we only prove the case $l=3$; results for $l=1,2$ follow via identical methodology.
\begin{equation}
    \begin{aligned}
        &\int_0^1 t^3\|P^{\pm}F_{\geq 1}(|p|)e^{\pm i(-\Delta)^2t}|p|^{9}\langle x\rangle^{-\delta}\|_{2\rightarrow2}dt\\
        \leq& \sum_{j=0}^{\infty}\int_0^1 t^3\|P^{\pm}F_{2^j}(|p|)e^{\pm i|p|^4t}|p|^{9}\langle x\rangle^{-\delta}\|_{2\rightarrow2}dt\\
        = &\sum_{j=0}^{\infty}\int_0^{2^{-\frac{5}{2}j}} t^3\|P^{\pm}F_{2^j}(|p|)e^{\pm i|p|^4t}|p|^{9}\langle x\rangle^{-\delta}\|_{2\rightarrow2}dt\\
        &+\sum_{j=0}^{\infty}\int_{2^{-\frac{5}{2}j}}^{1} t^3\|P^{\pm}F_{2^j}(|p|)e^{\pm i|p|^4t}|p|^{9}\langle x\rangle^{-\delta}\|_{2\rightarrow2}dt \\
        \coloneqq&\sum_{j=0}^{\infty}(A_{j,1}+A_{j,2}).
    \end{aligned}
\end{equation}
Similar to \eqref{circular_decomp}, we have
\begin{equation}
     \begin{aligned}
           A_{j,1}\lesssim&\int_0^{2^{-\frac{5}{2}j}} t^3 2^{9j}\|P^{\pm}F_{2^j}(|p|)e^{\pm i|p|^4t}\langle x\rangle^{-\delta}\|dt \\
           \lesssim&\int_0^{2^{-\frac{5}{2}j}} t^3 2^{9j}dt   \lesssim2^{-j}
     \end{aligned}
    \end{equation}
and
\begin{equation}
    \begin{aligned}
        A_{j,2}\lesssim \int_{2^{-\frac{5}{2}j}}^{1} t^3 2^{9j} \frac{1}{\langle2^{4j}t\rangle^\delta}2^{j\delta}dt\lesssim \frac{1}{2^{3j(\delta-3)}}.
    \end{aligned}
\end{equation}
Thus we have
\begin{align}
    \int_0^1 t^3\|P^{\pm}F_{\geq 1}(|p|)e^{\pm i(-\Delta)^2}|p|^{9}\langle x\rangle^{-\delta}\|_{2\rightarrow2}dt\lesssim \sum_{j=0}^{\infty} (2^{-j}+2^{-3j(\delta-3)})\lesssim_{n,\delta}1.
\end{align}

\noindent\textbf{\textit{Proof of \eqref{NTE}}}:
The same as the statement for \eqref{HEE}, it suffices to estimate \eqref{NTE} when $t\geq1$.
By estimate \eqref{PSE} with $l=0$ and  $M=\frac{1}{\langle t\rangle^{1/4-\epsilon}}$, we have
\begin{equation}\label{NTE1}
    \begin{aligned}
         \|\frac{|p|^{\alpha}}{\langle p\rangle^{\alpha}}P^{\pm} F_{>1/\langle t\rangle^{1/4-\epsilon}}(|p|)e^{\pm i(-\Delta)^2} \langle x\rangle^{-\delta}\|_{2\rightarrow2}&\lesssim \frac{1}{M^{3\delta}\langle t\rangle^{\delta}}\lesssim_{ n,\alpha,\epsilon,\delta} \frac{1}{\langle t\rangle^{(\frac{1}{4}+3\epsilon)\delta}} \\
         & \lesssim_{ n,\alpha,\epsilon,\delta}\frac{1}{\langle t\rangle^{\frac{\delta}{4}+\epsilon}}.
    \end{aligned}
\end{equation}
On the other hand, by Lemma \ref{BoundP-Lem}, we have
\begin{equation}\label{NTE2}
    \begin{aligned}
         &\|\frac{|p|^{\alpha}}{\langle p\rangle^{\alpha}}P^{\pm} F_{\leq 1/\langle t\rangle^{1/4-\epsilon}}(|p|)e^{\pm i(-\Delta)^2} \langle x\rangle^{-\delta}\|_{2\rightarrow2}\\
         \lesssim & \||p|^{\alpha}F_{\leq 1/\langle t\rangle^{1/4-\epsilon}}(|p|)e^{\pm i(-\Delta)^2} \langle x\rangle^{-\delta}\|_{2\rightarrow2}\\
         \lesssim& \||p|^{\alpha+\min\{\frac{n}{2},\delta\}}F_{\leq 1/\langle t\rangle^{1/4-\epsilon}}(|p|)e^{\pm i(-\Delta)^2} |p|^{-\min\{\frac{n}{2},\delta\}}\langle x\rangle^{-\delta}\|_{2\rightarrow2}\\
         \lesssim&\frac{1}{\langle t\rangle^{(1/4-\epsilon)(\alpha+\min\{\frac{n}{2},\delta\})}}\|F_{\leq 1/\langle t\rangle^{1/4-\epsilon}}(|p|) e^{\pm i(-\Delta)^2}|p|^{-\min\{\frac{n}{2},\delta\}}\langle x\rangle^{-\delta}\|_{2\rightarrow2} \\
         \lesssim& \frac{1}{\langle t\rangle^{(1/4-\epsilon)(\alpha+\min\{\frac{n}{2},\delta\})}}.
    \end{aligned}
\end{equation}
Combining \eqref{NTE1} with \eqref{NTE2}, we get \eqref{NTE}.

\noindent\textbf{\textit{{Proof of  \eqref{PLSE}}}}:
For $\epsilon\in(0,1/4),$ we decompose this as \eqref{NTE}.
Due to \eqref{PSE}, we have
\begin{equation}\label{PLSE1}
    \begin{aligned}
         &\|\frac{1}{\langle x\rangle^{4-\delta}}P^{\pm} F_{>1/\langle t\rangle^{1/4-\epsilon}}(|p|)e^{\pm i(-\Delta)^2} p_j^{l}\langle x\rangle^{-(\frac{n}{2}+4+l-\delta)}\|_{2\rightarrow2} \\
         \lesssim&\|P^{\pm} F_{>1/\langle t\rangle^{1/4-\epsilon}}(|p|)e^{\pm i(-\Delta)^2} |p|^{l}\langle x\rangle^{-(\frac{n}{2}+4+l-\delta)}\|_{2\rightarrow2} \\
         \lesssim&_{n,l,\delta}\frac{1}{\langle t\rangle^{ \frac{1}{4}(\frac{n}{2}+4+l-\delta)+\frac{l}{4}+\epsilon(3(\frac{n}{2}+4+l-\delta))-l)}}\\
         \lesssim&_{ n, l,\delta }  \frac{1}{\langle t\rangle^{ \frac{n}{8}+\frac{4+l-\delta}{4}+\epsilon }}.
    \end{aligned}
\end{equation}
On the other hand, by the Hardy-Littlewood-Sobolev inequality in Lemma \ref{HLSvar}, we get
 \begin{align}\label{PLSE2}
         &\|\frac{1}{\langle x\rangle^{4-\delta}}P^{\pm} F_{\leq 1/\langle t\rangle^{1/4-\epsilon}}(|p|)e^{\pm i(-\Delta)^2} p_j^{l}\langle x\rangle^{-(\frac{n}{2}+4+l-\delta))}\|_{2\rightarrow2}\\
         \lesssim&_{ n, \epsilon }
         \|\frac{1}{\langle x\rangle^{4-\delta}} F_{\leq 1/\langle t\rangle^{1/4-\epsilon}}(|p|)e^{\pm i(-\Delta)^2} p_j^{l}\langle x\rangle^{-(\frac{n}{2}+4+l-\delta))}\|_{2\rightarrow2}\nonumber\\
         \lesssim& \|\frac{1}{\langle x\rangle^{4-\delta}}|p|^{-4+\delta}\|_{2\to 2}\| F_{\leq 1/\langle t\rangle^{1/4-\epsilon}}(|p|)e^{\pm i(-\Delta)^2} p_j^{l}|p|^{\frac{n}{2}+4-\delta-\epsilon}\|_{2\to 2}\||p|^{-(\frac{n}{2}-\epsilon)}\langle x\rangle^{-(\frac{n}{2}+4+l-\delta))}\|_{2\rightarrow2}\nonumber\\
         \lesssim& \frac{1}{\langle t\rangle^{(\frac{n}{2}+4+l-\delta-\epsilon)(1/4-\epsilon)}}\lesssim\frac{1}{\langle t\rangle^{\frac{n}{8}+\frac{4+l-\delta}{4}-\epsilon}}.\nonumber
    \end{align}
Combining \eqref{PLSE1} with \eqref{PLSE2}, this yields \eqref{PLSE}.
\end{proof}

\section[Appendix~\thesection: Introduction of Floquet system]{Introduction of Floquet system}\label{appendixB}

For $\vec{s}\in \mathbb{T}_1\times\cdots\times\mathbb{T}_N$, where the definition is defined as before, let $U_{\vec{s}}(t,0)$ denote the solution operator to system
\begin{align}\label{A1}
    i\partial_t U_{\vec{s}}(t,0)=((-\Delta)^2+V_0(x)+\sum_{j=1}^N V_j(x,t+s_j))U_{\vec{s}}(t,0).
\end{align}

We have the following lemma.
\begin{lemma}
Let $U_{\vec{s}}(t,0)$ be the solution operator to \eqref{A1}, $K$ and $K_0$  are the perturbed Floquet operator and the free Floquet operator defined by
\begin{align}
    K\coloneqq K_0+V_{F}(x,\vec{s}), \ \ K_0\coloneqq H_0+\sum_{j=1}^N p_{s_j}, \ \ p_{s_j}=-i\partial_{s_j}
\end{align}
where
\begin{align}
    V_F(x,\vec{s})\coloneqq V_0(x)+\sum_{j=1}^N V_j(x,s_j).
\end{align}
Then,  we have
\begin{align}
    e^{itH_0}U_{\vec{s}}(t,0)=e^{itK_0}e^{-itK}, \text{on } L^2_x(\mathbb{R}^n).
\end{align}
\end{lemma}
\begin{remark}
    K is the Hamiltonian of the Floquet system
    \begin{align}
        i\partial_t \psi(x,\vec{s},t)=K \psi(x,\vec{s},t),\ (x,\vec{s},t)\in \mathbb{R}^n\times\mathbb{T}_1\times\cdots\times\mathbb{T}_N\times \mathbb{R}.
    \end{align}
\end{remark}

\begin{proof}
    We only should to prove the fact that \begin{align}
        e^{-it\sum\limits_{j=1}^N p_{s_j}}U_{\vec{s}}(t,0)=e^{-itK}.
    \end{align}
    Acting $i\partial_t$ on both sides, we get \begin{equation}
        \begin{aligned}
        (H_0+V_0(x)+\sum_{j=1}^N p_{s_j})e^{-it\sum\limits_{j=1}^N p_{s_j}}U_{\vec{s}}(t,0)+e^{-it\sum\limits_{j=1}^N p_{s_j}}(\sum_{j=1}^N V_j(x,t+s_j))U_{\vec{s}}(t,0)=Ke^{-itK}.
    \end{aligned}
    \end{equation}
    Now we only need to prove that \begin{align}\label{commu}
        e^{-it\sum\limits_{j=1}^N p_{s_j}}(\sum_{j=1}^N V_j(x,t+s_j))=(\sum_{j=1}^N V_j(x,s_j))e^{-it\sum\limits_{j=1}^N p_{s_j}}.
    \end{align}
    There is a well-known conclusion that \begin{align}
        e^{-itp_s}f(t+s)e^{itp_s}=f(s),
    \end{align}which is following by the Fourier transform and Fourier inversion formula, verified through elementary commutator calculus. (\ref{commu}) is a directly consequence. Thus, we have prove this lemma.
\end{proof}

One can see it from the expression of such a projection
\begin{align}
    P_c(t)=\Omega_{+}(t)\Omega^*_{+}(t),
\end{align}
with the time-dependence comes from the definition of wave operator and its adjoint
\begin{align}
    \Omega_{+}(t)=s\text{-}\lim_{u\rightarrow\infty}U(t,t+u)e^{-iuH_0} \ \mbox{on}\  L_x^2(\mathbb{R}^n),
\end{align}
\begin{align}
    \Omega^*_{\alpha,+}(t)=s\text{-}\lim_{u\rightarrow\infty}F_{\leq1}\left(\frac{|x|}{u^{\alpha}}\right)e^{iuH_0}U(t+u,t) \  \mbox{on}\  L_x^2(\mathbb{R}^n),
\end{align}
for $\alpha\in (0,\frac{1}{2}-\frac{2}{n}), n\geq9.$
\begin{lemma}
    For $t \equiv s_j \pmod {T_j}$, $j=1,\cdots,N,$
    \begin{align}
        \Omega^*_{K,+}(t)\coloneqq s\text{-}\lim_{u\rightarrow\infty}e^{iuH_0}U_{\vec{s}}(u,0)P_c(t), \text{ on } L^2_x(\mathbb{R}^n),
    \end{align}
    where
    \begin{align}
    i\partial_t U_{\vec{s}}(t,0)=\left((-\Delta)^2+V_0(x)+\sum_{j=1}^N V_j(x,t+s_j)\right)U_{\vec{s}}(t,0).
\end{align}
\end{lemma}
Recall that $\mathcal{H}_F=L^{\infty}_{\vec{s}}L^2_x(\mathbb{R}^n\times \mathbb{T}_1\times\cdots\times\mathbb{T}_{N})$, the Floquet wave operators are defined by
\begin{align}
    \Omega_{K,+}^*\coloneqq s\text{-}\lim_{t\rightarrow\infty}e^{itK_0}e^{-itK}P_c(K) \text{ on } \mathcal{H}_F
\end{align}
and
\begin{align}
    \Omega_{K,+}\coloneqq s\text{-}\lim_{t\rightarrow\infty}e^{itK}e^{-itK_0} \text{ on } \mathcal{H}_F
\end{align}
with $P_c(K)$, the projection on the continuous spectrum of $K$ in $\mathcal{H}_F$ space. The existence of $\Omega_{K,+}$ follows from Cook's method.
\begin{lemma}
    If $V_{F}(x,\vec{s})\in L^{\infty}_{\vec{s}}L^2_x(\mathbb{R}^n\times \mathbb{T}_1\times\cdots\times\mathbb{T}_{N})$, for $\alpha\in(0,\frac{1}{2}-\frac{n}{2})$,
    \begin{align}
        \Omega_{K,\alpha,+}^*=s\text{-}\lim_{t\rightarrow\infty}F_{\leq 1}\left(\frac{|x|}{t^{\alpha}}\right)e^{itK_0}e^{-itK} \text{ on } \mathcal{H}_F
    \end{align}
    exists.
\end{lemma}
This lemma can imply the existence of $\Omega_{K,+}^*$ via the construction of the projection on the space of corresponding scattering states.

\section*{Acknowledgments}
We would like to thank Prof. Avy Soffer and Dr. Xiaoxu Wu for profitable suggestions and encouragements on this topic. This work was partially supported by the Zhejiang Provincial Natural Science Foundation of China (Grant No. LZ25A010003).

\vspace{2mm}
	
\noindent \textbf{Conflict of interest.} The authors do not have any possible conflicts of interest.
	
\vspace{2mm}
	
\noindent \textbf{Data availability statement.} Data sharing is not applicable to this article as no data sets were generated or analyzed during the current study.


\end{document}